\documentclass[11pt]{article}
\usepackage{amsthm} 
\usepackage{fullpage, graphicx, hyperref, color,amsmath, txfonts, amsfonts, bbm, url, wasysym, amssymb, tkz-euclide,caption, subcaption,soul} 
\usepackage{graphicx}
\usepackage{float} 
\usepackage{pgfplots}
\pgfplotsset{compat=1.18}

\usepackage[final]{pdfpages}
\usepackage[utf8]{inputenc}
\usepackage[mathscr]{euscript} 
\usepackage[english]{babel} 

\usepackage[nottoc]{tocbibind}

\makeatletter
\g@addto@macro{\UrlBreaks}{\UrlOrds}
\makeatother 

\newcommand{\distas}[1]{\mathbin{\overset{#1}{\kern\z@\sim}}}%
\newsavebox{\mybox}\newsavebox{\mysim}
\newcommand{\distras}[1]{%
	\savebox{\mybox}{\hbox{\kern3pt$\scriptstyle#1$\kern3pt}}%
	\savebox{\mysim}{\hbox{$\sim$}}%
	\mathbin{\overset{#1}{\kern\z@\resizebox{\wd\mybox}{\ht\mysim}{$\sim$}}}%
}

\newtheorem{definition}{Definition}[section] 
\newtheorem{Condition}{Condition}[section] 
\newtheorem{proposition}{Proposition}[section] 
\newtheorem{theorem}{Theorem}[section] 
\newtheorem{corollary}{Corollary}[section] 
\newtheorem{lemma}[theorem]{Lemma} 

\renewcommand{\Re}{\mathop{\mathrm{Re}}}
\renewcommand{\Im}{\mathop{\mathrm{Im}}}
\newcommand{\Tr}{\mathop{\mathrm{Tr}}}
\newcommand{\Var}{\mathop{\mathrm{Var}}}
\newcommand{\mysqrt}[1]{\left(#1\right)^{\frac{1}{2}}}

\definecolor{lightgray}{gray}{0.5}
\title{Mesoscopic Edge Universality of Orthogonal Polynomial Ensembles}
\author{Wenkui Liu\footnote{This work was done while the author was at Department of Mathematics, KTH Royal Institute of Technology. The author is currently at UCLouvain, contacted by wenkui.liu@uclouvain.be \\Supported by the Swedish Research Council (VR), grant no. 2021-06015, and the European Research Council (ERC), Grant Agreement No. 101002013.}}

\date{}

\begin{document}
	\maketitle
	
\begin{abstract}
		
		In this paper, we study the mesoscopic fluctuations at edges of orthogonal polynomial ensembles with both continuous and discrete measures. Our main result is a Central Limit Theorem (CLT) for linear statistics at mesoscopic scales. We show that if the recurrence coefficients for the associated orthogonal polynomials are slowly varying, a universal CLT holds. Our primary tool is the resolvent for the truncated Jacobi matrices associated with the orthogonal polynomials. While the Combes-Thomas estimate has been successful in obtaining bulk mesoscopic fluctuations in the literature, it is too rough at the edges. Instead, we prove an estimate for the resolvent of Jacobi matrices with slowly varying entries. Particular examples to which our CLT applies are Jacobi, Laguerre and Gaussian unitary ensembles, as well as discrete ensembles from random tilings. 
		
	\end{abstract}
	\tableofcontents{}
	
\section{Introduction}

Let $\mu$ be a Borel measure on $\mathbb{R}$ with finite moments, i.e., for all $k\in \mathbb{N}$, $\int_{\mathbb{R}}|x|^kd\mu(x)<\infty$ . The orthogonal polynomial ensemble (OPE) of order $n\in\mathbb{N}$ associated to $\mu$ is a probability measure on $\mathbb{R}^n$ proportional to 
\begin{equation}\label{Poly_ensemble}
\prod_{1\leq i<j\leq n}(x_i-x_j)^2d\mu(x_1)\cdots d\mu(x_n) .
\end{equation}
In this paper, we will be interested in studying the behaviour of this ensemble for large $n$.

OPEs arise naturally in many models of statistical mechanics, probability theory, combinatorics, and random matrix theory. For some surveys, we refer to \cite{borodin2009determinantal, hough2006determinantal,johansson2006random,konig2005orthogonal}. An important and well-known source of examples where OPEs appear is the eigenvalues of random Hermitian matrices, where the probability measure is invariant under the conjugation with unitary matrices. classical examples are the Jacobi, Laguerre and Gaussian Unitary Ensembles. In these cases, the measure $\mu$ is absolutely continuous with an analytic density. The support of the measure can be compact, for example, the (modified) Jacobi Unitary Ensemble, or unbounded, for example the Gaussian and Laguerre Unitary Ensembles. When the support of the measure is unbounded, it is natural to rescale the measure by $n$, so that the eigenvalues accumulate on finitely many intervals with high probability as $n\to\infty$. For this reason, we will allow the measure $\mu$ to depend on $n$ and write $\mu=\mu_n$.

OPEs with discrete measures are also natural objects to be considered. For instance, uniformly distributed random lozenge tilings of a hexagon give rise to (extended) OPEs associated with the Hahn measure, see \cite{johansson2006random}. Similarly, uniformly distributed random domino tilings of an Aztec diamond give rise to (extended) OPEs associated with the Krawtchouk measure, see \cite{johansson2002non}. In such cases, the measure $\mu$ is discrete and its parameters depend on the size of the tiling, i.e., $\mu=\mu_n$. This is another reason that we allow the measure to be $n$-dependent. 

 We will study the asymptotic behaviour as $n\to\infty$ and explore cases of $\mu$ (or $\mu_n$) both continuous and discrete in this paper. There are three regimes that one can look into, i.e., macroscopic, mesoscopic and microscopic regimes. In the macroscopic regime, one studies the global behaviour of the point processes, whereas in the microscopic regime, one zooms in near a point and studies the nearest neighbour interactions between points. This paper studies the mesoscopic regime which is intermediate between the other two and we will focus on the edges. 
 
 Let us, for clarity, start with the example of the Laguerre Unitary Ensemble, whose measure is given by $d\mu_n(x) = x^\gamma e^{-n x}dx$ on $[0, +\infty)$, for some parameter $\gamma>-1$. It describes the eigenvalue distribution of a Wishart matrix. It is a classical example that has both a hard and soft edge and will, therefore, be an interesting  example for the techniques developed in this paper.   For the Laguerre Unitary Ensemble,  the eigenvalues accumulate on the interval $[0,4]$ almost surely as $n\to \infty$ and have the limiting distribution $d\rho_{\mu}(x)=\frac{1}{2\pi}\sqrt{\frac{4-x}{x}}dx$, see Figure~\ref{fig: Laguerre}.  In the bulk, i.e., any point in the interval $(0,4)$, the typical distance between neighbouring points is of order $\sim n^{-1}$. Near the edges, the scaling is different. The origin is a hard edge of the interval, since no eigenvalues can be negative. Near the origin, the limiting density behaves as $d\rho_{\mu}(x)\sim 1/\sqrt x dx$, which is typical at hard edges, and the neighbouring points are of distance $\sim n^{-2}$.  The right-end point of the interval is a typical example of soft edge where the density vanishes as  square root $d\rho_{\mu}(x)\sim \sqrt{4-x} dx$  and the distance between neighbours is $\sim n^{-\frac{2}{3}}$. It is well-known that the microscopic process of this Laguerre Unitary Ensemble converges to the Bessel point process at the hard edge, sine point process in bulk and Airy point process at the soft edge.  These scaling limits on the microscopic scale are universal and observed in a wide variety of models, see \cite{deift1999orthogonal,deift2007universality, levin2008universality,pastur2011eigenvalue}. We also point out that \cite{kuijlaars2011universality} offers an overview of universality.

\begin{figure}
  \center
  \begin{tikzpicture}[scale=0.7]
    \begin{axis}[
      axis lines=middle,
      xlabel={$x$},
      ylabel={$\qquad \qquad d\rho_{\mu}(x)/dx$},
      xmin=0,
      xmax=4.5,
      ymin=0,
      ymax=5,
      xtick={0,4},
      ytick={0},
      domain=0.1:4,
      samples=200,
      smooth,
      restrict y to domain=0:5,
      enlargelimits=false,
      grid=major,
      grid style={dashed,gray!50},
      ]
      \addplot[black,ultra thick] {sqrt((4-x)/x)};
    \end{axis}
    
    \foreach \i in {0,...,7}
    \fill (5 * \i/7,0) circle (2pt);
    \foreach \i in {0,...,7}
    \fill (2 * \i/7,0) circle (2pt);
    \foreach \i in {0,...,5}
    \fill (0.5 * \i/5,0) circle (2pt);
  \end{tikzpicture}
  \caption{The equilibrium measure of Laguerre Unitary Ensemble. Hard edge $x=0$, soft edge $x=4$. }
  \label{fig: Laguerre}
\end{figure}

We will study the mesoscopic scales at edges using the scaled linear statistics of the following form $X_{f,\alpha,x_0}^{(n)} $, where the test function $f$ is a compactly supported real-valued function, the location is $x_0\in\mathbb{R}$, and the scale is $n^{\alpha}$ for some $\alpha>0$,
\begin{equation}\label{eq:meso linear}
  X_{f,\alpha,x_0}^{(n)} \coloneqq \sum_{i=1}^nf(n^\alpha(x_i-x_0)).
\end{equation}
These scaled linear statistics \eqref{eq:meso linear} are to examine the points that are at a distance at most of the order of $n^{-\alpha}$ around $x_0$. Given that the test function $f$ is compactly supported, the points $\{x_j\}_j$, whose distance between $x_0$ are of order far larger than $n^{-\alpha}$,  will eventually fall outside the support of $f$ as $n\to\infty$ and only points within a distance of $x_0$ of order at most $n^{-\alpha}$ fall within the support. Note that if $\alpha=0$, the scaled linear statistics are reduced to the global linear statistics and "sees" all points simultaneously. If $\alpha$ is big enough such that $n^\alpha$ is the microscopic scale around the point $x_0$, then we are at the local (or microscopic) regime, and the scaled linear statistics will only "see" finitely many points in the limit. Hence, the scale $n^\alpha$, between the macroscopic and microscopic scales, is called the mesoscopic scale. Typically, in the bulk $\alpha\in(0,1)$, at the hard edge $\alpha\in(0,2)$, and at the soft edge $\alpha\in (0,2/3)$. 

In this mesoscopic regime in the literature,  the limiting fluctuations are studied by \cite{soshnikov2000central} for the classical compact groups. Later, the mesoscopic fluctuations in the bulk and at the soft edges for deformed Wigner matrix are studied by \cite{li2021central}, which includes GUE. The mesoscopic fluctuations for Dyson’s Brownian motion in the bulk are also studied in \cite{duits2018mesoscopic} via loop equations. The bulk universal mesoscopic fluctuations for OPEs that can be approximated by modified Jacobi ensembles are considered in \cite{breuer2016universality}. The bulk mesoscopic fluctuations for the sparsely perturbed Jacobi Unitary Ensembles  are considered in \cite{ofner2024stability} . For Wigner matrices and $\beta$-ensembles, the bulk mesoscopic limit is shown in \cite{landon2020applications}.  The bulk mesoscopic fluctuations for OPEs that priorly have a sine universality are considered in \cite{lambert2018mesoscopic}. The bulk mesoscopic fluctuations of the linear statistics with smooth test functions ($f\in C_c^6$)for $\beta$ ensembles with smooth potentials ($V\in C^7$ as defined as \eqref{eq:def potential}) are considered in \cite{bekerman2018mesoscopic}, by the loop equation technique. The mesoscopic fluctuations for circular orthogonal polynomial ensembles are considered in \cite{breuer2024mesoscopic}. Recently, mesoscopic eigenvalue statistics of Wigner-type matrices in the bulk, at edges and at cusp-like singularities have been studied by \cite{riabov2025linear,riabov2025mesoscopic}. By far, a substantial portion of research is concentrated about the bulk. Most methods designed for the bulk break down at edges, which makes the analysis more difficult. Hence, non-trivial modifications to these methods are needed to study the edge behaviour. This paper aims to address the gap in understanding the mesoscopic limit at the edges (both soft and hard). It is dedicated to establishing general conditions for such behaviour at the mesoscopic scale for orthogonal polynomial ensembles at the edges.

Here, we present our first two pedagogical results, which are consequences of our general theorems.

\subsubsection*{Non-varying Weights}
 A canonical example of a compactly supported non-varying measure is the modified Jacobi Unitary Ensemble
\begin{equation}
  d\mu(x) = (1-x)^{\gamma_1} (1+x)^{\gamma_2}h(x)dx, \quad x\in[-1, 1], \label{eq:modified Jacobi}
\end{equation}
where $\gamma_1,\gamma_2>-1$ are parameters and $h$ is an analytic function in a neighbourhood of $[-1,1]$ and strictly positive on $[-1,1]$. The limiting distribution for the point process is the arcsine measure $d\rho_{\mu}(x)=\frac{1}{\pi\sqrt{1-x^2}}dx$ on $[-1,1]$.  In the bulk, its mesoscopic limit has been studied by \cite{breuer2016universality}, whose result directly implies that, for $f\in C_c^1(\mathbb{R})$, $x_0\in(-1,1)$  and $\alpha\in(0,1)$, as $n\to\infty$, the following converges in distribution
\begin{equation}
  X_{f,\alpha,x_0}^{(n)}-\mathbb{E}\left[X_{f,\alpha,x_0}^{(n)}\right]\to  \mathcal{N}\left(0,\frac{1}{4\pi^2}\int\int_{\mathbb{R}^2}\left(\frac{f(x)-f(y)}{x-y}\right)^2dxdy\right). \label{eq:bulk}
\end{equation}
Our result is at the (hard) edges $x_0=1$ or $-1$. 
\begin{theorem}\label{thm: modify Jacobi}
  Let $\mu$ be the modified Jacobi weight as \eqref{eq:modified Jacobi}, given $\gamma_1,\gamma_2>-1$. Then we have that, for all $f\in C^1_c(\mathbb{R})$ and $\alpha\in(0,2)$, the following converges in distribution as $n\to\infty$
    \begin{align}
    X_{f,\alpha,x_0}^{(n)}-\mathbb{E}\left[X_{f,\alpha,x_0}^{(n)}\right]\to & \mathcal{N}\left(0,\sigma^2_f\right)	\label{eq:regular analytic clt}\\
    \sigma_f^2 = & \begin{cases}  \frac{1}{8\pi^2}\int\int_{\mathbb{R}^2}\left(\frac{f(-x^2)-f(-y^2)}{x-y}\right)^2dxdy, \qquad x_0=1\\
      \frac{1}{8\pi^2}\int\int_{\mathbb{R}^2}\left(\frac{f(x^2)-f(y^2)}{x-y}\right)^2dxdy, \qquad x_0=-1	. \end{cases} \label{eq:regular analytic clt variance}
  \end{align}
\end{theorem}

Theorem~\ref{thm: modify Jacobi} is a consequence of more general results Theorem~\ref{thm: main 3} or Theorem~\ref{thm: main 2}, which we will discuss in Subsection~\ref{sec:results}. The proof of Theorem~\ref{thm: modify Jacobi} is presented in Section~\ref{sec:proof of 123}.

\subsubsection*{Varying Weights}
Let us consider the case where there exists a real potential $V$ on $\mathbb{R}$ such that 
\begin{equation}
  d\mu_n(x) = e^{-nV(x)}dx, \quad x\in \mathbb{R}, \label{eq:def potential}
\end{equation}
\begin{equation}
  \frac{V(x)}{\log(x^2+1)} \to +\infty \quad \text{as } |x|\to\infty \label{eq:potential infinity}. 
\end{equation}
Condition \eqref{eq:potential infinity} is to ensure all finite moments of $\mu_n$. This is a classical case where we are expecting soft edges, since the measure $\mu_n$ is supported on the real line. The equilibrium measure with respect to the potential $V$ is defined as the unique probability measure $\rho_\mu$ that minimises the potential
\begin{equation} \label{eq:potential min}
  I_V(\nu)\coloneqq-\int\int \log|x-y|d \nu(x) d\nu(y) +\int V(x) d\nu(x).
\end{equation}
The existence and uniqueness of the minimum $I_V(\rho_\mu)$ is achieved via potential theory. For $V$ being analytic and strictly convex, it is also known that $\rho_\mu$ is supported in a single interval. For more about the potential theory, one may refer to \cite{saff2013logarithmic}.

\begin{theorem} \label{thm: regular}
  Assume $d\mu_n(x)=e^{-nV(x)}dx$ to be supported on the real line. Let $V$ be an analytic and strictly convex potential satisfying \eqref{eq:potential infinity}. Without loss of generality, up to some normalisation, assume the equilibrium measure, which is the minimiser of \eqref{eq:potential min}, to be supported in a single interval $[-1,1]$. Then the same asymptotic result as \eqref{eq:regular analytic clt} and \eqref{eq:regular analytic clt variance} in Theorem~\ref{thm: modify Jacobi} follows for all  $f\in C^1_c(\mathbb{R})$ and $\alpha\in(0,\frac{2}{3})$, as $n\to\infty$. 

\end{theorem}
Theorem~\ref{thm: regular} is a consequence of a general result Theorem~\ref{thm: main 1}, which we will discuss in Subsection~\ref{sec:results}. Analyticity and convexity are not the necessary conditions for the general results in this paper. However, they ensure the equilibrium measure to be supported on a single interval, see \cite{deift1999uniform}. The proof of Theorem~\ref{thm: regular} is presented in Section~\ref{sec:proof of 123}.

The conclusion of the results above holds significance in several aspects. First of all, the limiting variance \eqref{eq:regular analytic clt variance} is scaling invariant. Set $g(x)=f(a^2x)$ and we have $\sigma_f^2=\sigma_g^2$. This gives a heuristic explanation of the independence of $\alpha$ in the limit. However, compared with the bulk mesoscopic limit in \eqref{eq:bulk}, $\sigma_f^2$ is not translation invariant in $f$.

Finally, the variances for both the soft and the hard edges share the same formula \eqref{eq:regular analytic clt variance}  as indicated in the two theorems above, though microscopically, the OPEs may have different limiting processes as discussed. Our mesoscopic theorems do not require any knowledge about microscopic information a priori. It is, therefore, reasonable to anticipate that there is a broad range of OPEs  for which the linear statistics should yield limiting mesoscopic fluctuations at the edges.

Note that Basor, Widom and Ehrhardt studied, in separate works, the asymptotics of the determinants of Airy and Bessel operators  \cite{basor2003asymptotics,basor1999determinants}. Their results imply that the linear statistics of the scaled Airy or Bessel point processes have the limiting Gaussian distribution with the variance same as \eqref{eq:regular analytic clt variance}, when zooming out. However, when applying their results to OPEs, for example GUE, one gets double limit asymptotics rather than the direct mesoscopic limits considered in this paper. 

Our purpose of this paper is to extend Theorems~\ref{thm: modify Jacobi} and~\ref{thm: regular} to a wider range of $\mu$ and $\mu_n$. Analyticity and convexity of the potential $V$ are not necessary. We include cases where $\mu$ or $\mu_n$ is a discrete measure, for example, Hahn polynomials (from the uniform measure on lozenge tilings of a hexagon), Krawtchouk polynomials (from the uniform measure on domino tilings on the Aztec diamond)  and Tricomi-Carlitz polynomials in Subsection~\ref{sec:eg TC}. Our starting point is the three-term recurrence relation for the orthogonal polynomials. We give conditions on the coefficients that imply a mesoscopic CLT. Our criteria will be satisfied by the classical (properly scaled) hypergeometric orthogonal polynomials where the recurrence coefficients are explicitly known, cf. \cite{koekoek2010hypergeometric}. Our method is inspired by \cite{ breuer2016universality,breuer2017central}, which use  the recurrence coefficients to characterise the moments of the linear statistics of an OPE. The difficulty of adapting the methods of \cite{breuer2016universality} is that the standard Combes-Thomas estimate to prove mesoscopic limit in bulk is unsuitable for the analysis at the edges (cf. Sections~\ref{sec:CT discussion}). In this paper, we develop estimates at the edges in place of the Combes-Thomas estimate (cf. Propositions~\ref{thm: overview F entry estimation} and~\ref{prop: overview tri-inverse}). Moreover, our estimates work for the varying, as well as non-varying, weights, while Breuer and Duits only focus on the non-varying measure in \cite{breuer2016universality}. A further ingredient is a quantitative determinate estimate for finite matrices that is approximately a Toeplitz (cf. Proposition~\ref{prop:logdetOrderBound}). This is related to Ehrhardt's generalised  Helton–Howe–Pincus formula (see \eqref{eq:ehrhardt}) which is used by \cite{breuer2016universality} for semi-infinite Toeplitz operators.  Last but not least, we directly analyse the mesoscopic linear statistics without assuming microscopic kernel universality.

\section{Statement of Results}
In this section, we will state our main results. The proofs are given in the later sections. The general results are described  in terms of the three-term recurrence relations of orthogonal polynomials, which we will introduce first. 
\subsection{Orthonormal Polynomials and Recurrence Relations}
Given a measure $\mu_n$ on a subset of $\mathbb{R}$, with finite moments, we let $\{p_{j,n}(x)\}_{j\geq 0}$ be the family of orthonormal polynomials with respect to $\mu_n$. The second subscript $n$ emphasises that the measure $\mu_n$ may be varying in $n$ as pointed out by the examples from \eqref{eq:def potential}. Precisely, such $p_{j,n}$ is the unique polynomial of degree $j$ with a positive leading coefficient. For $\mu_n$ is supported on continuous sets, we have
\begin{equation}
  \int p_{j,n}(x)p_{k,n}(x)d\mu_n(x)=\delta_{j,k}. 
\end{equation} 
For $\mu_n$ supported on discrete sets, the integral above is understood as the corresponding sum over the mass points, namely, $\sum_{l=0}^{M_n} p_{j,n}(x_l)p_{k,n}(x_l)\mu_n(x_l)=\delta_{j,k}
  $, where $M_n\in \mathbb{N}\cup\{\infty\}$. In the case $M_n<\infty$, we require $(M_n-n)n^{-\frac{\alpha}{2}}$ to grow at least at the order of $n^{\varepsilon}$ for some $\varepsilon>0$ (cf. Condition~\ref{ass:slowly varying recurrence coe}, which guarantees this requirement.). 
  
It is well-known that orthonormal polynomials satisfy a three-term recurrence relation
\begin{align}
  xp_{0,n}(x)=&a_{1,n}p_{1,n}(x)+b_{0,n}p_{0,n}(x) \label{eq:def recurrence 0},\\
  x p_{j,n}(x)=&a_{j+1,n}p_{j+1,n}(x)+b_{j,n}p_{j,n}(x)+a_{j,n}p_{j-1,n}(x) \label{eq:def recurrence 1},
\end{align}
for $j=1,2,\dots$ and some coefficients $a_{j,n}>0$ and $b_{j,n}\in\mathbb{R}$. These coefficients uniquely determine the orthonormal polynomials and are fundamental to our analysis. However, the recurrence coefficients do not, in general, determine the measure $\mu_n$ uniquely. They do so only if $\mu_n$ is determined by its moments. Throughout this paper, we do not assume that the corresponding moment problem is determinate. 

For the non-varying weight $\mu$, the above discussion follows in the same way. 

The asymptotic results of linear statistics of OPEs can be obtained by studying the recurrence coefficients. Among them, Breuer and Duits study the global fluctuations  \cite{breuer2017central}. One case of \cite{breuer2017central} is that, given the existence 
\begin{equation}\label{eq:recurrence 1 limit}
  \lim_{n\to\infty}a_{n+k,n}\eqqcolon a, \quad \lim_{n\to\infty}b_{n+k,n}\eqqcolon b,
\end{equation} 
for all $k\in \mathbb{Z}$, the following converges in distribution,  as $n\to\infty$,
\begin{equation}\label{eq:global result}
  X_f^{(n)}-\mathbb{E}[X_f^{(n)}]\to\mathcal{N}\left(0,\sum_{k=1}^\infty k \left|\frac{1}{2\pi}\int_0^{2\pi}f(2a\cos(\theta)+b)e^{ik\theta}d\theta\right|^2\right)	.
\end{equation}
For the non-varying weights the assumption \eqref{eq:recurrence 1 limit} can be simplified to $\lim_{n\to\infty}a_{n}\eqqcolon a$ and $\lim_{n\to\infty}b_{n}\eqqcolon b$. Compared with the mesoscopic limit in Theorems~\ref{thm: modify Jacobi} and~\ref{thm: regular}, which are universal results, the variance of the limiting global fluctuations depends on the limits of the recurrence coefficients. 

Consequently, the edges of the global fluctuations in this case are $b-2a$ and $b+2a$. 
One should be aware that though the edges of fluctuations coincide with the boundary of the equilibrium measure as in Theorem~\ref{thm: modify Jacobi} and~\ref{thm: regular}, this is not the case in general. For example, the Tricomi Carlitz Polynomials, given in Subsection~\ref{sec:eg TC}, have the equilibrium measure supported on $\mathbb{R}$, while the fluctuations of the associated OPE take place on $[-2,2]$ only. 

Another important result about the limiting fluctuations of mesoscopic linear statistics in the bulk is obtained by \cite{breuer2016universality} i.e., \eqref{eq:bulk} holds, whenever the recurrence coefficients are such that
\begin{equation}
  a_{j}=a+O(j^{-1}), \quad b_{j}=b+O(j^{-1}), \quad \text{ as } j\to\infty,
\end{equation}
for some $a>0$ and $b\in\mathbb{R}$. The recurrence relations also play a central role in the study of mesoscopic fluctuations of Circular OPEs and sparse perturbations of JUE in \cite{breuer2024mesoscopic,ofner2024stability} respectively.

\subsection{Results on Non-Varying Weights} \label{sec:results non-vary}
Consider an OPE \eqref{Poly_ensemble} with respect to the real Borel measure $\mu$ with finite moments, where the associated recurrence coefficients are $a_{j}$ and $b_{j}$, given by \eqref{eq:def recurrence 0}  and \eqref{eq:def recurrence 1}. Note that the recurrence coefficients do not have a second subscript, since they come from the measure $\mu$, non-varying in $n$. 

Recall that $X_{f,\alpha,x_0}^{(n)}$ is the mesoscopic linear statistics of an OPE for a test function $f$ around the point $x_0$ which is defined by \eqref{eq:meso linear}. Now we are going to state our general theorems for non-varying weights.

\begin{theorem}\label{thm: main 3}
  Consider $0<\alpha<2$. Suppose there exist $a>0$ and $b\in \mathbb{R}$ such that  the recurrence coefficients associated with the OPE satisfy
  \begin{equation}\label{eq:ass recurrence perturbation non-vary}
    a_{j} =a+O(j^{-\alpha-\varepsilon}),\quad b_{j} =b+ O(j^{-\alpha-\varepsilon})
  \end{equation}
  for some $\varepsilon\in(0, 1-\alpha/2)$ small, as $j\to\infty$. Then, for any $f\in C^1_c(\mathbb{R})$, as $n\to\infty$, the following converges in distribution
  \begin{align}
    X_{f,\alpha,x_0}^{(n)}-\mathbb{E}\left[ X_{f,\alpha,x_0}^{(n)} \right]\to\mathcal{N}(0,\sigma^2_f) ,
  \end{align}
  where   
  \begin{align} 
    \sigma_f^2 = &  \frac{1}{8\pi^2}\int\int_{\mathbb{R}^2}\left(\frac{f(-x^2)-f(-y^2)}{x-y}\right)^2dxdy, \quad \text{ for } x_0=b + 2a+o(n^{-\alpha}),\\
    \sigma_f^2 = &  \frac{1}{8\pi^2}\int\int_{\mathbb{R}^2}\left(\frac{f(x^2)-f(y^2)}{x-y}\right)^2dxdy, \quad \text{ for } x_0=b - 2a+o(n^{-\alpha}).
  \end{align}
\end{theorem}

The assumption \eqref{eq:ass recurrence perturbation non-vary}  is natural. Note that the Chebyshev polynomials of the second kind have constant recurrence coefficients, cf. Subsection~\ref{sec:eg JUE}, which comes from the Jacobi Unitary Ensemble from random matrix theories. More generally, the modified Jacobi  polynomials also satisfy \eqref{eq:ass recurrence perturbation non-vary} (cf. \eqref{eq:recurrence modified Jacobi} due to \cite{kuijlaars2004riemann}). Hence, Theorem~\ref{thm: modify Jacobi} is a direct consequence of Theorem~\ref{thm: main 3}. The proof is elaborated in Section~\ref{sec:proof of 123}. Another interesting example is the OPE with a logarithm weight studied by \cite{deift2024recurrence}. The assumption \eqref{eq:ass recurrence perturbation non-vary} is also satisfied in this case, see Subsection~\ref{sec:example log}.   
 Moreover, the assumption \eqref{eq:ass recurrence perturbation non-vary} with exactly the same rate of convergence is considered by \cite{breuer2016universality} to show the bulk mesoscopic universality, i.e., \eqref{eq:bulk},  in their work.
 
 Also note that the exact edges are at $b+2a$ and $b-2a$. Theorem~\ref{thm: main 3} allows $x_0$ to be close to these edges as long as they have  a distance of order at most  $o(n^{-\alpha})$. This is also natural, since the mesoscopic linear statistics \eqref{eq:meso linear} is scaled as $n^{\alpha}$ and any perturbations smaller than the window size should not change the result. 

Another remark is that, Theorem~\ref{thm: main 3} also holds for an OPE associated with a varying weight $\mu_n$ with varying recurrence coefficients $a_{j,n}$ and $b_{j,n}$ such that
\begin{equation}\label{eq:ass recurrence perturbation}
  \sup\limits_{j\geq n-n^{\frac{\alpha}{2}+\varepsilon}}|a_{j,n} -a|=O(n^{-\alpha-\varepsilon}),\quad \sup\limits_{j\geq n-n^{\frac{\alpha}{2}+\varepsilon}}|b_{j,n} -b|= O(n^{-\alpha-\varepsilon}).
\end{equation}
Though assumption \eqref{eq:ass recurrence perturbation} is slightly weaker than the assumption \eqref{eq:ass recurrence perturbation non-vary}, in most reasonable examples of orthogonal polynomial ensembles, the measures of such are of bounded supports. Typically, there is no scaling with $n$ in such cases. Hence, Theorem~\ref{thm: main 3} is stated for non-varying cases. To maintain the consistency of the proofs in this paper, we will use \eqref{eq:ass recurrence perturbation} in the proof of Theorem~\ref{thm: main 3} in Section~\ref{sec:proof of 3}.

\subsection{Results on Varying Weights} \label{sec:results}
Now, we state the results about varying weights. Given an OPE \eqref{Poly_ensemble} with respect to the real Borel measure $\mu_n$ with finite moments, let  the recurrence coefficients associated with measure $\mu_n$ be $a_{j,n}$ and $b_{j,n}$ defined by \eqref{eq:def recurrence 0}  and \eqref{eq:def recurrence 1}. 

Let $0<\alpha<2$ and $\varepsilon\in(0,1-\frac{\alpha}{2})$ small.  Define $I_n^{(\alpha,\varepsilon)}$ to be an indexing set
\begin{equation}\label{eq:index set}
  I_n^{(\alpha,\varepsilon)}\coloneqq\{j\in \mathbb{N}: n-n^{\frac{\alpha}{2}+\varepsilon}\leq j\leq n+n^{\frac{\alpha}{2}+\varepsilon}\},
\end{equation} 
In  Theorems~\ref{thm: main 1} and~\ref{thm: main 2}, we will assume the recurrence coefficients are slowly varying as follows.
\begin{Condition} \label{ass:slowly varying recurrence coe}
  
  There exist absolute constants $c_0,c_1>0$ such that, for all $j\in I_n^{(\alpha,\varepsilon)}$ as \eqref{eq:index set} and $n\in \mathbb{N}$,
  \begin{align}
    c_0<|a_{j,n}|<c_1 & , \quad |b_{j,n}|<c_1, \label{ass:recurrence bound 1}\\
    |a_{j,n}-a_{j-1,n}|\leq \frac{c_1}{n} & , \quad |b_{j,n}-b_{j-1,n}|\leq \frac{c_1}{n}.
  \end{align} 
\end{Condition}
The parameter $\alpha$ is the same as the "scaling" parameter in the mesoscopic linear statistics. Condition~\ref{ass:slowly varying recurrence coe} means that we only assume the recurrence coefficients of order around $n$ with a window size of $2n^{\alpha/2+\varepsilon}$ are bounded and slowly varying. As we will show, to prove the mesoscopic limit of the linear statistics, it is enough to consider recurrence coefficients of orders only inside this window. Many interesting examples in random matrix theories satisfy this condition and it is not hard to verify this. The recurrence coefficients from Theorem~\ref{thm: regular} will satisfy this condition, and so are those from the classical (hypergeometric) orthogonal polynomials along the Askey scheme when scaled (cf. \cite{koekoek2010hypergeometric}). 

We will zoom in around a point $x_0$, that may depend on $n$, such that 
\begin{equation}\label{eq:def edges}
  x_0=b_{n-1,n} - 2\sqrt{a_{n,n}a_{n-1,n}}+o(n^{-\alpha}) \quad \text{on the left, or } 
  x_0=b_{n-1,n} + 2\sqrt{a_{n,n}a_{n-1,n}}+o(n^{-\alpha}), \quad \text{on the right}.
\end{equation}
That is, for  varying weights, instead of $b\pm 2a$ in Theorem~\ref{thm: main 3}, we centre $x_0$ around the points $b_{n-1,n} \pm2\sqrt{a_{n,n}a_{n-1,n}}$. We note that if $b_{n-1,n} - 2\sqrt{a_{n,n}a_{n-1,n}}-x_0>c>0$, one can show with the same proof in this paper, the variance of $ X_{f,\alpha,x_0}^{(n)}$ is $o(1)$. The same is true if $x_0$ is right to the right edge. Also note that whenever \eqref{eq:recurrence 1 limit} holds, this is the case where $x_0$ is outside the global fluctuation range according to \cite{breuer2017central}. One can also see an example in Subsection~\ref{sec:eg TC} for further discussion.

Now we are going to state our general theorems for varying weights. Recall that $X_{f,\alpha,x_0}^{(n)}$, defined in \eqref{eq:meso linear}, is the mesoscopic linear statistics of an OPE for a test function $f$ around the point $x_0$.
\begin{theorem}\label{thm: main 1}
Consider $0<\alpha<\frac{2}{3}$. Assume there is an $\varepsilon\in(0,1-\frac{\alpha}{2})$ such that Condition~\ref{ass:slowly varying recurrence coe} is satisfied for all $j\in I_n^{(\alpha,\varepsilon)}$ and $n\in \mathbb{N}$. Then, for any $f\in C^1_c(\mathbb{R})$, the following converges in distribution as $n\to\infty$
\begin{align}
  X_{f,\alpha,x_0}^{(n)}-\mathbb{E}\left[ X_{f,\alpha,x_0}^{(n)} \right]\to\mathcal{N}(0,\sigma^2_f) ,
\end{align}
where   
\begin{align} 
  \sigma_f^2 = &  \frac{1}{8\pi^2}\int\int_{\mathbb{R}^2}\left(\frac{f(-x^2)-f(-y^2)}{x-y}\right)^2dxdy, \quad \text{ for } x_0=b_{n-1,n} + 2\sqrt{a_{n,n}a_{n-1,n}}+o(n^{-\alpha}), \label{eq:variance right}\\
  \sigma_f^2 = &  \frac{1}{8\pi^2}\int\int_{\mathbb{R}^2}\left(\frac{f(x^2)-f(y^2)}{x-y}\right)^2dxdy, \quad \text{ for } x_0=b_{n-1,n} - 2\sqrt{a_{n,n}a_{n-1,n}}+o(n^{-\alpha}).\label{eq:variance left}
\end{align}
\end{theorem}

Theorem~\ref{thm: main 1} only deals with the case $0<\alpha<\frac{2}{3}$. For soft edges with square root decay, this is optimal. But for hard edges, the result should also hold for $0<\alpha<2$. We will show that this is true under additional assumptions. Indeed, Theorems~\ref{thm: modify Jacobi}, ~\ref{thm: regular} and~\ref{thm: main 1} are special cases of the following. 

\begin{theorem}\label{thm: main 2}
Consider $0<\alpha<2$. Assume there is an $\varepsilon\in(0,1-\frac{\alpha}{2})$ such that the Condition~\ref{ass:slowly varying recurrence coe} is satisfied for all $j\in  I_n^{(\alpha,\varepsilon)}$ and $n\in \mathbb{N}$. Also assume the following holds, as $n\to\infty$
\begin{equation}
  \max_{j\in I_n^{(\alpha,\varepsilon)}}\left| 	a_{j,n}a_{j-2,n}-a_{j-1,n}^2  \right|= o(n^{-\alpha-\varepsilon}),  \label{eq:ass:recurrence 1} 
\end{equation}
\begin{equation}
  \max_{j\in I_n^{(\alpha,\varepsilon)}}\left| (|b_{j-1,n}-x_0|-a_{j,n})a_{j-2,n}-(|b_{j-2,n}-x_0|-a_{j-1,n})a_{j-1,n} \right|=o(n^{-\frac{3\alpha}{2}-\varepsilon}) \label{eq:ass:recurrence 2}.
\end{equation}

Then, for any $f\in C^1_c(\mathbb{R})$, as $n\to\infty$, the following converges in distribution
\begin{equation}
  X_{f,\alpha,x_0}^{(n)}-\mathbb{E}\left[ X_{f,\alpha,x_0}^{(n)} \right]\to\mathcal{N}(0,\sigma^2_f) ,
\end{equation}
where $\sigma_f^2$ is given by \eqref{eq:variance right} and \eqref{eq:variance left}.

\end{theorem}	

Note that the Condition~\ref{ass:slowly varying recurrence coe} implies that both left-hand sides of \eqref{eq:ass:recurrence 1} and \eqref{eq:ass:recurrence 2} are $O(n^{-1})$. In other words, Theorem~\ref{thm: main 1} is a direct consequence of Theorem~\ref{thm: main 2} considering $\alpha\in(0, \frac{2}{3})$.

These conditions \eqref{eq:ass:recurrence 1} and \eqref{eq:ass:recurrence 2} are technical in the proof. To understand where they come from, more background is needed. We will postpone the explanations to the end of Subsection~\ref{sec:pre inversion}.

There are two examples that motivate conditions \eqref{eq:ass:recurrence 1} and \eqref{eq:ass:recurrence 2}. First, it is shown in \cite{kuijlaars2004riemann} that the recurrence coefficients of the modified Jacobi polynomials with measure \eqref{eq:modified Jacobi} are slowly varying in a stronger sense, i.e., 
\begin{align*}
|a_{j}-a_{j-1}| = O(j^{-3}), \quad 	|b_{j}-b_{j-1}| = O(j^{-3}), \qquad \text{as } j\to\infty. 
\end{align*}
Note that there is no scaling in this model, and we remove the second subscript $n$ from the notations. Hence, the conditions \eqref{eq:ass:recurrence 1} and \eqref{eq:ass:recurrence 2} are satisfied for all $0<\alpha<2$. The limit of mesoscopic fluctuations holds for all $\alpha\in(0,2)$, .i.e., Theorem~\ref{thm: modify Jacobi}. For details, see proof of Theorem~\ref{thm: modify Jacobi} in Section~\ref{sec:proof of 123}.

Another motivation for the conditions is the Laguerre Unitary Ensemble, which exhibits a hard edge on the left (at $0$). In this case, $a_{j,n}=\sqrt{j(j+\gamma)}/n$, $b_{j,n}=(2j+\gamma+1)/n$ and $x_0=0$. When $\gamma=0$, it is easy to verify that assumption \eqref{eq:ass:recurrence 1} holds for all $0<\alpha<2$ and the left-hand side of \eqref{eq:ass:recurrence 2} vanishes for all $j$. Hence, Theorem~\ref{thm: main 2} applies and we obtain the limit of mesoscopic fluctuations for all $\alpha\in(0,2)$ at the hard (left) edge. For general $\gamma>0$, Theorem~\ref{thm: main 2} also applies, see Subsection~\ref{sec:eg Laguerre} for details. The Laguerre case demonstrates that a slowly varying condition in a stronger sense is not necessary. It suffices that the non-linear second-order combinations appearing in \eqref{eq:ass:recurrence 1} and \eqref{eq:ass:recurrence 2} vary sufficiently slowly. 

Last but not the least, the assumption $f\in C_c^1(\mathbb{R})$ in
Theorems~\ref{thm: main 3}, \ref{thm: main 1} and~\ref{thm: main 2} is unlikely to be optimal. Indeed, the limiting variance $\sigma_f^2$ suggests that $f\in H^{1/2+\delta}(\mathbb{R})$ for some $\delta>0$ and compactly supported. This extension is obtained for the bulk Wigner case \cite{landon2022almost}. To extend our results to $ H^{1/2+\delta}(\mathbb{R})$, a uniform variance
estimate in $H^{1/2+\delta}$ is needed to replace Proposition 5.1 of
\cite{breuer2016universality} (cf.
Proposition~\ref{proposition: variance est}). We do not pursue it here.

\subsection{Overview of the Rest of the Paper}

The rest of the paper is organised as follows.

In the Preliminaries, Section~\ref{sec:preliminary}, we introduce the cumulant expansion of OPEs. We will express it in terms of recurrence coefficients. Special care is needed in the case where the moment problem is indeterminate. We will also present the main tool of analysis, an estimate of a three-diagonal matrix with slowly varying entries and explain why the usual Combes-Thomas estimate is insufficient for our setup. 

The varying weights are more difficult to prove than the non-varying ones. Hence, most of the content of this paper is to develop techniques for the varying weights, the proof of which we will show first. 

In Section~\ref{sec:overview of proof}, we provide the proof of Theorem~\ref{thm: main 2}. We illustrate the essence of proof by demonstrating several key propositions. The proofs of these key propositions are deferred in later Sections~\ref{sec:improve CB},~\ref{sec:cumulant truncation} and~\ref{sec:strong szego}.

In Section~\ref{sec:proof of 123}, we explore the conditions of Theorem~\ref{thm: main 2} and subsequently prove Theorems~\ref{thm: modify Jacobi}, ~\ref{thm: main 1} and~\ref{thm: regular}, in that specific order.

Theorem~\ref{thm: main 3} requires a different strategy of proof than the other theorems, since it is essentially a result of non-varying weights. It is explained and proved in Section~\ref{sec:proof of 3}.

In Section~\ref{sec:examples}, we give several examples that Theorems~\ref{thm: main 3},~\ref{thm: main 1} and~\ref{thm: main 2} are applicable. There includes examples of continuous and discrete weights.

\section{Preliminaries}\label{sec:preliminary}
In this section, we will introduce some notations and recall some preliminary facts regarding cumulants for linear statistics. We will follow the approach proposed by Breuer and Duits, \cite{breuer2016universality,breuer2017central}. They are the first to discover that the cumulants of the linear statistics can be expressed in terms of the Jacobi (semi-infinite) matrix corresponding to the measure $\mu_n$ that defines the OPE. This discovery successfully leads to several results about linear statistics of OPEs, regarding both global and mesoscopic scales. This approach can also be applied to extended-OPEs that have extra parameters indicating the time transition, that rises naturally in many random tiling models, see \cite{duits2018global, duits2024lozenge}.  

We will also introduce the Combes-Thomas estimate, which is an essential element proving the mesoscopic limit in the bulk, \cite{breuer2016universality}. However, as we will explain in this section, at the edges, it is too rough and its improvement is required. To this end, we will introduce a formula for symmetric tri-diagonal matrices, that is well-known, but a key element for what follows. 

\subsection{Some Notations}
We start this section with some notations that we will use. For a general reference see \cite{simon2005trace}. For a compact operator $A$ on a (separable) Hilbert space, we denote the singular values by $\sigma_j(A)$, which are the square roots of the eigenvalues of the compact self-adjoint operator $A^*A$. Then we define 
\begin{enumerate}
\item $\|A\|_\infty\coloneqq \sup_j\sigma_j(A)$ to be the operator norm,
\item $\|A\|_1\coloneqq \sum_j\sigma_j(A)$ to be the trace norm,
\item $\|A\|_2\coloneqq \sqrt{\sum_j\sigma_j(A)^2}$ to be the Hilbert-Schmidt norm.
\end{enumerate}
Then we have the following inequalities
\begin{enumerate}
\item For $j=1,2,\infty$, 
\begin{equation*}
  \|AB\|_j\leq \|A\|_j\|B\|_\infty, \quad \|AB\|_j\leq \|A\|_\infty\|B\|_j ,
\end{equation*}
\begin{equation*}
  \|AB\|_1\leq\|A\|_2\|B\|_2 .
\end{equation*}
\item If $A$ is trace class, 
\begin{equation*}
  |\Tr A|\leq \|A\|_1.
\end{equation*}
\end{enumerate}
We will frequently view a semi-infinite matrix $A=\left( (A)_{i,j} \right)_{i,j\in \mathbb{N}}$ with entries $(A)_{i,j}\in \mathbb{C}$, as an operator on $l^2(\mathbb{N})$. Then
\begin{equation*}
\|A\|_\infty\leq \sum_{j=-\infty}^\infty\sup_k\left| (A)_{k,k+j}  \right|,
\end{equation*}
\begin{equation*}
\|A\|_1\leq \sum_{i,j=1}^\infty\left| (A)_{i,j}  \right|,
\end{equation*}
\begin{equation*}
\|A\|_2=\left( \sum_{i,j=1}^\infty|(A)_{i,j}|^2  \right)^{1/2}.
\end{equation*}

If $A$ is further Hermitian, $\eta\in \mathbb{C}$ with $\Im \eta\neq 0$ and $\mathrm{Id}$ is the identity operator, the following holds
\begin{equation}\label{eq:hermitian inequality}
\|(A-\eta \mathrm{Id})^{-1}\|_\infty\leq \frac{1}{|\Im \eta|}.
\end{equation} 
Note that among many trace norm inequalities, the one above may not be the optimal nor elegant one. However, in our cases, it is sufficient. 
\subsection{Bounded  and Unbounded Jacobi Operators}
Let the Jacobi matrix $\mathcal{J}$ be an semi-infinite tri-diagonal matrix associated with the measure $\mu_n$ with entries being the recurrence coefficients defined as \eqref{eq:def recurrence 0} and \eqref{eq:def recurrence 1}, i.e., 

\begin{equation}\label{Jacobi}
\mathcal{J} \coloneqq  \begin{pmatrix}b_{0,n} & a_{1,n}      \\a_{1,n} & b_{1,n}       & a_{2,n}        \\ & a_{2,n} & b_{2,n} & a_{3,n} &   \\&  & a_{3,n} & b_{3,n} & a_{4,n} & \\& & &\ddots &\ddots & \ddots \end{pmatrix}. 
\end{equation}
This $\mathcal{J}$ is also called the Jacobi operator associated with the measure $\mu_n$, viewed as a linear operator on $l^2(\mathbb{N})$ with initial domain to be the space of finitely supported sequences. For finitely supported $\mu_n$ (like for Krawchuk polynomials), then $J$ is, in fact, a finite matrix, which can be embedded into a semi-infinite matrix adjoining zero rows and columns. It is a bounded linear operator only if both sequences $\{a_{j,n}\}_j$ and $\{b_{j,n}\}_j$ are bounded.  For example, $\mathcal{J}$ is bounded for the Chebyshev polynomials where the recurrence coefficients are constants. However, in many cases, $\mathcal{J}$ is not necessarily a bounded operator. For example, $\mathcal{J}$ is unbounded for the Hermite and Laguerre polynomials where $a_{j,n}$ tends to infinity as $j\to\infty$. These examples are included in Section~\ref{sec:examples}. Breuer and Duits studied the resolvent of bounded $\mathcal{J}$  and they successfully obtained the mesoscopic fluctuations in the bulk for the non-varying cases \cite{breuer2016universality}.   However, though $\mathcal{J}$ is symmetric with all entries real numbers, its resolvent is only well-defined if $\mathcal{J}$  is a self-adjoint operator. However, $\mathcal{J}$  is essentially self-adjoint, if and only if the measure $\mu_n$  is fully characterised by its moments, see Theorems~$6.10.$ and~$6.16$ in \cite{schmudgen2017moment}. In terms of recurrence coefficients, that is if $\sum_{j\geq 1}a_{j,n}^{-1}=+\infty$, the moment problem is determinate. If $\{b_{j,n}\}_j$ is bounded, $a_{j-1,n}a_{j+1,n}\leq a_{j,n}^2$ for all $j\geq j_0$ for some $j_0\in\mathbb{N}$, and $\sum_{j\geq 1}a_{j,n}^{-1}<+\infty$, then the moment problem is indeterminate. See Corollary $6.19.$ in \cite{schmudgen2017moment} and Theorem~$1.5$ in \cite{berezanskiui1968expansions}. For example, for Freud weight $d\mu_n(x)=e^{-n|x|^\gamma}dx$, $\gamma>0$, shown in \cite{kriecherbauer1999strong,lubinsky1988proof}, the recurrence coefficients have the asymptotics as
\begin{equation}
a_{j,n}=c_{\gamma}\left( \frac{j}{n} \right)^{\frac{1}{\gamma}}(1+O(j^{-1})), \quad \text{as } j\to\infty, \quad\text{and } b_{j,n}=0,
\end{equation}
where $c_{\gamma}$ is an universal constant only depending on $\gamma$. For a precise statement, see Section~\ref{sec:eg Freud}. Hence, its Jacobi operator is essentially self-adjoint if and  only if $\gamma\geq 1$. 

In the rest of the paper, we will avoid calling a general $\mathcal{J}$ an operator and only treat it as a semi-infinite tri-diagonal matrix, which is always well-defined as \eqref{Jacobi}. We will construct a bounded linear operator from $l^2(\mathbb{N})$ to $l^2(\mathbb{N})$ associated with $\mathcal{J}$, that will be suitable for our further analysis.

Let $P(x;n)$ be the column vector of the orthonormal polynomials 
\begin{equation}\label{eq:vertor of poly}
P(x;n)= (p_{0,n}(x), p_{1,n}(x), \cdots)^{T},
\end{equation}
where the superscript $T$ is the transpose of a vector. We define a  semi-infinite matrix for any $l=1,2,\dots$
\begin{equation}\label{eq:selfadjoint G}
G_z\coloneqq \int_{\mathbb{R}} (x-z)^{-1}P(x;n)P(x;n)^{T}d\mu_n(x), \quad z\in\mathbb{C}, \quad \Im z\neq 0.
\end{equation}
Note that each entry of $G_z$ is well-defined and bounded, i.e., 
\begin{multline}\label{eq:selfadjoint inverse bound}
\left|  \int_{\mathbb{R}} (x-z)^{-1}p_{j,n}(x)p_{k,n}(x)d\mu_n(x) \right| \leq \frac{1}{|\Im z|} \int_{\mathbb{R}} \left| p_{j,n}(x)p_{k,n}(x) \right| d\mu_n(x) \\
\leq \frac{1}{|\Im z|}\left(  \int_{\mathbb{R}} \left|p_{j,n}(x)\right|^2d\mu_n(x)\int_{\mathbb{R}} \left|p_{k,n}(x)\right|^2d\mu_n(x) \right)^{\frac{1}{2}}=  \frac{1}{|\Im z|},
\end{multline}
where we use  the fact $	\left| (x-z)^{-1} \right|\leq \frac{1}{|\Im z|}$ for the first inequality, the Cauchy-Schwartz inequality for the second inequality and the normality of the orthonormal polynomials for the last equality. More generally, we have the following.

\begin{lemma}\label{lemma: selfadjoint 1}
For any $\mathbb{C}$-valued function $f$ such that $\|f\|_{\infty}\coloneqq \sup_{x\in\mathbb{R}}|f(x)|<\infty$, the following represents a bounded linear operator from $l^2(\mathbb{N})$ to $l^2(\mathbb{N})$, given the canonical basis in $l^2(\mathbb{N})$
\begin{equation}\label{eq:self-adjointOPT}
  \int_{\mathbb{R}} f(x)P(x;n)P(x;n)^{T}d\mu_n(x).
\end{equation} 
Its operator norm is bounded by $\|f\|_{\infty}$.
\end{lemma}
\begin{proof}
Define the multiplication operator $\mathcal{M}_f:L^2(\mu_n)\to L^2(\mu_n)$ via $\mathcal{M}_fh(x)=f(x)h(x)$, for $h\in Dom(\mathcal{M}_f)\coloneqq \{h\in L^2(\mu_n): \mathcal{M}_fh\in L^2(\mu_n)\}$. The multiplication operator $\mathcal{M}_f$ is self-adjoint and hence its operator norm is bounded by $\left\| \mathcal{M}_f\right\|_{\infty}\leq \left\| f\right\|_\infty$.

Denote \eqref{eq:self-adjointOPT} as $K_f$. Let $e_j=(0,0, \dots, 1, 0, 0,\dots)^T$ be the $j$-th vector of the canonical basis for $l^{2}(\mathbb{N})$, i.e., the sequence with only the $j$-th element to be $1$ and the rests are zeros. Then,  each entry  of \eqref{eq:self-adjointOPT} equals to
\begin{equation}\label{eq:selfadjoint inner product}
  (e_j,K_fe_k)_{l^2} = \int_{\mathbb{R}} f(x)p_{j-1,n}(x)p_{k-1,n}(x)d\mu_n(x) =\int_{\mathbb{R}} \mathcal{M}_fp_{j-1,n}(x)p_{k-1,n}(x)d\mu_n(x) = \left(\mathcal{M}_fp_{j-1,n},p_{k-1,n}\right)_{L^2(\mu_n)}.
\end{equation}
Let the subspace $\mathcal{H}_{\mu_n}\coloneqq \overline{\mathrm{span}\{p_{j,n}: j\geq 0 \}}^{L^2(\mu_n)}$ be the closure of the linear space spanned by the orthonormal polynomials. Clearly, $L^2(\mu_n)=\mathcal{H}_{\mu_n}\oplus\mathcal{H}_{\mu_n}^{\perp}$.  Let $\mathcal{P}:L^2(\mu_n)\mapsto \mathcal{H}_{\mu_n}$ be the orthogonal projection operator onto  $\mathcal{H}_{\mu_n}$. Define an unitary operator $\mathcal{U}:\mathcal{H}_{\mu_n} \mapsto l^2(\mathbb{N})$ such that $\mathcal{U}p_{j,n}=e_{j+1}$. Then  \eqref{eq:selfadjoint inner product} shows that 
\begin{equation}\label{eq:selfadjoint unitary}
  K_f=\mathcal{U}\mathcal{P} \mathcal{M}_f\mathcal{U}^{-1}.
\end{equation}
Therefore, \eqref{eq:self-adjointOPT} is a bounded operator and its operator norm $\|K_f\|_\infty\leq\| \mathcal{M}_f\
|_\infty\leq \|f\|_\infty$. 
\end{proof}

Lemma~\ref{lemma: selfadjoint 1} shows that $G_z$ defined in \eqref{eq:selfadjoint G} is a bounded linear operator. By the three-term recurrence relation \eqref{eq:def recurrence 0} and \eqref{eq:def recurrence 1}, $G_z$ is the algebraic inverse of $\mathcal{J}-z$, i.e., the following holds  entry-wise,
\begin{equation}\label{eq:selfadjoint alg 1}
\left( (\mathcal{J}-z)G_z \right)_{j,k}=\left( G_z(\mathcal{J}-z) \right)_{j,k}=\begin{cases}
  1, \quad & j=k,\\0,\quad & j\neq k.
\end{cases}
\end{equation}
For any bounded operator $A:l^2\mapsto l^2$ whose inverse $A^{-1}$ exists and is also a bounded linear operator, the following holds algebraically,  i.e., the following holds  entry-wise as a semi-infinite matrix,
\begin{equation}\label{eq:selfadjoint resolvent identity}
G_z-A^{-1}=G_z(A-\mathcal{J}+z)A^{-1}.
\end{equation}
Note that $\mathcal{J}$ is a tri-diagonal matrix, and hence $\mathcal{J}A^{-1}$ is always well-defined as a semi-infinite matrix. If further $\mathcal{J}-z-A$ represents a bounded linear operator from $l^2(\mathbb{N})$ to $l^2(\mathbb{N})$, the right-hand side of \eqref{eq:selfadjoint resolvent identity} can be viewed as the operator compositions.  In such case, with an abuse of terminology, we call \eqref{eq:selfadjoint resolvent identity} the resolvent identity.

The bounded operator $G_z$  agrees with the resolvent of $\mathcal{J}$, whenever $\mathcal{J}$ is self-adjoint. In such case, $G_z=(\mathcal{J}-z)^{-1}$.

\subsection{Cumulants via Recurrence Relations}\label{sec:pre cumulant}
For a real-valued random variable $X$, whose moments are all finite, the cumulants $\mathcal{C}_m(X)$ are uniquely defined by the cumulant generating function (if exists), for $t\in\mathbb{C}$ in a neighbourhood of $0$, 
$$\log\mathbb{E}[e^{tX}]=\sum_{m\geq 1}\frac{t^m}{m!}\mathcal{C}_m(X).$$ 
From this expression, the moments of $X$ can be recovered from its cumulants and vice versa. For example, the mean and variance of $X$ are given by $\mathcal{C}_1(X)$ and $\mathcal{C}_2(X)$ respectively. Such moment-cumulant relationship can also be used to define cumulants, if the cumulant generating function does not exist. Note also that, if $X$ follows a  Gaussian distribution, we have $\log\mathbb{E}[e^{tX}]=t \mathcal{C}_1(X) + t^2 \mathcal{C}_2(X)/2$. Later we will show that the cumulant expansion is absolutely convergent in our setup  (cf.Lemma~\ref{lemma: dominated}).

Hence, to show that the linear statistics $X_{f,\alpha,x_o}^{(n)}$of an OPE as defined as \eqref{eq:meso linear} converges to a Gaussian distribution for a suitable class of test functions $f$, it is sufficient for us to study the asymptotic behaviour of its cumulants. First, we consider the test function $f$ to be
\begin{equation}\label{eq:f imaginary}
f(x)=\Im\sum_{r=1}^M  d_r \frac{1}{x-\lambda_r}
\end{equation}
where $M\in\mathbb{N}$,  $d_r\in\mathbb{R}$, $\Im$ takes the imaginary part of a complex number, and $\Im \lambda_r>0$. We point out that any compactly supported and continuously differentiable function can be well approximated by such $f$ in a certain Lipschitz space. This claim will be explained in Lemma ~\ref{lemma: recolvent approx}, which is a result in \cite{breuer2016universality}. 

Denote $\overline{\lambda}_r$ to be the complex conjugate of $\lambda_j$.  Then $f$ can be rewritten as 
\begin{multline}  \label{eq:resolvent fun}
f(x)=\sum_{r=1}^{2M}c_r \frac{1}{x-\eta_r}, \qquad\\ \text{ where} \quad 
c_r \coloneqq \begin{cases}
  \frac{d_r}{2i} \qquad r=1,\dots,M \\ 	-\frac{d_{r-M}}{2i} \quad r=1+M,\dots,2M\\
\end{cases} \eta_r \coloneqq \begin{cases}
  \lambda_r \quad r=1,\dots,M \\ \overline{\lambda}_{r-M} \qquad r=1+M,\dots,2M\\
\end{cases}.
\end{multline}

and, for some $x_0\in \mathbb{R}$, define 
\begin{equation}\label{eq:mesoscopic ffun}
f_{\alpha,x_0}^{(n)}(x)\coloneqq f(n^\alpha(x-x_0)). 
\end{equation}
\begin{equation}\label{eq:mesoscopic f}
f_{\alpha,x_0}^{(n)}(\mathcal{J})\coloneqq \int_{\mathbb{R}} 	f_{\alpha,x_0}^{(n)}(x)P(x;n)P(x;n)^{T}d\mu_n(x)=\sum_{r=1}^{2M}\frac{c_r}{n^{\alpha}}G_{x_0+\frac{\eta_r}{n^{\alpha}}},
\end{equation}
\begin{equation}\label{eq:mesoscopic ftol}
    \left(f_{\alpha,x_0}^{(n)}(\mathcal{J})\right)^{(l)}\coloneqq \int_{\mathbb{R}} 	\left(f_{\alpha,x_0}^{(n)}(x)\right)^lP(x;n)P(x;n)^{T}d\mu_n(x), \quad l=1,2,\dots
\end{equation}
where $G_{x_0+\frac{\eta_r}{n^{\alpha}}}$ is defined as \eqref{eq:selfadjoint G}. Similar to the properties of $G_z$, in general, the following holds for \eqref{eq:mesoscopic ftol}.

\begin{corollary}\label{lemma: selfadjoint 3} 
$\left(f_{\alpha,x_0}^{(n)}(\mathcal{J})\right)^{(l)}$ is a well defined semi-infinite matrix and its entries are real and symmetric, i.e., 
\begin{equation}
  \left( \left(f_{\alpha,x_0}^{(n)}(\mathcal{J})\right)^{(l)}\right)_{j,k}=\left( \left(f_{\alpha,x_0}^{(n)}(\mathcal{J})\right)^{(l)}\right)_{k,j}=\overline{\left(\left(f_{\alpha,x_0}^{(n)}(\mathcal{J})\right)^{(l)}\right)_{k,j}}, \quad j,k\in\mathbb{N},
\end{equation}
where the overline is the complex conjugate. Moreover, 	$f_{\alpha,x_0}^{(n)}(\mathcal{J})$ represents a bounded linear operator from $l^2(\mathbb{N})$ to $l^2(\mathbb{N})$. Its operator norm is estimated by 
\begin{equation}\label{eq:selfadjoint fJ norm}
  \left\| \left(f_{\alpha,x_0}^{(n)}(\mathcal{J})\right)^{(l)}\right\|_{\infty}\leq \left(\sum_{r=1}^{2M}\left| \frac{c_r}{\Im \eta_r} \right|\right)^l.
\end{equation}

\end{corollary}
\begin{proof}
The statement follows directly from Lemma~\ref{lemma: selfadjoint 1}.
\end{proof}

Breuer and Duits obtained a beautiful formula of the cumulant generating function of the linear statistics in terms of $\mathcal{J}$, see the preliminaries of \cite{breuer2016universality}. They showed the formula for $\mathcal{J}$ being a bounded semi-infinite matrix. By Corollary~\ref{lemma: selfadjoint 3}, we now extend their result to general $\mathcal{J}$, which can be unbounded. 
\begin{lemma}[Breuer-Duits]\label{lemma: cumulant generating}
The cumulant generating function for the linear statistics $X_{f,\alpha,x_0}^{(n)}$ of an OPE has the following determinantal structure,
\begin{equation}
  \log\mathbb{E}\left[\exp\left(tX_{f,\alpha,x_0}^{(n)}\right)\right] = \log \det\left(\int e^{t	f_{\alpha,x_0}^{(n)}(x)}p_{j-1,n}(x)p_{i-1,n}(x)d\mu_n(x)\right)_{j,k=1}^n,
\end{equation}
where  $f_{\alpha,x_0}^{(n)}(\mathcal{J})$ is defined as \eqref{eq:mesoscopic f}, $\mathrm{Id}$ is the identity operator and $P_n$ is the canonical projection onto the first $n$ coordinates. 
\end{lemma}
\begin{proof}
Note that the Vandermonde determinant admits the following formula
\begin{equation}
  \prod_{1\leq i<j\leq n}(x_i-x_j) =\frac{1}{\prod_{j=1}^n\gamma_{j-1,n}}\det (p_{j-1,n}(x_i))_{i,j=1}^n
\end{equation}
where $p_{j,n}$ is the orthonormal polynomial and $\gamma_{j,n}$ is its leading coefficient. 

Hence, use \eqref{Poly_ensemble} to obtain that
\begin{multline}
  \mathbb{E}\left[\exp\left(tX_{f,\alpha,x_0}^{(n)}\right)\right] =\frac{1}{Z_n}\int\cdots \int \prod_{i= n}^ne^{t	f_{\alpha,x_0}^{(n)}(x_i)}\left( \det (p_{j-1,n}(x_i))_{i,j=1}^n \right)^2 d\mu_n(x_1)\cdots d\mu_n(x_n)\\
  =\frac{1}{Z_n}\int\cdots \int\det (e^{t	f_{\alpha,x_0}^{(n)}(x_i)}p_{j-1,n}(x_i))_{i,j=1}^n \det (p_{j-1,n}(x_i))_{i,j=1}^n d\mu_n(x_1)\cdots d\mu_n(x_n), 
\end{multline}
where $Z_n>0$ is some normalising constant.
Recall the Andr\'esief's identity, for a measure $\mu$ and measurable functions $f_j,g_j\in L^2(\mu)$ for $j=1,\dots n$, 
\begin{equation}\label{prop: Andreief}
  \int\cdots \int\det\left(f_j(x_k)\right)_{j,k=1}^n\det\left(g_j(x_k)\right)_{j,k=1}^n d\mu(x_1)\cdots d\mu(x_n) 
  = n! \det\left(\int f_j(x)g_k(x)d\mu(x)\right)_{j,k=1}^n.
\end{equation}
Then,
\begin{equation}
  \mathbb{E}\left[\exp\left(tX_{f,\alpha,x_0}^{(n)}\right)\right] 
  =\frac{n!}{Z_n} \det\left(\int e^{t	f_{\alpha,x_0}^{(n)}(x)}p_{j-1,n}(x)p_{i-1,n}(x)d\mu_n(x)\right)_{j,k=1}^n.
\end{equation}

In particular, take $t=0$ to obtain $\mathbb{E}\left[1\right] 
=\frac{n!}{Z_n}$. Hence $Z_n=n!$. This concludes the Lemma.  
\end{proof}

Breuer and Duits also computed the cumulants in Sections $2.2$ and $2.3$ of \cite{breuer2016universality}, which we extend and formulate as the following lemma. We also give the proof for completeness.
\begin{lemma}[Breuer-Duits]
The cumulants for the linear statistics $X_{f,\alpha,x_0}^{(n)}$ of an OPE can be expressed to be
\begin{equation}
  \mathcal{C}_1(X_{f,\alpha,x_0}^{(n)}) = 	\mathbb{E}[X_{f,\alpha,x_0}^{(n)}] = \Tr\left(P_nf_{\alpha,x_0}^{(n)}(\mathcal{J})P_n\right),
\end{equation}
where $P_n$ is the canonical projection onto the first $n$ coordinates. For $m\geq2$ 
\begin{multline}\label{eq:cumulant 2}
  \mathcal{C}_m(X_{f,\alpha,x_0}^{(n)})= m!\sum_{j=2}^m\frac{(-1)^{j+1}}{j}
  \sum_{l_1+\dots+l_j=m, l_i\geq 1}\\
  \frac{\Tr\left( P_n\left( f_{\alpha,x_0}^{(n)}(\mathcal{J}) \right)^{(l_1)}P_n \left( f_{\alpha,x_0}^{(n)}(\mathcal{J}) \right)^{(l_2)}P_n \dots \left( f_{\alpha,x_0}^{(n)}(\mathcal{J}) \right)^{(l_j)}P_n   \right)-\Tr\left( \left(  f_{\alpha,x_0}^{(n)}(\mathcal{J}) \right)^{(m)}P_n   \right)}{l_1!\cdots l_j!},
\end{multline}
where $ \left(  f_{\alpha,x_0}^{(n)}(\mathcal{J}) \right)^{(l)}$ is defined in \eqref{eq:mesoscopic ftol}.
\end{lemma}
\begin{proof}
Note that for any non-singular square matrix $A$, we have $\log\det A=\Tr\log A$. Then for $t$ sufficiently small, by Lemma~\ref{lemma: cumulant generating}, we have
\begin{equation}
  \log\mathbb{E}\left[\exp\left(tX_{f,\alpha,x_0}^{(n)}\right)\right] =\Tr \log \left(1+\int (e^{t	f_{\alpha,x_0}^{(n)}(x)}-1)p_{j-1,n}(x)p_{i-1,n}(x)d\mu_n(x)\right)_{j,k=1}^n.
\end{equation}
we expand the $\log$ and $\exp$ and reorder the summations to obtain, 
\begin{multline}\label{eq:Fredholm det}
  \log\mathbb{E}\left[\exp\left(tX_{f,\alpha,x_0}^{(n)}\right)\right]
  = \Tr \sum_{j=1}^\infty\frac{(-1)^{j+1}}{j} \sum_{l_1\geq 1, l_2\geq 1, \dots, l_j\geq 1}\frac{t^{l_1+l_2+\dots+l_j}}{l_1!l_2!\dots l_j!}P_n\left( f_{\alpha,x_0}^{(n)}(\mathcal{J}) \right)^{(l_1)}P_n \dots \left( f_{\alpha,x_0}^{(n)}(\mathcal{J}) \right)^{(l_j)}P_n \\
  = \sum_{j=1}^\infty\sum_{m=j}^\infty t^m\frac{(-1)^{j+1}}{j} \sum_{l_1+\dots+l_j=m,l_j\geq 1}\frac{\Tr\left(P_n\left( f_{\alpha,x_0}^{(n)}(\mathcal{J}) \right)^{(l_1)}P_n \dots \left( f_{\alpha,x_0}^{(n)}(\mathcal{J}) \right)^{(l_j)}P_n\right)}{l_1!l_2!\dots l_j!}\\
  = \sum_{m=1}^\infty t^m\sum_{j=1}^m \frac{(-1)^{j+1}}{j} \sum_{l_1+\dots+l_j=m,l_j\geq 1}\frac{\Tr\left( P_n\left( f_{\alpha,x_0}^{(n)}(\mathcal{J}) \right)^{(l_1)}P_n \dots \left( f_{\alpha,x_0}^{(n)}(\mathcal{J}) \right)^{(l_j)}P_n  \right)}{l_1!l_2!\dots l_j!}.
\end{multline}

Then by expanding the right-hand side of \eqref{eq:Fredholm det}, the cumulants of the linear statistics can be written as, 
\begin{equation}\label{eq:cumulant 20}
  \mathcal{C}_m(X_{f,\alpha,x_0}^{(n)})=m!\sum_{j=1}^m\frac{(-1)^{j+1}}{j}\sum_{l_1+\dots+l_j=m, l_i\geq 1}\frac{\Tr\left( P_n\left( f_{\alpha,x_0}^{(n)}(\mathcal{J}) \right)^{(l_1)}P_n \dots \left( f_{\alpha,x_0}^{(n)}(\mathcal{J}) \right)^{(l_j)}P_n   \right)}{l_1!\cdots l_j!}.
\end{equation}
Similarly to \eqref{eq:Fredholm det}, by expanding the logarithm and exponential and reordering the summations, we obtain
\begin{multline*}
  \log (1 + (e^x -1)) =\sum_{j=1}^\infty \frac{(-1)^j}{j}\sum_{l_1\geq 1, l_2\geq 1,\dots, l_j\geq 1 }\frac{x^{l_1+\dots+l_j}}{l_1!\cdots l_j!} = \sum_{j=1}^\infty\sum_{m=j}^\infty\frac{(-1)^j}{j}\sum_{l_1+\cdots +l_j=m, l_i\geq 1}\frac{x^m}{l_1!\cdots l_j!} \\= \sum_{m=1}^\infty x^m\sum_{j=1}^m\frac{(-1)^j}{j}\sum_{l_1+\cdots +l_j=m, l_i\geq 1}\frac{1}{l_1!\cdots l_j!},
\end{multline*}
for all $x\in\mathbb{R}$. Note that $\log (1 + (e^x -1))=x$. Hence, $\sum_{j=1}^{m}\frac{(-1)^j}{j}\sum_{l_1+\cdots +l_j=m, l_i\geq 1}\frac{1}{l_1!\cdots l_j!}=0$ for all $m\geq 2$. 
Then, for $m\geq 2$, \eqref{eq:cumulant 20} can be rewritten to be \eqref{eq:cumulant 2}
\end{proof}

The extra term $\Tr\left( \left(  f_{\alpha,x_0}^{(n)}(\mathcal{J}) \right)^{(m)}P_n   \right)$ enables us to have good estimates of the cumulants. We will use this extensively to show that the cumulants of the linear statistics of an OPE only depend on the recurrence coefficients of order around $n$, given assumptions of Theorem~\ref{thm: main 2} (cf. Proposition~\ref{prop: overview trimming 2}). Another consequence is the following lemma, which is a by-product of Lemma~$2.2$ in \cite{breuer2014nevai}. 
\begin{lemma}[Breuer-Duits]\label{lemma: dominated}
We have for $m\geq 2$
\begin{equation}\label{eq:lemma dominated}
\left| 	\mathcal{C}_m(X_{f,\alpha,x_0}^{(n)})\right| \leq  \sqrt{\frac{2}{\pi }} m!m^{3/2}\left\|f_{\alpha,x_0}^{(n)}\right\|_\infty^{m-2}e^m\mathcal{C}_2(X_{f,\alpha,x_0}^{(n)}),
\end{equation}
where $\left\|f_{\alpha,x_0}^{(n)}\right\|_\infty\leq \sum_{r=1}^{2M}\left|\frac{c_r}{\eta_r}\right|$. In particular, for $|t|<\left(e\left\|f_{\alpha,x_0}^{(n)}\right\|_\infty\right)^{-1}$, there exists some constant $c>0$ such that
\begin{equation}
  \sum_{m\geq 2}\left| \frac{t^m}{m!}\mathcal{C}_m(X_{f,\alpha,x_0}^{(n)}) \right|\leq c t^2\mathcal{C}_2(X_{f,\alpha,x_0}^{(n)}). 
\end{equation}
\end{lemma}
\begin{proof}
  Recall  \eqref{eq:selfadjoint unitary}  in the proof of Lemma~\ref{lemma: selfadjoint 1}, we have 
  \begin{equation}
  \left( f_{\alpha,x_0}^{(n)}(\mathcal{J}) \right)^{(l)}=\mathcal{U}\mathcal{P} \left(\mathcal{M}_{f_{\alpha,x_0}^{(n)}}\right)^l\mathcal{U}^{-1}.
\end{equation}
Define $K_n:L^2(\mu_n)\mapsto \mathcal{H}_{\mu_n}$ such that $K_n\coloneqq \mathcal{U}^{-1}P_n\mathcal{U}\mathcal{P}$, which is a projection operator onto the space spanned by $\{p_{j,n}:j=0,\dots,n-1\}$. Then we rewrite the summand of \eqref{eq:cumulant 2} to be 
\begin{equation}
  \frac{\Tr\left( \left( \mathcal{M}_{f_{\alpha,x_0}^{(n)}}\right)^{l_1}K_n \left( \mathcal{M}_{f_{\alpha,x_0}^{(n)}} \right)^{l_2}K_n \dots \left( \mathcal{M}_{f_{\alpha,x_0}^{(n)}} \right)^{l_j}K_n  \right)-\Tr\left( \left( \mathcal{M}_{f_{\alpha,x_0}^{(n)}}\right)^{m}K_n   \right)}{l_1!\cdots l_j!}.
\end{equation}
In particular, we have $\mathcal{C}_2(X_{f,\alpha,x_0}^{(n)}) = \frac{1}{2}\left[\mathcal{M}_{f_{\alpha,x_0}^{(n)}},K_n\right]\left[K_n,\mathcal{M}_{f_{\alpha,x_0}^{(n)}}\right]$, where $[,]$ is the commutator. From this, we use $(2.6)$ in the proof of Lemma~$2.2$ in \cite{breuer2014nevai} to conclude \eqref{eq:lemma dominated} and the rest follows directly by the same proof as Lemma~$2.2$ in \cite{breuer2014nevai}. 
\end{proof}

Rather than studying the series directly, Lemma~\ref{lemma: dominated} allows us to study the asymptotics of each cumulant as $n\to\infty$ and to apply the dominated convergence theorem to the series.

In the case that the moment problem is determined we have the resolvent to be well-defined and $\left(\mathcal{J}-z\right)^{-1} = G_z$. Hence, the formula can be further simplified. We have by the three-term recurrence relation that $\left(G_z\right)^{(l)} = (\mathcal{J}-z)^{-l}$. Furthermore, $f_{\alpha,x_0}^{(n)}(\mathcal{J})=\sum_{r=1}^{2M}\frac{c_r}{n^{\alpha}}\left(\mathcal{J}-x_0-\frac{\eta_r}{n^{\alpha}}\right)^{-1},$ and
\begin{equation}\label{eq:selfadjointRemoveBracket}
    \left(f_{\alpha,x_0}^{(n)}(\mathcal{J})\right)^{(l)} =  \left(f_{\alpha,x_0}^{(n)}(\mathcal{J})\right)^{l}.
\end{equation}

For any linear operator $\mathcal{A}$, we define for $m\geq 2$,
\begin{equation}\label{eq:cumulant operator}
\mathcal{C}_m^{(n)}(\mathcal{A}) \coloneqq m! \sum_{j=2}^{m}\frac{(-1)^{j+1}}{j}\sum_{l_1+\dots +l_j=m, l_i\geq 1}\frac{\Tr(\mathcal{A})^{l_1}P_n\dots (\mathcal{A})^{l_j}P_n-\Tr(\mathcal{A}^mP_n)}{l_1!\dots l_j!}. 
\end{equation}
Then it is clear that if we define $F\coloneqq n^{\alpha}f_{\alpha,x_0}^{(n)}(\mathcal{J})$, Then, by definition \eqref{eq:mesoscopic f}, $F=\sum_{r=1}^{2M}c_rG_{x_0+ \frac{\eta_r}{n^{\alpha}}} $, which is a bounded linear operator due to Corollary~\ref{lemma: selfadjoint 3}. Then, for $m\geq 2$,
\begin{equation*}
\mathcal{C}_m(X_{f,\alpha,x_0}^{(n)})  =	\mathcal{C}^{(n)}_m(n^{-\alpha}F)= n^{-m\alpha}	\mathcal{C}^{(n)}_m(F),
\end{equation*}
where  $\mathcal{C}^{(n)}_m$ on the right-hand side is given by \eqref{eq:cumulant operator}. The study of the cumulants of the mesoscopic linear statistics for large $n$, then, boils down to the asymptotic behaviour of $G_{x_0+ \frac{\eta_r}{n^{\alpha}}}$, defined in \eqref{eq:selfadjoint G}. 

In the case that $\mathcal{J}$ is not essentially self-adjoint, the relation \eqref{eq:selfadjointRemoveBracket} may not be true, though both hand sides are well defined. To obtain a similar formula, one needs to control the boundary term carefully.  Generally, we have the following. 
\begin{lemma}\label{lemma:boundary}
Let $N\in\mathbb{N}$, $N\geq n$, Define $\mathcal{J}_N \coloneqq P_N\mathcal{J}P_N + Q_N$. Then we have
    \begin{equation}\label{eq:boundaryRecurrence}
         P_n\left(f_{\alpha,x_0}^{(n)}(\mathcal{J})\right)^{(l)} =   P_n\left(f_{\alpha,x_0}^{(n)}(\mathcal{J}_N)\right)^{l}-\sum_{j=0}^{l-1}\sum_{r=1}^{2M} P_nf^j(\mathcal{J}_N)R_rB_{r,n}^{(l-1-j)},
    \end{equation}
    where we let $E_N$ be a semi-infinite matrix with only one entry non-zero $(E_N)_{i,j} \coloneqq \mathrm{1}_{i=N, j=N+1} $ and 
    \begin{equation}\label{eq:boundary1}
        R_r \coloneqq \frac{a_{N,n}}{n^{\alpha}}\left(\mathcal{J}_N-x_0-\frac{\eta_r}{n^{\alpha}}\right)^{-1}E_N,
    \end{equation}
    and $B_{r,n}^{(l-1-j)}$ is a bounded linear operator explicitly by \eqref{eq:def:B}, whose operator norm is $\mathcal{O}(1)$ uniformly in $n$ .
\end{lemma}
\begin{proof}
    Rewrite the three-term recurrence relations \eqref{eq:def recurrence 0} and \eqref{eq:def recurrence 1} into a matrix form
    \begin{equation}
        xP_NP(x;n) =  P_N\mathcal{J}_N P(x;n) + a_{N,n} E_N P(x;n), 
    \end{equation}
    where the term $a_{N,n} E_N P(x;n)$ comes from the recurrence relation of the last entry, i.e.,  $xp_{N-1,n}(x)$. Note that $\mathcal{J}_N$ is a bounded self-adjoint operator and its resolvent is well defined. Hence we have for any $z\in\mathbb{C}$ with $\Im z\neq 0$,
    \begin{equation}\label{eq:recurrenceTruncatedProof}
        (x-z)^{-1}P_NP(x;n) =  P_N(\mathcal{J}_N -z)^{-1}P(x;n) - (x-z)^{-1} a_{N,n} (\mathcal{J}_N -z)^{-1} E_N P(x;n). 
    \end{equation}
    Then by \eqref{eq:recurrenceTruncatedProof}, the definition of $f_{\alpha,x_0}^{(n)}$ \eqref{eq:mesoscopic ffun}, $R_r$ as \eqref{eq:boundary1}, and linearity, we have
    \begin{equation}\label{eq:recurrenceTruncated_fProof}
        f_{\alpha,x_0}^{(n)}(x)P_NP(x;n) =  P_Nf_{\alpha,x_0}^{(n)}(\mathcal{J}_N)P(x;n) - \sum_{r=1}^{2M}  R_r P(x;n)\frac{c_r}{n^{\alpha}}\left(x-x_0-\frac{\eta_r}{n^{\alpha}}\right)^{-1}. 
    \end{equation}
    Then by either repeating \eqref{eq:recurrenceTruncated_fProof} $l$ times or using induction argument, we obtain 
    \begin{multline}\label{eq:recurrenceTruncated_flProof}
        \left(f_{\alpha,x_0}^{(n)}(x)\right)^lP_NP(x;n) \\=  P_N\left(f_{\alpha,x_0}^{(n)}(\mathcal{J}_N)\right)^lP(x;n) - \sum_{j=0}^{l-1}\sum_{r=1}^{2M}  \left(f_{\alpha,x_0}^{(n)}(\mathcal{J}_N)\right)^{j} R_rP(x;n)\frac{c_r}{n^{\alpha}}\left(x-x_0-\frac{\eta_r}{n^{\alpha}}\right)^{-1}\left(f_{\alpha,x_0}^{(n)}(x)\right)^{l-1-j}. 
    \end{multline}
    Plug \eqref{eq:recurrenceTruncated_flProof} into the definition \eqref{eq:mesoscopic ftol}. Then for any $N\geq n$, we have \begin{equation}
         P_n\left(f_{\alpha,x_0}^{(n)}(\mathcal{J})\right)^{(l)} = P_n\left(f_{\alpha,x_0}^{(n)}(\mathcal{J}_N)\right)^l - \sum_{j=0}^{l-1}\sum_{r=1}^{2M}  P_n\left(f_{\alpha,x_0}^{(n)}(\mathcal{J}_N)\right)^{j} R_rB_{r,n}^{(l-1-j)}, 
    \end{equation}
    where 
    \begin{equation}\label{eq:def:B}
        B_{r,n}^{(l-1-j)} \coloneqq \int  \frac{c_r}{n^{\alpha}}\left(x-x_0-\frac{\eta_r}{n^{\alpha}}\right)^{-1}\left(f_{\alpha,x_0}^{(n)}(x)\right)^{l-1-j}P(x;n)P(x;n)^Td\mu_n(x),
    \end{equation}
    which, by Lemma~\ref{lemma: selfadjoint 1}, is well-defined and bounded linear operator such that $$\left\|B_{r,n}^{(l-1-j)}\right\|_\infty \leq \sup_{x} \left|\frac{c_r}{n^{\alpha}}\left(x-x_0-\frac{\eta_r}{n^{\alpha}}\right)^{-1}\left(f_{\alpha,x_0}^{(n)}(x)\right)^{l-1-j}\right|\leq \left|\frac{c_r}{\eta_r}\right|\left(\sum_{r'}\left|\frac{c_r'}{\eta_r'}\right|\right)^{l-1-j}.$$ . 
\end{proof}
The significance of Lemma~
\ref{lemma:boundary} is that though the left-hand side of \eqref{eq:boundaryRecurrence} still depends on the orthogonal measure $\mu_n$, it can be approximated by finitely many recurrence coefficients at the cost of the boundary terms. As we will show later that with a carefully chosen $N$, the term $P_nf^j(\mathcal{J}_N)R_r$  in Lemma~\ref{lemma:boundary} is to be exponentially small in the trace norm. Hence,  $\mathcal{C}_m(X_{f,\alpha,x_0}^{(n)}) $ can be approximated well enough by $\mathcal{C}_m^{(n)}(\mathcal{A}) $ as \eqref{eq:cumulant operator} with a carefully chosen $\mathcal{A}$ (cf. Proposition~\ref{prop: overview trimming}).

\subsection{Combes–Thomas Estimate and its Limitation} \label{sec:CT discussion}
In our analysis, we will need to estimate entries of the resolvent that are far away from the main diagonal. One way of doing that (which is successful in \cite{breuer2016universality}) is by the Combes-Thomas estimate \cite{combes1973asymptotic}. However, for our purpose, this Combes-Thomas estimate is not sufficient, and we will use conditions in Theorem~\ref{thm: main 2} to improve the estimate. Before doing so, we will recall the Combes-Thomas estimate and show its limitations for our purposes. 

\textbf{Combes–Thomas Estimate}
Given $\eta\in \mathbb{C}$ with $\Im \eta \neq 0$. Consider 
\begin{equation} 
  J = \begin{pmatrix}b_{0} & a_{1}  \\a_{1} & b_{1} & a_{2}  \\ & a_{2} & b_{2} & a_{3} \\ & &\ddots &\ddots & \ddots \end{pmatrix}  -\left( x_0+\frac{\eta}{n^\alpha} \right)\mathrm{Id}, 
\end{equation}
where $b_{j},x_0\in\mathbb{R}$, $0<a_{j} \leq c$ for some constant $c>0$ and $\mathrm{Id}$ is the identity operator. Then for any $j,k \in \mathbb{N}$ we have 
\begin{equation}
  \left| \left( J^{-1} \right)_{j,k}  \right| \leq \frac{2n^\alpha}{|\Im \eta|} e^{-\min\{1,\frac{|\Im \eta|}{4ecn^\alpha}\} |j-k|}.
\end{equation}
For a proof of the Combes–Thomas estimate, see \cite{combes1973asymptotic} or Proposition 2.3. in \cite{breuer2016universality}.

Note that the Combes-Thomas estimate works for both varying $a_{j,n}$ and $b_{j,n}$ and non-varying $a_{j}$ and $b_{j}$ . Here we omit the second subscript for the convenience of following discussions.

The Combes-Thomas estimate is a useful result and is successful in studying the bulk universality for OPEs, see \cite{breuer2016universality}. However, in our setup, it has some major limitations. First of all, $a_{j}$ or $a_{j,n}$ can be unbounded. In case of the scaled Hermite polynomials (corresponding to the Gaussian Unitary Ensemble of the random matrix theory) we have $a_{j,n} = \sqrt{\frac{j}{n}}$ and $b_{j,n}=0$. However, in the cumulant expansion \eqref{eq:cumulant 2}, one only needs to estimate $J^{-l}P_n$. The following argument enables us to replace $J$ by a truncated version. Take $l=1$ for example. Take $N=n+\delta$ for some $\delta>0$ and define
\begin{equation*}
J_N\coloneqq P_N J P_N +Q_N.
\end{equation*}
By the resolvent identity and the trace norm inequality, we have 
\begin{equation}
\| (J_N^{-1}-J^{-1})P_n\|_1 \leq \|J^{-1}\|_{\infty}\|(J-J_N)J_N^{-1}P_n\|_1 \label{eq:CT cutoff}.
\end{equation}
Recall that $J_N$ is a block matrix and $N>n$. Hence, $J_N^{-1}P_n=P_NJ_N^{-1}P_n$. Also, recall that $J$ and $J_N$ are tri-diagonal matrices. Hence, the only non-zero entry of $(J-J_N)P_N$ is the $N+1, N$-th entry whose value is $a_N$. Now we estimate the trace norm with entries, i.e., 
\begin{equation}
\|(J-J_N)J_N^{-1}P_n\|_1=\|(J-J_N)P_NJ_N^{-1}P_n\|_1\leq \sum_{k=1}^n\left| a_N\left( J_N^{-1} \right)_{N,k} \right|.\label{eq:CT cutoff ess}
\end{equation}
Then, we only need to estimate $\left( J_N^{-1} \right)_{j,k}  $ for $j=N$ and $k=1,2,\dots, n$. We use the Combes-Thomas estimate for $J_N$ to conclude that \eqref{eq:CT cutoff ess} is exponentially small if $\delta\gg n^\alpha$. Note that $\|J^{-1}\|_\infty\leq \frac{n^\alpha}{|\Im \eta|}$, and hence \eqref{eq:CT cutoff} is exponentially small as well. Recall that the cumulant formula given by \eqref{eq:cumulant operator}. This implies that the difference of the second cumulants $	\mathcal{C}^{(n)}_2(J^{-1})$ and $	\mathcal{C}^{(n)}_2(J_N^{-1})$ is exponentially small. The cumulants of higher order can also be compared in a similar way. The truncation $J_N$ can be a good estimate of $J$ in the cumulant.

However, this reveals a second limitation. This argument works well for $\delta =n ^\beta$ and $0<\alpha<\beta<1$, which is not good enough near the edge. In some examples, for instance, the modified Jacobi case, we want to zoom in further and need an estimate that holds for any $0<\alpha<2$.

Note that if $J$ is a Toeplitz operator, the Combes-Thomas estimate can be improved near the edges. For example, take $b_{j,n}=0$ and $a_{j,n}=1$, which are the recurrence coefficients for Chebyshev polynomials. Take $x_0=2$. One can use Wiener-Hopf factorisation to compute the resolvent explicitly, that is 

\begin{equation}\label{eq:CB should be finer}
\left( J^{-1} \right)_{j,k} = \frac{\omega(\eta)}{1-\omega(\eta)^2}\left( \frac{1}{\omega(\eta)^{|j-k|}} -\frac{1}{\omega(\eta)^{j+k}}\right), 
\end{equation}
where 
\begin{equation*}
\omega(\eta) \coloneqq \frac{2+\frac{\eta}{n^{\alpha}}+\mysqrt{\left( 2+\frac{\eta}{n^{\alpha}} \right)^2-4}}{2}.
\end{equation*}
We take the principal square root so that $|\omega(\eta)|>1$. 
Hence, we estimate, 
\begin{equation}\label{eq:CB should be}
\left| \left( J^{-1} \right)_{j,k}  \right| \leq d_1n^{\frac{\alpha}{2}} e^{-d_2n^{-\frac{\alpha}{2}}|j-k|},
\end{equation}
for some constants $d_1,d_2>0$. Hence, we say that for $|j-k|>n^{\beta}$ with $0<\frac{\alpha}{2}<\beta<1$, $\left| \left( J^{-1} \right)_{j,k}  \right| $ is exponentially small. One should note that $\frac{\alpha}{2}<\beta$ is possible since the spectral density of $J$ vanishes at the edge of the spectrum as a square root. If $x_0\in (-2,2)$ is chosen, we will get that the entries are exponentially small for all $0<\alpha<\beta<1$.

Therefore, these limitations suggest a finer estimation about $J_N^{-1}$, which should resemble \eqref{eq:CB should be}, for some special class of $J$. In a nutshell, the Combes-Thomas estimate works for general Jacobi matrices, but can be improved under further conditions.  

\subsection{Inverse of a Symmetric Tri-diagonal Matrix}\label{sec:pre inversion}
Recall that $J_N$ is a tri-diagonal symmetric matrix. To find the inverse of such a matrix is a classic algebraic question and can be reduced to computing the first and last columns of $J_N^{-1}$. For a review, see \cite{meurant1992review}. Here we only present the expressions needed. 

Let us look at some general properties of an $N\times N$ non singular symmetric tri-diagonal matrix with $a_j\in\mathbb{R}, b_j\in\mathbb{R}$, $N\in\mathbb{N}$, and $z\in \mathbb{C}$,
\begin{equation} \label{eq:pre assumption_J_inverse}
J_{N}= \begin{pmatrix}b_{0} & a_{1} \\a_{1} & b_{1} & a_{2}  \\& a_{2} & b_{2} & a_{3} \\ & &\ddots &\ddots & \ddots \\& && a_{N-2} & b_{N-2} & a_{N-1} \\& & &  & a_{N-1} & b_{N-1}  \end{pmatrix} -z \mathrm{Id}, 
\end{equation}
where $\mathrm{Id}$ is the identity matrix. 

Then for $j< k$
\begin{align}
(J_N^{-1})_{j,k} & = (-1)^{k-j}a_ja_{j+1}\cdots a_{k-1}\frac{d_{k+1}\cdots d_{N}}{\delta_{j}\cdots \delta_{N}} ,\label{tri-inverse raw 1} \\
(J_N^{-1})_{j,j} & = \frac{d_{j+1}\cdots d_{N}}{\delta_{j}\cdots \delta_{N}}, \label{tri-inverse raw 2}
\end{align}
where $d_{N}=b_{N-1}-z$ and, for all $j=1,\cdots N-1$, $d_j$ are solutions to the following difference equation
\begin{equation}
d_j  = b_{j-1}-z-\frac{a_{j}^2}{d_{j+1}}\label{eq:pre inverse_difference_u}.
\end{equation}
We also have $\delta_{1}=b_{0}-z$ and, for all $j=2,\cdots N$, $\delta_j$ are solutions to the following difference equation
\begin{equation}
\delta_j  = b_{j-1}-z-\frac{a_{j-1}^2}{\delta_{j-1}}\label{eq:pre inverse_difference_l}.
\end{equation}

Note that $J_N^{-1}$ is symmetric and we have $(J_N^{-1})_{k,j} =(J_N^{-1})_{j,k} $. Hence, we have an expression for each entry. 

In the following we will linearise both difference equations above and rewrite the inverse formula in different ways that are suitable for our further analysis.

\subsection*{Linearization of the difference equation I}

Recursively define $\beta_j$  by  
\begin{equation*}
\beta_{N} = a_{N-1} ,\quad \beta_{N-1} = b_{N-1}-z ,\quad \frac{\beta_j}{\beta_{j-1}} = \frac{a_{j-1}}{d_j}, \quad \text{for $j=1,\dots,N$}.
\end{equation*}
Hence, $\beta_j$ is the solution to the following linear difference equation, by \eqref{eq:pre inverse_difference_u},
\begin{equation}
\begin{pmatrix}
  \beta_j \\ \beta_{j-1} 
\end{pmatrix} =  A_j \begin{pmatrix}
  \beta_{j+1} \\ \beta_{j} 
\end{pmatrix}, \quad \text{ where } 	A_j \coloneqq \begin{pmatrix}0 & 1 \\ -\frac{a_{j}}{a_{j-1}} & \left( b_{j-1}-z  \right)\frac{1}{a_{j-1}}\end{pmatrix}, \label{eq:pre transfer matrix A}
\end{equation}
for $j=2,\dots, N-1$. Iteratively, 
\begin{equation}\label{eq:pre beta recurrsive}
\begin{pmatrix}
  \beta_j \\ \beta_{j-1} 
\end{pmatrix} =  A_jA_{j+1}\cdots A_{N-1} \begin{pmatrix}
  \beta_{N} \\ \beta_{N-1} 
\end{pmatrix} .
\end{equation}

Let $N_0\in\mathbb{N}$ and $N_0\leq N$. Take the ratio of $\left(J_N^{-1}\right)_{j,k}$ and $\left( J_N^{-1} \right)_{j,N_0}$ which are found by \eqref{tri-inverse raw 1} or \eqref{tri-inverse raw 2}, and we rewrite the inverse formula for $j\leq N_0<k$ to be
\begin{equation}\label{eq:pre  difference_beta 1}
\left(J_N^{-1}\right)_{j,k}  = (-1)^{k-N_0}\frac{\beta_{k}}{\beta_{N_0}}\left( J_N^{-1} \right)_{j,N_0}  .
\end{equation} 

One purpose of the coming Section~\ref{sec:improve CB} is to show that in the case $b_j=b_{j,n}-x_0$, $a_j=a_{j,n}$, $z=\eta/n^\alpha$, $N=n+2mn^{\alpha/2+\varepsilon/3}$ and $N_0=n-2mn^{\alpha/2+\varepsilon/3}$ for $m\in \mathbb{N}$, given assumptions of Theorem~\ref{thm: main 2}, we have $|\beta_k/\beta_{N_0}|\leq e^{-d_2(k-N_0)n^{-\alpha/2}}$, for some constant $d_2>0$. The same is also true for the case  $b_j=b_{j+n-2mn^{\alpha/2+\varepsilon/3},n}-x_0$, $a_j=a_{j+n-2mn^{\alpha/2+\varepsilon/3},n}$, $z=\eta/n^\alpha$, $N=4mn^{\alpha/2+\varepsilon/3}$ and $N_0=2mn^{\alpha/2+\varepsilon/3}$ for $m\in \mathbb{N}$. Hence, we have an estimate as in Proposition~\ref{thm: overview F entry estimation}, which  is an analogue of the exponential part of \eqref{eq:CB should be}.

To this end, the following quantities are essential in this estimate. The eigenvalues of $A_j$ are 
\begin{equation}
\omega_j^+ \coloneqq  \frac{b_{j-1}-z +\mysqrt{\left( b_{j-1}-z \right)^2-4a_ja_{j-1}}}{2a_{j-1}},\quad \omega_j^-\coloneqq \frac{ b_{j-1}-z -\mysqrt{\left(b_{j-1}-z \right)^2-4a_ja_{j-1}}}{2a_{j-1}} . \label{eq:pre eigenvalue_difference_eq}
\end{equation}
The square root is taken with respect to the principal branch, so that $|\omega_j^+|>|\omega_j^-|$, whenever $\Re (b_{j-1}-z)>0$. Note that $A_j$ can be diagonalised as 

\begin{equation*}
A_j = V_j \Omega_j V_j^{-1}, \quad \text{where } V_j \coloneqq \begin{pmatrix}
  1 & 1 \\ \omega_j^+ & \omega_j^- 
\end{pmatrix}, \quad \Omega_j  \coloneqq \begin{pmatrix}
  \omega_j^+ & 0 \\0 &  \omega_j^- 
\end{pmatrix}.
\end{equation*}

Also of importance is 
\begin{equation}
M_j \coloneqq V_j^{-1}V_{j+1}-\mathrm{Id}=  \frac{1}{\omega_j^--\omega_j^+}\begin{pmatrix}\omega_{j}^+-\omega_{j+1}^+ & \omega_{j}^--\omega_{j+1}^-\\ \omega_{j+1}^+-\omega_{j}^+ & \omega_{j+1}^--\omega_{j}^-\end{pmatrix} \label{eq:eigenvalue_difference_eq_M_j},
\end{equation}
where $\mathrm{Id}$ is the identity matrix.

Note that Proposition~\ref{thm: overview F entry estimation} only matches the exponential part of \eqref{eq:CB should be}. To match the $n^{\alpha/2}$ factor of \eqref{eq:CB should be}, we need the following ingredients as well.
\subsection*{Linearization of the difference equation II}
Similarly, we recursively define $\gamma_j$ by  
\begin{equation*}
\gamma_{1} = a_{1} ,\quad \gamma_{2} = b_{0} -z,\quad \frac{\gamma_j}{\gamma_{j+1}} = \frac{a_{j}}{\delta_j} , \quad  \text{for $j=1,\dots,N$} .
\end{equation*}

Then  $\gamma_j$ also admits a recursive formula by equation \eqref{eq:pre inverse_difference_l} 
\begin{equation}
\begin{pmatrix}
  \gamma_j \\ \gamma_{j+1} 
\end{pmatrix} =  B_j\begin{pmatrix}
  \gamma_{j-1}\\ \gamma_{j} 
\end{pmatrix}, \quad \text{ where } B_j \coloneqq \begin{pmatrix}0 & 1 \\ -\frac{a_{j-1}}{a_{j}} & \left( b_{j-1}-z \right)\frac{1}{a_{j}}\end{pmatrix}, \label{eq:pre transfer matrix B}
\end{equation}
for all $j=2,\dots, N-1$. The matrix $B_j$ is called the transfer matrix for $\gamma_j$. Iteratively
\begin{equation*}
\begin{pmatrix}
  \gamma_j \\ \gamma_{j+1} 
\end{pmatrix} =  B_jB_{j-1}\dots B_{2} \begin{pmatrix}
  \gamma_{1} \\ \gamma_{2} 
\end{pmatrix}.
\end{equation*}

Choose any $ a_{N}\neq 0 $ (e.g. One can pick $a_N=a_{N-1}$), rewrite \eqref{tri-inverse raw 1} and \eqref{tri-inverse raw 2} by $\beta_j$ and $\gamma_j$,  and we have another inverse formula for any $j\leq k$
\begin{equation}\label{eq:pre  difference_beta 2}
(J_N^{-1})_{j,k}  = \frac{(-1)^{k-j}\gamma_{j}\beta_{k}}{\beta_{N}a_{N}\gamma_{N+1}}  .
\end{equation}
Note that the inverse formula is independent on the choice of $a_{N}$ since $a_{N}\gamma_{N+1}=\gamma_{N}\delta_N$ and $\gamma_{N}$ and $\delta_N$ are independent of $a_N$. Also note that $J_N$ is symmetric, and hence we have an expression for all entries. 

One purpose of the coming Section~\ref{sec:cumulant truncation} is to show that in the case $b_j=b_{j+n-2mn^{\alpha/2+\varepsilon/3},n}-x_0$, $a_j=a_{j+n-2mn^{\alpha/2+\varepsilon/3},n}$, $z=\eta/n^\alpha$, $N=4mn^{\alpha/2+\varepsilon/3}$ for $m\in \mathbb{N}$, we have precise asymptotics of $\beta_k/\beta_{N_0}$ and $\gamma_j/\gamma_{N+1}$ as $n\to\infty$ whenever the assumptions of Theorem~\ref{thm: main 2} are satisfied. Hence, we will have an asymptotic analogue to \eqref{eq:CB should be finer}, cf. Proposition~\ref{prop: overview tri-inverse}

To this end, the following quantities are essential in this estimate. The eigenvalues of $B_j$ are 
\begin{equation}
\lambda_j^+\coloneqq \frac{b_{j-1}-z+\mysqrt{\left( b_{j-1} -z\right)^2-4a_ja_{j-1}}}{2a_{j}}, \quad \lambda_j^-\coloneqq  \frac{b_{j-1}-z-\mysqrt{\left( b_{j-1} -z\right)^2-4a_ja_{j-1}}}{2a_{j}} .
\end{equation}
The square root is taken with respect to the principal branch, such that $|\lambda_j^+|>|\lambda_j^-|$, whenever $\Re (b_{j-1}-z)>0$. Note that $B_j$ can be diagonalised as
\begin{equation*}
B_j  = W_j \Lambda_j W_j^{-1}, \quad 
\text{where }  W_j \coloneqq \begin{pmatrix}
  1 & 1 \\ \lambda_j^+ & \lambda_j^- 
\end{pmatrix} , \quad  \Lambda_j  \coloneqq \begin{pmatrix}
  \lambda_j^+ & 0 \\0 &  \lambda_j^- 
\end{pmatrix} .
\end{equation*}
Similarly we define $E_j$, which is an essential element for the further analysis, to be 
\begin{equation}
E_j \coloneqq W_j^{-1}W_{j-1}-\mathrm{Id}=  \frac{1}{\lambda_j^--\lambda_j^+}\begin{pmatrix}\lambda_{j}^+-\lambda_{j-1}^+ & \lambda_{j}^--\lambda_{j-1}^-\\ \lambda_{j-1}^+-\lambda_{j}^+ & \lambda_{j-1}^--\lambda_{j}^-\end{pmatrix} \label{eq:eigenvalue_difference_eq_E_j}.
\end{equation}

The identities \eqref{eq:pre  difference_beta 1} and \eqref{eq:pre  difference_beta 2} are the key ingredients in the proofs of Propositions~\ref{thm: overview F entry estimation} and~\ref{prop: overview tri-inverse}, respectively.

\section{Proof of Theorem~\ref{thm: main 2}}\label{sec:overview of proof}
In this section, we decompose the proof of Theorem~\ref{thm: main 2} into a number of propositions, which we will prove in the upcoming sections, cf. Sections~\ref{sec:improve CB}, ~\ref{sec:cumulant truncation},  and~\ref{sec:strong szego}.

Our strategy is to first prove Theorem~\ref{thm: main 2} for some special test functions as described in \eqref{eq:resolvent fun} , i.e.,   $$f(x)=\sum_{r=1}^{2M}c_r \frac{1}{x-\eta_r}$$  for some $\eta_r\in\{x+iy | x\in\mathbb{R}, y \neq 0\}$, in Subsection~\ref{sec:resolvent proof}. Then, extend this result to compactly supported and continuously differentiable test functions in Subsection~\ref{sec:Lipschitz extension}. 

\subsection{Proof of Theorem~\ref{thm: main 2} for Special Class of Test Functions}\label{sec:resolvent proof}
Recall that $\mathcal{J}$ defined in \eqref{Jacobi} is the Jacobi matrix associated with an OPE and $f_{\alpha,x_0}^{(n)}\left( \mathcal{J} \right)$ is defined as \eqref{eq:mesoscopic f}. We say $x_0\in \mathbb{R}$ is near the edges if one of the conditions in \eqref{eq:def edges} is satisfied.
Recall, from Section~\ref{sec:pre cumulant}, that the $m$-th cumulant with $m\geq 2$ for the mesoscopic linear statistics $\mathcal{C}_m(X^{(n)}_{f,\alpha,x_0}) $ is given, via \eqref{eq:cumulant 2}.
To show that the limiting fluctuations of $X^{(n)}_{f,\alpha,x_0}$ are Gaussian, it is sufficient to show that the second cumulant converges to a positive number and all cumulants of order $m\geq 3$ converge to zero, as indicated by Lemma~\ref{lemma: dominated}. 

It turns out that the asymptotics of the $m$-th cumulant ($m\in \mathbb{N}$) only depends on the recurrence coefficients of order around $n$, i.e.,  $a_{j,n}$, $b_{j,n}$ for all $j\sim n$ as $n\to\infty$. The precise window depends on the scale considered. This has been observed in various setups as in \cite{breuer2016universality,breuer2017central,duits2024lozenge}. We will show that this is the case in our setup as well. To this end, we will truncate the (semi-infinite) Jacobi matrix $\mathcal{J}$ into the block around the $n,n$-th entry. For technical reasons, we will conduct a two-step truncation. To be precise, let $\beta=\frac{\alpha}{2}+\frac{\varepsilon}{3}$ such that $0<\frac{\alpha}{2}<\beta<\frac{\alpha+1}{3}<1$. Let us define $J^{(r)}\coloneqq \mathcal{J}-x_0-\frac{\eta_r}{n^\alpha}$ and further define

\begin{equation}\label{eq:truncation 1}
J_{n+2mn^\beta}^{(r)} \coloneqq P_{n+2mn^\beta}J^{(r)}P_{n+2mn^\beta}+Q_{n+2mn^\beta} , \qquad F_{n+2mn^\beta}\coloneqq \sum_{r=1}^{2M}c_r  \left( J_{n+2mn^\beta}^{(r)} \right)^{-1},
\end{equation}
\begin{equation}
J_{n\pm 2mn^\beta}^{(r)}\coloneqq P_{n+2mn^\beta}Q_{n-2mn^\beta}J^{(r)}Q_{n-2mn^\beta}P_{n+2mn^\beta} +P_{n-2mn^\beta}+Q_{n+2mn^\beta}, \label{eq:truncation 2}
\end{equation}
\begin{equation}
F_{n\pm 2mn^\beta}\coloneqq \sum_{r=1}^{2M}c_r  \left( J_{n\pm 2mn^\beta}^{(r)} \right)^{-1} .\label{eq:truncation 3}
\end{equation}
The subscript $n\pm 2mn^\beta$ emphasises that the relevant entries are those with indices ranging between $n-2mn^\beta$ and $n+2mn^\beta$. The parameters $M$ and $c_r$ come from \eqref{eq:resolvent fun}.

The following is our first result about these truncated matrices, which is an alternative to the Combes-Thomas estimates for $J_{n+ 2mn^\beta}^{(r)}$ and  $	J_{n\pm 2mn^\beta}^{(r)}$. 
\begin{figure}[ht]
\begin{center}
\begin{tikzpicture}[
    x=1cm,y=1cm,
    every node/.style={font=\small},
    mat/.style={draw, line width=0.8pt},
    shade/.style={fill=gray!18},
    lbl/.style={font=\scriptsize},
    title/.style={font=\normalsize}
]


\begin{scope}[shift={(0,0)}]

\node[title,left] at (0,2)
{$J^{(r)}_{n+2mn^\beta}=$};

\fill[shade] (2.2,0) rectangle (4,1.8);
\node[lbl] at (3.1,1.35){Conditions};
\node[lbl] at (3.1,0.90){of};
\node[lbl] at (3.1,0.45){Theorem~\ref{thm: main 2}};

\draw[mat] (0,0) rectangle (4,4);

\fill (2.2,1.8) circle (1.5pt);
\node[lbl,above left] at (2.2,1.8)
{$n-2mn^\beta$};
\fill (4,0) circle (1.5pt);
\node[lbl,below] at (4,0)
{$n+2mn^\beta$};

\end{scope}


\begin{scope}[shift={(8,0)}]

\node[title,left] at (0,2)
{$\left(J^{(r)}_{n+2mn^\beta}\right)^{-1}=$};

\draw[dotted] (2.2,0) rectangle (4,1.8);
\fill[shade]
(0,0)--(3.5,0)--(2.2,1.3)--(0,1.3)--cycle;
\fill[shade]
(2.7,4)--(2.7,1.8)--(4,0.5)--(4,4)--cycle;
\draw[mat] (0,0) rectangle (4,4);

\fill (0,1.3) circle (1.5pt);
\node[lbl,left] at (0,1.3)
{$n-(2m-1)n^\beta$};
\fill (2.7,4) circle (1.5pt);
\node[lbl,above] at (2.7,4)
{$n-(2m-1)n^\beta$};
\fill (2.2,1.8) circle (1.5pt);
\node[lbl,above left] at (2.2,1.8)
{$n-2mn^\beta$};
\fill (4,0) circle (1.5pt);
\node[lbl,below] at (4,0)
{$n+2mn^\beta$};

\node[lbl] at (1.2,0.8)
{Exponentially};
\node[lbl] at (1.2,0.5){Small};

\node[lbl,rotate=90] at (3.3,3)
{Exponentially};
\node[lbl,rotate=90] at (3.6,3)
{Small};

\end{scope}


\begin{scope}[shift={(0,-5)}]

\node[title,left] at (0,2)
{$J^{(r)}_{n\pm2mn^\beta}=$};

\fill[shade] (2.2,0) rectangle (4,1.8);
\node[lbl] at (3.1,1.35){Conditions};
\node[lbl] at (3.1,0.90){of};
\node[lbl] at (3.1,0.45){Theorem~\ref{thm: main 2}};

\draw (0,1.8) -- (4,1.8);
\node[lbl] at (1,3)
{$P_{n-2mn^\beta}$};
\draw (2.2,0) -- (2.2,4);
\node[lbl] at (1.1,1)
{$0$};
\node[lbl] at (3.1,3)
{$0$};
\draw[mat] (0,0) rectangle (4,4);

\fill (2.2,1.8) circle (1.5pt);
\node[lbl,above left] at (2.2,1.8)
{$n-2mn^\beta$};
\fill (4,0) circle (1.5pt);
\node[lbl,below] at (4,0)
{$n+2mn^\beta$};

\end{scope}


\begin{scope}[shift={(8,-5)}]
\node[title,left] at (0,2)
{$\left(J^{(r)}_{n\pm2mn^\beta}\right)^{-1}=$};

\fill[shade]
(2.2,0)--(3.5,0)--(2.2,1.3)--cycle;
\fill[shade]
(2.7,1.8)--(4,1.8)--(4,0.5)--cycle;
\draw (2.2,0) rectangle (4,1.8);

\draw (0,1.8) -- (4,1.8);
\node[lbl] at (1,3)
{$P_{n-2mn^\beta}$};
\draw (2.2,0) -- (2.2,4);
\node[lbl] at (1.1,1)
{$0$};
\node[lbl] at (3.1,3)
{$0$};
\draw[mat] (0,0) rectangle (4,4);

\fill (3.5,0) circle (1.5pt);
\node[lbl,below left] at (3.5,0)
{$n+(2m-1)n^\beta$};
\fill (4,0.5) circle (1.5pt);
\node[lbl,right] at (4,0.5){$n+(2m-1)n^\beta$};
\fill (4,0) circle (1.5pt);
\node[lbl,below] at (4.5,-0.05){$n+2mn^\beta$};

\node[lbl,anchor=west] at (2.2,0.5)
{Exponentially};
\node[lbl,anchor=west] at (2.2,0.2){Small};

\node[lbl,anchor=east] at (4,1.6)
{Exponentially};
\node[lbl,anchor=east] at (4,1.3){Small};

\end{scope}

\end{tikzpicture}
\end{center}
\caption{Illustration of Proposition~\ref{thm: overview F entry estimation}. Given the entries of $J^{(r)}_{n+2mn^\beta}$ and $J^{(r)}_{n\pm2mn^\beta}$ with both indices ranging between $n-2mn^\beta$ and $n+2mn^\beta$ are well behaved, we have that the entries of $\left( J^{(r)}_{n+2mn^\beta} \right)^{-1}$ and $\left( J^{(r)}_{n\pm2mn^\beta} \right)^{-1}$ are exponentially small in the shaded areas respectively. Note that the matrices are block matrices and we ignore the trivial blocks in the illustration. }\label{fig: overview tri-inverse-rough}
\end{figure}



\begin{proposition}\label{thm: overview F entry estimation}
Let $\beta=\frac{\alpha}{2}+\frac{\varepsilon}{3}$ be such that $0<\frac{\alpha}{2}<\beta<\frac{\alpha+1}{3}<1$. Consider $\mathcal{J}$ described above whose entries are $(a_{l,n},b_{l,n})_{l\geq 1}$.  Let $m\in\mathbb{N}$. Assume conditions in Theorem~\ref{thm: main 2} are satisfied. Consider $x_0$ to be near the edges, i.e., \eqref{eq:def edges}. Then there exist constants $C_0>0$,  $d_0>0$ $n_0\in\mathbb{N}$ such that, for any $n>n_0$, $j,k$ with  $|j-k|\geq n^{\beta}$ and $\max\{j,k\}\geq n-(2m-1)n^{\beta}$, 
\begin{align}
  \left| \left( \left( J_{n+ 2mn^\beta}^{(r)} \right)^{-1}\right)_{j,k} \right|\leq & C_0n^{\alpha}e^{-d_0n^{\beta-\frac{\alpha}{2}}} , \label{eq:overview improve CT J1}\\
  \left| \left(\left( J_{n\pm 2mn^\beta}^{(r)} \right)^{-1} \right)_{j,k} \right|\leq & C_0n^{\alpha}e^{-d_0n^{\beta-\frac{\alpha}{2}}}. \label{eq:overview improve CT J2}
\end{align}
Consequently, 
\begin{align}
  \left| \left( F_{n+2mn^\beta} \right)_{j,k} \right|\leq & C_0n^{\alpha}e^{-d_0n^{\beta-\frac{\alpha}{2}}} , \label{eq:overview improve CT 1}\\
  \left| \left( F_{n\pm 2mn^\beta} \right)_{j,k} \right|\leq & C_0n^{\alpha}e^{-d_0n^{\beta-\frac{\alpha}{2}}}. \label{eq:overview improve CT 2}
\end{align}
\end{proposition}
For an illustration see Figure~\ref{fig: overview tri-inverse-rough}. For the proof of Proposition~\ref{thm: overview F entry estimation}, see Section~\ref{sec:proofs overview CT}.

We note that the factor $n^\alpha$ is not optimal in the estimates $\left( J_{n\pm 2mn^\beta}^{(r)} \right)^{-1}$ or $F_{n\pm 2mn^\beta}^{(r)}$. These are improved in Proposition~\ref{prop: overview tri-inverse}. However, this is enough for the next step, which is the following proposition lying in the core of the proof of Theorem~\ref{thm: main 2}. 
\begin{proposition} \label{prop: overview trimming}
 Let $\beta=\frac{\alpha}{2}+\frac{\varepsilon}{3}$ such that $0<\frac{\alpha}{2}<\beta<\frac{\alpha+1}{3}<1$. Let $f$ and $M$ as in \eqref{eq:f imaginary}. Let $m\in\mathbb{N}$ and $m\geq 2$.  Assume the conditions in Theorem~\ref{thm: main 2} are satisfied and $x_0$ is near the edges, i.e., \eqref{eq:def edges}. Then there exist constants $d'>0$, $n_0\in\mathbb{N}$, and   $C_m'>0$ that only depend on $m$ and $M$ such that for all $n>n_0$
\begin{equation}\label{eq:overview cutoff} 
  \left| n^{\alpha m}\mathcal{C}_m(X^{(n)}_{f,\alpha,x_0})-\mathcal{C}_m^{(n)}\left( F_{n\pm 2mn^\beta} \right) \right|\leq C_m' e^{-d'n^{\beta-\frac{\alpha}{2}}} .
\end{equation}
\end{proposition}

We give the proof for $m=2$ here. The proof for general $m$ follows the same principles, but is rather technical and we postpone the details to Section~\ref{sec:proof overview trim}.

\begin{proof}[Proof of Propsition~\ref{prop: overview trimming} ]

Consider the case $m=2$.

The second cumulant is $\mathcal{C}_2(X^{(n)}_{f,\alpha,x_0}) =  \Tr \left( f_{\alpha,x_0}^{(n)}\left( \mathcal{J} \right)P_nf_{\alpha,x_0}^{(n)}\left( \mathcal{J} \right)P_n \right)-\Tr\left( \left(f_{\alpha,x_0}^{(n)}\left( \mathcal{J} \right)\right)^{(2)}P_n \right)$. Then by Lemma~\ref{lemma:boundary} we have 

\begin{equation}\label{eq:c2Proof1}
  \left\|n^\alpha P_n\left(f_{\alpha,x_0}^{(n)}(\mathcal{J})\right) - P_nF_{n+2mn^\beta}\right\|_1\leq \left|a_{n+2mn^\beta,n}\right|\sum_{r=1}^{2M} \left\|P_n\left( J_{n\pm 2mn^\beta}^{(r)} \right)^{-1} E_{n+2mn^\beta}\right\|_1\left\|B_{r,n}^{(0)}\right\|_\infty.
\end{equation}
 By \eqref{eq:overview improve CT 1} in Proposition~\ref{thm: overview F entry estimation} we conclude that the trace norm is exponentially small. Since the operator norm and $a_{n+2mn^\beta,n}$ is uniformly bounded as $n\to\infty$, we obtain that \eqref{eq:c2Proof1} is exponentially small. Similarly,
\begin{multline}\label{eq:c2Proof2}
  \left\|n^{2\alpha} P_n\left(f_{\alpha,x_0}^{(n)}(\mathcal{J})\right)^{(2)} -   P_nF_{n+2mn^\beta}^{2}\right\|_1\leq \left|a_{n+2mn^\beta,n}\right|\sum_{r=1}^{2M} n^{\alpha} \left\|P_n\left( J_{n\pm 2mn^\beta}^{(r)} \right)^{-1} E_{n+2mn^\beta}\right\|_1\left\|B_{r,n}^{(1)}\right\|_\infty \\
  + \left|a_{n+2mn^\beta,n}\right|\sum_{r=1}^{2M} n^{\alpha}\left\|P_nF_{n+2mn^\beta}\left( J_{n\pm 2mn^\beta}^{(r)} \right)^{-1} E_{n+2mn^\beta}\right\|_1\left\|B_{r,n}^{(0)}\right\|_\infty.
\end{multline}
The first line on the right-hand side of \eqref{eq:c2Proof2} is exponentially small by the same argument as \eqref{eq:c2Proof1}. Note that the trace norm of the second line on the right-hand side of \eqref{eq:c2Proof2} can be estimated to be a double sum
\begin{equation}
  \left\|P_nF_{n+2mn^\beta}\left( J_{n\pm 2mn^\beta}^{(r)} \right)^{-1} E_{n+2mn^\beta}\right\|_1\leq \sum_{j=1}^{n}\sum_{k=1}^{n+2mn^\beta}\left|\left(F_{n+2mn^\beta}\right)_{j,k}\left(\left( J_{n\pm 2mn^\beta}^{(r)} \right)^{-1}\right)_{k,n+2mn^\beta}\right|,
\end{equation}
which is also exponentially small by \eqref{eq:overview improve CT 1} and \eqref{eq:overview improve CT 2} in Proposition~\ref{thm: overview F entry estimation}. Recall definition  $\mathcal{C}_m^{(n)}$ as \eqref{eq:cumulant operator}. This shows that $\left|n^{m\alpha}\mathcal{C}_m(X^{(n)}_{f,\alpha,x_0})-\mathcal{C}_m^{(n)}(F_{n+2mn^\beta})\right|$ is exponentially small for $m=2$. 

Next, for any linear operators $\mathcal{A}$ and $\mathcal{B}$, we can write the difference of the cumulants by adding and subtracting extra terms to be
\begin{equation}\label{eq:overview algebra of cumulant 2}
  \mathcal{C}_2^{(n)}(\mathcal{A})-\mathcal{C}_2^{(n)}(\mathcal{B}) = -\Tr \left( \left( \mathcal{A}-\mathcal{B} \right)Q_n\mathcal{A}P_n \right) -\Tr \left(\mathcal{B}Q_n\left( \mathcal{A}-\mathcal{B} \right)P_n \right).
\end{equation}
Assume  $\mathcal{A}$ and $\mathcal{B}$ are symmetric (not necessarily Hermitian). Use the trace norm inequality $\|AB\|_1\leq \|A\|_1\|B\|_\infty$ to obtain 
\begin{equation}\label{eq:overview cumulant cancellation}
  \left| \mathcal{C}_2^{(n)}(\mathcal{A})-\mathcal{C}_2^{(n)}(\mathcal{B}) \right| \leq \left\|P_n\left( \mathcal{A}-\mathcal{B} \right)Q_n\right\|_1\left( \|\mathcal{A}\|_\infty+\|\mathcal{B}\|_\infty\right).
\end{equation}
Then we take $\mathcal{A}=F_{n+2mn^\beta}$ as \eqref{eq:truncation 1} and $\mathcal{B}=F_{n\pm2mn^\beta}$ as \eqref{eq:truncation 3}. We use the resolvent identity and the fact that $\left( J^{(r)}_{n\pm 2mn^\beta} \right)^{-1}Q_n=Q_{n-2mn^\beta}\left( J^{(r)}_{n\pm 2mn^\beta} \right)^{-1}Q_n$ to obtain
\begin{equation*}
  \left( \left( J^{(r)}_{n+2mn^\beta} \right)^{-1} -\left( J^{(r)}_{n\pm2mn^\beta} \right)^{-1}\right)Q_n=-\left( J^{(r)}_{n+2mn^\beta} \right)^{-1} \tilde{E}\left( J^{(r)}_{n\pm2mn^\beta} \right)^{-1}Q_n,
\end{equation*}
where $\tilde{E}$ only has two entries that are not zeros, i.e., 
\begin{equation}
  \begin{cases}
    (\tilde{E})_{i,j}=(\tilde{E})_{j,i}=a_{n-2mn^\beta}, &\quad \text{for $i=n-2mn^\beta$ and $j=n-2mn^\beta-1$,}\\
    (\tilde{E})_{i,j}=0, & \quad \text{otherwise}.
  \end{cases}
\end{equation}
Note that the entries of matrix $\tilde{E}\left( J^{(r)}_{n\pm2mn^\beta} \right)^{-1}Q_n$ are zeros except for those at the last $2mn^\beta$ columns of the first row. Hence, 
\begin{multline}\label{eq:preview estimate diff2}
  \left| \mathcal{C}_2^{(n)}(F_{n+2mn^\beta})-\mathcal{C}_2^{(n)}(F_{n\pm2mn^\beta}) \right| \\
  \leq \sum_{r=1}^{2M}|c_r| \left\|\left(J^{(r)}_{n+2mn^\beta} \right)^{-1}  \right\|_\infty\left\|\tilde{E}\left( J^{(r)}_{n\pm2mn^\beta} \right)^{-1}Q_n\right\|_1\bigg( \|F_{n+2mn^\beta}\|_\infty+\|F_{n\pm2mn^\beta}\|_\infty\bigg).
\end{multline}
Similarly, by \eqref{eq:overview improve CT J2} of Proposition~\ref{thm: overview F entry estimation}, we have that \eqref{eq:preview estimate diff2} is exponentially small.

The estimate \eqref{eq:preview estimate diff2} show that  $\left| \mathcal{C}_m^{(n)}(F_{n+2mn^\beta})-\mathcal{C}_m^{(n)}(F_{n\pm2mn^\beta}) \right| $ is exponentially small for $m=2$.

For the general case $m>2$, the algebra of the difference of the cumulants is more complicated than \eqref{eq:overview algebra of cumulant 2}. However, the strategy is similar. For details, see the complete proof of Proposition~\ref{prop: overview trimming} in Section~\ref{sec:proof overview trim}.
\end{proof}

Proposition~\ref{prop: overview trimming} implies that it is sufficient to study the truncated operator $F_{n\pm2mn^\beta}$ as $n\to\infty$. By the definition of $F_{n\pm2mn^\beta}$ in \eqref{eq:truncation 3}, essentially, we are left to study the resolvent $\left( J_{n\pm 2mn^\beta}^{(r)} \right)^{-1} $ defined in \eqref{eq:truncation 2}.

\begin{proposition}\label{prop: overview tri-inverse}
Let $\beta=\frac{\alpha}{2}+\frac{\varepsilon}{3}$ such that $0<\frac{\alpha}{2}<\beta<\frac{\alpha+1}{3}<1$. Assume conditions in Theorem~\ref{thm: main 2} are satisfied. Consider $x_0$ to be near the edges, i.e., \eqref{eq:def edges}. Then there exists a decomposition of  $\left( J_{n\pm 2mn^\beta}^{(r)} \right)^{-1}  $ such that

\begin{equation}\label{eq:overview J inverse}
  \left( J_{n\pm 2mn^\beta}^{(r)} \right)^{-1}  = T_{n\pm 2mn^\beta}(\eta_r) + H_{n\pm 2mn^\beta}(\eta_r) 
\end{equation} 
where $ T_{n\pm 2mn^\beta}(\eta_r)$ and $ H_{n\pm 2mn^\beta}(\eta_r) $ are  such that as $n\to\infty$,  for all $ j, k=n-2mn^\beta,\dots, n+2mn^\beta$,
\begin{equation}\label{eq:overview T Toeplitz}
  (T_{n\pm 2mn^\beta}(\eta_r))_{j,k} 
  =  \frac{s(-s)^{|k-j|}\prod_{l=\min\{j,k\}}^{\max\{j,k\}-1}\left( 1- \mysqrt{\frac{-s\eta_r}{ a_{n,n}n^\alpha}}+\xi_l^{(r)}\right)}{2a_{n,n}\mysqrt{\frac{-s\eta_r}{a_{n,n}n^\alpha }}}\left( 1+o(1) \right),
\end{equation}
where $s = 1$ for left edges and $s=-1$ for right edges and  $\max_{l=n-2mn^\beta,\dots, n+2mn^\beta}|\xi_l^{(r)}|=o(n^{-\alpha/2})$, and by convention $\prod_{l=j}^{j-1}\equiv 1$. We also have as $n\to\infty$, for all $j,k,j+l,k+l=n-2mn^\beta,\dots, n+2mn^\beta$,
\begin{equation}
  \frac{(T_{n\pm 2mn^\beta}(\eta))_{j,k}  }{(T_{n\pm 2mn^\beta}(\eta))_{j+l,k+l} }
  = 1+o\left( n^{-\beta+\frac{\alpha}{2}} \right),  \label{eq:overview tri-inverse T ratio}
\end{equation}
\begin{equation}\label{eq:overview H small}
  \left| (H_{n\pm 2mn^\beta}(\eta_r))_{j,k} \right| \leq C n^{\frac{\alpha}{2}} \left( e^{-d n^{-\frac{\alpha}{2}}(\max\{j,k\}-n+2mn^\beta)}+e^{-d n^{-\frac{\alpha}{2}}(n+2mn^\beta-\min\{j,k\})} \right), 
\end{equation}
for some constant $C,d>0$.
\end{proposition}
For the proof of Proposition~\ref{prop: overview tri-inverse}, see Section~\ref{sec:proof tri-inverse}. 

Compared with \eqref{eq:CB should be finer},  \eqref{eq:overview J inverse} can be regarded as an analogue of the inverse formula of a Toeplitz operator. Formula~\eqref{eq:overview T Toeplitz} implies that the entries along the same diagonals of $T_{n\pm 2mn^\beta}$ are almost the same. The error is controlled by \eqref{eq:overview tri-inverse T ratio}. 

Note that the trace norm of $ H_{n\pm 2mn^\beta}(\eta_r) $ is much smaller than that of the $ T_{n\pm 2mn^\beta}(\eta_r) $ by \eqref{eq:overview H small}. We will also show that its contribution to the cumulant is exponentially small in the following proposition.

\begin{proposition} \label{prop: overview trimming 2} 
Let $\beta=\frac{\alpha}{2}+\frac{\varepsilon}{3}$ such that $0<\frac{\alpha}{2}<\beta<\frac{\alpha+1}{3}<1$. 	Let $m\in\mathbb{N}$ and $m\geq 2$.  Assume conditions in Theorem~\ref{thm: main 2} are satisfied. Consider $x_0$ to be at the edges.  Then there exists a constant $C_m'>0$ that only depends on $m$ and $M$ such that 
\begin{equation}\label{eq:trim_cumulant}
  \left| n^{\alpha m}\mathcal{C}_m(X^{(n)}_{f,\alpha,x_0})-\mathcal{C}_m^{(n)}\left( \sum_{r=1}^{2M}c_rT_{n\pm 2mn^\beta}\left(\eta_r \right) \right) \right|\leq  C_m' e^{-d'n^{\beta-\frac{\alpha}{2}}}, 
\end{equation}
for all $n>n_0$ for some constant $n_0\in\mathbb{N}$.
\end{proposition}
For the proof of Proposition~\ref{prop: overview trimming 2}, see Section~\ref{sec:proof overview trim 2}.

Proposition~\ref{prop: overview trimming 2} reduces asymptotic analysis of the cumulants to the matrices $T_{n\pm2mn^\beta}(\eta_r)$.  These matrices have exponential off-diagonal decay and approximately constant diagonals, but are not exactly Toeplitz. We will, in Section~\ref{sec:strong szego}, give a quantitative estimate of the cumulant generating function of  $\sum_{r}c_rT_{n\pm 2mn^\beta}\left(\eta_{r} \right)$, which is controlled by the trace of a truncated commutator. This results in the following proposition.

\begin{proposition}\label{thm: overview strong szego}
Let $\beta=\frac{\alpha}{2}+\frac{\varepsilon}{3}$ such that $0<\frac{\alpha}{2}<\beta<\frac{\alpha+1}{3}<1$. Assume conditions in Theorem~\ref{thm: main 2} are satisfied. Consider $x_0$ to be at the edges. The test function $f$ is defined as \eqref{eq:resolvent fun}. Then

\begin{equation}
  \lim_{n\to\infty}\mathcal{C}_k^{(n)}\left(\sum_{r=1}^{2M}c_rn^{-\alpha}T_{n\pm 2mn^\beta}\left(\eta_r \right)\right)
  =\begin{cases}
    \sigma_f^2, \quad &k=2\\
    0, \quad &k=3,\dots m.
  \end{cases}
\end{equation}
where, $\sigma_f^2$ is given by \eqref{eq:variance right} and \eqref{eq:variance left}. 

Consequently, Theorem~\ref{thm: main 2} holds for such test functions $f$ as defined as \eqref{eq:resolvent fun}.
\end{proposition}
For the proof of Proposition~\ref{thm: overview strong szego}, see Section~\ref{sec proof overview strong szego}.

The last step of the proof of Theorem~\ref{thm: main 2} is to extend such $f$ to compactly supported functions following a standard procedure, for example, \cite{breuer2016universality, duits2018mesoscopic}. This is elaborated in Section~\ref{sec:Lipschitz extension}.

\subsection{Proof of Theorem~\ref{thm: main 2} }\label{sec:Lipschitz extension}
To finish the proof of Theorem~\ref{thm: main 2}, we need to extend the result of Proposition~\ref{thm: overview strong szego} to the test functions $f\in C^1_c$. An important role of extension is played by the following space of functions. 

\begin{definition} \label{def: weighted Lip}
Let $\mathcal{L}_w$ be a space of functions of $f: \mathbb{R}\to \mathbb{R}$ such that $\lim_{x\to\pm \infty} f(x)=0$ and 
\begin{align}
  \|f\|_{\mathcal{L}_w} \coloneqq \sup_{x,y\in\mathbb{R}} \sqrt{1+x^2}\sqrt{1+y^2}\left| \frac{f(x)-f(y)}{x-y} \right| <\infty.
\end{align}
\end{definition}
$\mathcal{L}_w$ is a normed space with the weighted Lipschitz norm  $\|f\|_{\mathcal{L}_w}$. Note that $C_c^1\subset \mathcal{L}_w$.

\begin{proposition}[Proposition 5.1 in \cite{breuer2016universality}]\label{proposition: variance est}
Let $g(x) = \frac{1}{x-i}$. Then for any $f\in\mathcal{L}_w$ and $n\in\mathbb{N}$ we have 
\begin{equation}\label{eq:var inequality}
  \Var \left(X_{f,\alpha,x_0}^{(n)}\right) \leq \|f\|_{\mathcal{L}_w}^2\left( \Var \left(X_{\Im g,\alpha,x_0}^{(n)}\right)  +\Var \left(X_{\Re g,\alpha,x_0}^{(n)}\right) \right).
\end{equation}
\end{proposition}
A consequence of Proposition~\ref{proposition: variance est} is the following corollary.
\begin{corollary}\label{corollary: variance norm}
Assume all conditions of Theorem~\ref{thm: main 2} are satisfied. Then for any $f\in\mathcal{L}_w$ we have 
\begin{align}
  \limsup_{n\to\infty}\Var \left(X_{f,\alpha,x_0}^{(n)}\right) \leq \frac{1}{8}\|f\|_{\mathcal{L}_w}^2.
\end{align}
\end{corollary}
\begin{proof}
We are to estimate the right-hand side of \eqref{eq:var inequality}. Note that for $g(x)=\frac{1}{x-i}$ we have
\begin{equation*}
  \Im g(x)  =\frac{1}{2i} \left( \frac{1}{x-i}- \frac{1}{x+i} \right), \qquad \Re g(x)  =\frac{1}{2} \left( \frac{1}{x-i}+ \frac{1}{x+i} \right). 
\end{equation*}
Using Proposition~\ref{thm: overview strong szego} for these two functions, we find that at the right edge

\begin{align*}
  \lim\limits_{n\to\infty} \Var \left(X_{\Im g,\alpha,x_0}^{(n)}\right) = \frac{1}{8\pi^2}\int\int_{\mathbb{R}^2}\left(\frac{\Im g (-x^2)-\Im g(-y^2)}{x-y}\right)^2dxdy = \frac{3}{32}, \\
  \lim\limits_{n\to\infty} \Var \left(X_{\Re g,\alpha,x_0}^{(n)}\right)  = \frac{1}{8\pi^2}\int\int_{\mathbb{R}^2}\left(\frac{\Re g (-x^2)-\Re g(-y^2)}{x-y}\right)^2dxdy  = \frac{1}{32}.
\end{align*}		
Similarly, we also obtain that at the left edge
\begin{align*}
  \lim\limits_{n\to\infty} \Var \left(X_{\Im g,\alpha,x_0}^{(n)}\right) = \frac{1}{8\pi^2}\int\int_{\mathbb{R}^2}\left(\frac{\Im g (x^2)-\Im g(y^2)}{x-y}\right)^2dxdy = \frac{3}{32}, \\
  \lim\limits_{n\to\infty} \Var \left(X_{\Re g,\alpha,x_0}^{(n)}\right) = \frac{1}{8\pi^2}\int\int_{\mathbb{R}^2}\left(\frac{\Re g (x^2)-\Re g(y^2)}{x-y}\right)^2dxdy  = \frac{1}{32}.
\end{align*}
These imply that at both edges
\begin{equation*}
  \limsup_{n\to\infty}\left( \Var \left(X_{\Im g,\alpha,x_0}^{(n)}\right)  +\Var \left(X_{\Re g,\alpha,x_0}^{(n)}\right) \right)\leq \frac{1}{8}.
\end{equation*}
Then combine with Proposition~\ref{proposition: variance est} and we conclude this corollary.		
\end{proof}

\begin{lemma}[Lemma 5.3 in \cite{breuer2016universality}]\label{lemma: recolvent approx}
Let $f\in C_c^1(\mathbb{R})$. For any $\varepsilon>0$, there exists an $M\in \mathbb{N}$, $d_r\in \mathbb{R}$, and $\lambda_r\in \mathbb{C}$ with $\Im (\lambda_r)>0$ for all $r=1,\dots, M$ such that, 

\begin{equation}
  \left\| f(x)-\Im \sum_{r=1}^M \frac{d_r}{x-\lambda_r}\right\|_{\mathcal{L}_w}<\varepsilon, 
\end{equation}
where the weighted Lipschitz norm is defined in Definition ~\ref{def: weighted Lip}.
\end{lemma}

Now we are ready to prove Theorem~\ref{thm: main 2}.
\begin{proof}[Proof of Theorem~\ref{thm: main 2} ]
Recall that we define $f_{\alpha,x_0}(x)\coloneqq f(n^\alpha(x-x_0))$. By the inequality $|1-e^{ix}|\leq |x|$ for all real $x$, we have $|e^{ix}-e^{iy}|\leq|x-y|$ for any $x,y$ real. Then use Jensen's inequality to deduce the following bound, for any real-valued functions $f,h$  and real number $t$,
\begin{equation}\label{varapprox3}
  \left|\mathbb{E}\left[e^{i t X_{f,\alpha,x_0}^{(n)}}\right]e^{-i t \mathbb{E}\left[X_{f,\alpha,x_0}^{(n)}\right]}-\mathbb{E}\left[e^{i t X_{h,\alpha,x_0}^{(n)}}\right]e^{-i t\mathbb{E}\left[X_{h,\alpha,x_0}^{(n)}\right]}\right|\leq t\sqrt{\Var\left(X_{f-h,\alpha,x_0}^{(n)}\right)}.
\end{equation}

Moreover, define 
\begin{equation}
  \sigma_{f,L}^2 \coloneqq \frac{1}{8\pi^2}\int\int_{\mathbb{R}^2}\left(\frac{f(x^2)-f(y^2)}{x-y}\right)^2dxdy, \quad 	\sigma_{f,R}^2 \coloneqq \frac{1}{8\pi^2}\int\int_{\mathbb{R}^2}\left(\frac{f(-x^2)-f(-y^2)}{x-y}\right)^2dxdy. 
\end{equation}

We have for both $j\in \{L,R\}$, by Definition~\ref{def: weighted Lip}, 
\begin{equation}
  \sigma_{f,j}^2  \leq \|f\|_{\mathcal{L}_w}^2 \frac{1}{8\pi^2}\int\int_{\mathbb{R}^2}\left(\frac{x+y}{\sqrt{x^4+1}\sqrt{y^4+1}}\right)^2dxdy.
\end{equation}
Use Cauchy residue theorem to compute the double integral $\int\int_{\mathbb{R}^2}\left(\frac{x+y}{\sqrt{x^4+1}\sqrt{y^4+1}}\right)^2dxdy=\pi^2$. Hence,
\begin{equation}\label{varapprox4}
  \sigma_{f,j}^2  \leq  \frac{1}{8}\|f\|_{\mathcal{L}_w}^2.
\end{equation}
Moreover, use the Cauchy-Schwartz inequality to obtain for both $j\in \{L,R\}$ 
\begin{equation}\label{varapprox5}
  \left| \sigma_{f,j}^2-\sigma_{h,j} \right| \leq \sqrt{\sigma_{f+h,j}^2\sigma_{f-h,j}^2 }\leq \frac{1}{8}\|f-h\|_{\mathcal{L}_w}\|f+h\|_{\mathcal{L}_w}.
\end{equation}

Then use triangle inequality, \eqref{varapprox3}, \eqref{varapprox4}, \eqref{varapprox5} and Corollary \ref{corollary: variance norm}, to obtain the following, as $n\to\infty$,
\begin{multline}\label{eq:Fourier C1}
  \left|\mathbb{E}\left[e^{i t X_{f,\alpha,x_0}^{(n)}}\right]e^{-i t \mathbb{E}\left[X_{f,\alpha,x_0}^{(n)}\right]}-e^{-\frac{t^2}{2}\sigma_{f,j}^2}\right| \\
  \leq \left|\mathbb{E}\left[e^{i t X_{f,\alpha,x_0}^{(n)}}\right]e^{-i t \mathbb{E}\left[X_{f,\alpha,x_0}^{(n)}\right]}-\mathbb{E}\left[e^{i t X_{h,\alpha,x_0}^{(n)}}\right]e^{-i t \mathbb{E}\left[X_{h,\alpha,x_0}^{(n)}\right]}\right| 
  + 	\left|e^{-\frac{t^2}{2}\sigma_{f,j}^2}-e^{-\frac{t^2}{2}\sigma_{h,j}^2}\right|+ 	\left|\mathbb{E}\left[e^{i t X_{h,\alpha,x_0}^{(n)}}\right]e^{-i t \mathbb{E}\left[X_{h,\alpha,x_0}^{(n)}\right]}-e^{-\frac{t^2}{2}\sigma_{h,j}^2}\right| \\
  \leq t\left( \frac{1}{2\sqrt{2}}\|f-h\|_{\mathcal{L}_w} +o(1)\right) + \frac{t^2}{16}\|f-h\|_{\mathcal{L}_w}\|f+h\|_{\mathcal{L}_w} 
  + \left|\mathbb{E}\left[e^{i t X_{h,\alpha,x_0}^{(n)}}\right]e^{-i t \mathbb{E}\left[X_{h,\alpha,x_0}^{(n)}\right]}-e^{-\frac{t^2}{2}\sigma_{h,j}^2}\right|.
\end{multline}

Let $f\in C_c^1(\mathbb{R})$ and any $\varepsilon>0$. By Lemma \ref{lemma: recolvent approx}, there exists $$h=\Im \sum_{r=1}^M \frac{d_r}{x-\lambda_r}$$ such that 
\begin{equation}\label{eq:Lw estimate}
  \left\| f-h\right\|_{\mathcal{L}_w}<\varepsilon. 
\end{equation}
Moreover, since $f\in C_c^1(\mathbb{R})$, we have $	\left\| f\right\|_{\mathcal{L}_w}<\infty$. By triangle inequality, we have
\begin{equation*}
  \left\| f+h\right\|_{\mathcal{L}_w}<2\left\| f\right\|_{\mathcal{L}_w}+\varepsilon<\infty .
\end{equation*}
Apply Proposition~\ref{thm: overview strong szego} to $h$ and we have 
\begin{equation}\label{eq:clt resolvent}
  \lim_{n\to\infty} \left|\mathbb{E}\left[e^{i t X_{h,\alpha,x_0}^{(n)}}\right]e^{-i t \mathbb{E}\left[X_{h,\alpha,x_0}^{(n)}\right]}-e^{-\frac{t^2}{2}\sigma_{h,j}^2}\right|=0.
\end{equation}
Plug \eqref{eq:Lw estimate} and \eqref{eq:clt resolvent} into \eqref{eq:Fourier C1} and we get, for any $j=L$ or $R$, and any $\varepsilon>0$,
\begin{equation*}
  \limsup_{n\to\infty}
  \left|\mathbb{E}\left[e^{i t X_{f,\alpha,x_0}^{(n)}}\right]e^{-i t \mathbb{E}\left[X_{f,\alpha,x_0}^{(n)}\right]}-e^{-\frac{t^2}{2}\sigma_{f,j}^2}\right| \leq \frac{t \varepsilon}{2\sqrt{2}}+\frac{t^2\varepsilon^2}{16}\left( 2\|f\|_{\mathcal{L}_w}+\varepsilon\right).
\end{equation*}
Hence 
\begin{equation*}
  \lim_{n\to\infty}
  \left|\mathbb{E}\left[e^{i t X_{f,\alpha,x_0}^{(n)}}\right]e^{-i t \mathbb{E}\left[X_{f,\alpha,x_0}^{(n)}\right]}-e^{-\frac{t^2}{2}\sigma_{f,j}^2}\right|=0.
\end{equation*}
This completes the proof of Theorem~\ref{thm: main 2} for any $f\in C^1_c(\mathbb{R})$ at both left and right edges.
\end{proof}

Note that the proof of Propositions~\ref{thm: overview F entry estimation}, ~\ref{prop: overview trimming}, ~\ref{prop: overview tri-inverse}  and~\ref{prop: overview trimming 2} is postponed to Sections~\ref{sec:improve CB} and~\ref{sec:cumulant truncation} below.

\section{Proof of Propositions~\ref{thm: overview F entry estimation} and~\ref{prop: overview trimming}}\label{sec:improve CB}
In this section we are to prove Propositions~\ref{thm: overview F entry estimation} and~\ref{prop: overview trimming}.

In Subsection~\ref{sec:improve CB 1}, we continue the discussion in Subsection~\ref{sec:pre inversion} and derive Proposition~\ref{prop: tri inverse}, a general theory of tri-diagonal matrices that will be used to prove Propositions~\ref{thm: overview F entry estimation} and~\ref{prop: overview trimming}.

In Subsection~\ref{sec:proofs overview CT}, we estimate the resolvent of both $J_{n+2mn^\beta}^{(r)}$ and the middle block of $J_{n\pm2mn^\beta}^{(r)}$ defined in \eqref{eq:truncation 1} and \eqref{eq:truncation 2} by  Proposition~\ref{prop: tri inverse}. Then, complete the proof of Propositions~\ref{thm: overview F entry estimation} and~\ref{prop: overview trimming}.

\subsection{Inverse of a Tri-diagonal Matrix: Part I} \label{sec:improve CB 1}
The goal of this subsection is to estimate the inverse of the tri-diagonal symmetric matrix $J_N$ defined in \eqref{eq:pre assumption_J_inverse} entry-wise.

\begin{proposition}\label{prop: tri inverse}
  Consider an $N\times N$ non singular tri-diagonal symmetric matrix $J_N$ defined as \eqref{eq:pre assumption_J_inverse} with entries $a_j,b_j\in\mathbb{R}$, $N_1\in\mathbb{N}$ and $z\in\mathbb{C}$ with $\Im z\neq 0$. Let $N_1 \in\mathbb{N}$ with $N_1<N$. Assume $ \Re(b_j- z)>0$ for all $j\geq N_0$. 
  
  Let $\omega_j^+, \omega_j^-$ be defined as \eqref{eq:pre eigenvalue_difference_eq} and $M_j$ be defined as \eqref{eq:eigenvalue_difference_eq_M_j}. Denote
  \begin{equation}\label{eq:ass_betatilde}
    c=\left|\frac{-\omega_{N-1}^+a_{N-1}+b_{N-1}-z}{\omega_{N-1}^-a_{N-1}-b_{N-1}+z}\right| ,
  \end{equation}
  \begin{equation}
    \varepsilon_1 =(N-N_1) \max_{l\geq N_1}\|M_l\|_\infty . \label{eq:ass epsilon1}
  \end{equation}
  and
  \begin{equation}
    \varepsilon_2 =\prod_{l=N_1}^{N-1}\left| \frac{\omega_l^-}{\omega_l^+} \right|. \label{eq:ass epsilon2}
  \end{equation}
  Assume that $\varepsilon_1, \varepsilon_2\in(0,1)$ and $4(1+c)\varepsilon_1+c\varepsilon_2<1$.
  Then for all $j\leq N_1<k$ we have
  \begin{align}
    \left| \left(J_N^{-1}\right)_{j,k}  \right|
    \leq \frac{1+4(1+c)\varepsilon_1+c}{1-4(1+c)\varepsilon_1-c\varepsilon_2}\left\| J_N^{-1} \right\|_{\infty} \prod_{l=N_1}^{k-1}\left|\omega_l^+ \right|^{-1} .
  \end{align} 
  
\end{proposition}

\begin{proof}
  Apply the inverse formula for tri-diagonal matrix \eqref{eq:pre  difference_beta 1} to obtain, for $j\leq N_1<k$,
  \begin{equation}\label{eq:difference_beta 1}
    \left(J_N^{-1}\right)_{j,k}  = (-1)^{k-N_1}\frac{\beta_{k}}{\beta_{N_1}}\left( J_N^{-1} \right)_{j,N_1}  .
  \end{equation} 
  Recall that $\beta_j$ is defined recursively by \eqref{eq:pre beta recurrsive}, i.e., 
  \begin{equation}\label{eq:beta recurrsive}
    \begin{pmatrix}
      \beta_k \\ \beta_{k-1} 
    \end{pmatrix} =  A_kA_{k+1}\cdots A_{N-1} \begin{pmatrix}
      \beta_{N} \\ \beta_{N-1}
    \end{pmatrix}, \quad \beta_{N}=a_{N-1}, \quad \beta_{N-1}=b_{N-1}-z,
  \end{equation}
  where the matrix $A_k$ is defined as \eqref{eq:pre transfer matrix A}. Recall that the eigenvalues $\omega_k^+$ and $\omega_k^-$ of $A_k$ are defined as \eqref{eq:pre eigenvalue_difference_eq}. Then $A_k$ can be diagonalised as $A_k = V_k \Omega_k V_k^{-1} $ where $V_k \coloneqq \begin{pmatrix}
    1 & 1 \\ \omega_k^+ & \omega_k^- 
  \end{pmatrix}$ and $\Omega_k  \coloneqq \begin{pmatrix}
    \omega_k^+ & 0 \\0 &  \omega_k^- 
  \end{pmatrix} $.
  Further define $C_N\coloneqq 0$ and 
  \begin{equation}\label{eq:def Ck}
    C_k\coloneqq V_k^{-1}\left(A_{k}A_{k+1}\dots A_{N-1} - V_k\left(\prod_{l=k}^{N-1} \Omega_l\right)V_{N-1}^{-1}\right)V_{N-1}, \quad k=1,2,\dots, N-1
  \end{equation}
  Rewrite \eqref{eq:beta recurrsive} for $k=1,2,\dots, N-1$ as
  \begin{align}
    \begin{pmatrix}
      \beta_k \\ \beta_{k-1} 
    \end{pmatrix}  = & V_k\left(C_k+\prod_{l=k}^{N-1} \Omega_l\right)V_{N-1}^{-1}\begin{pmatrix}
      \beta_{N} \\ \beta_{N-1} 
    \end{pmatrix} .
  \end{align}
  Denote $\tilde{\beta}_1= \omega_{N-1}^-a_{N-1}-b_{N-1}+z$, $\tilde{\beta}_2= -\omega_{N-1}^+a_{N-1}+b_{N-1}-z$ such that $\begin{pmatrix}
    \tilde{\beta}_1 \\ \tilde{\beta}_2
  \end{pmatrix}=(\omega_{N-1}^--\omega_{N-1}^+)V_{N-1}^{-1}\begin{pmatrix}
    \beta_{N} \\ \beta_{N-1} 
  \end{pmatrix}$. Hence, we have for $k=1,\dots,N-1$ 
  \begin{equation}\label{eq:beta 1}
    \beta_{k} =  \frac{\left((C_k)_{1,1}+(C_k)_{2,1}+\prod_{l=k}^{N-1}\omega_l^+\right)\tilde{\beta}_1+\left((C_k)_{1,2}+(C_k)_{2,2}+\prod_{l=k}^{N-1}\omega_l^-\right)\tilde{\beta}_2}{\omega_{N-1}^--\omega_{N-1}^+} .
  \end{equation}
  By convention, let $C_N\equiv 0 $ and $\prod_{l=N}^{N-1}\equiv 1$. Then we extend \eqref{eq:beta 1} to $k=1,\dots, N$, since $\beta_N=a_{N-1}$. Use the construction of $M_j$ in \eqref{eq:eigenvalue_difference_eq_M_j} to write $V_{k}^{-1}V_{k+1} = M_k+\mathrm{Id}$. Then rewrite $C_k$ via the following equality 
  \begin{equation}\label{eq:difference_telescoping}
    A_kA_{k+1}\dots A_{N-1}  = V_k \Omega_k(\mathrm{Id}+M_k) \Omega_{k+1}(\mathrm{Id}+M_{k+1})\dots  \Omega_{N-2}(\mathrm{Id}+M_{N-2}) \Omega_{N-1}V_{N-1}^{-1}.
  \end{equation}
  By assumption $ \Re(b_l- z)>0$, we have $|\omega_l^+|\geq |\omega_l^-|$. Hence, $\|\Omega_l\|_\infty<|\omega_l^+|$. Plug \eqref{eq:difference_telescoping} into \eqref{eq:def Ck} to estimate 
  \begin{equation}
    \|C_k\|_\infty \leq \left( \prod_{l=k}^{N-2}\bigg( 1+\|M_l\|_\infty \bigg) -1\right)\prod_{l=k}^{N-1}|\omega_l^+|.
  \end{equation}
  Then by \eqref{eq:ass epsilon1}, $0<\varepsilon_1<1$ and the fact that $(1+x)^y< e^{xy}< 1+2xy$ for all $x,y>0$ with $0<xy<1$, we have
  \begin{equation}
    \|C_k\|_\infty 	 \leq  2\varepsilon_1 \prod_{l=k}^{N-1}|\omega_l^+| \label{eq:ck_epsilon1} .
  \end{equation}
  Denote $(C_k)_{j,l}$ to be the  entry of matrix $C_k$ at the $j-th$ row and $l-th$ column. Then let 
  \begin{equation}\label{eq:estiamte ck}
    C(k)\coloneqq  \left((C_k)_{1,1}+(C_k)_{2,1}\right)\prod_{l=k}^{N-1}(\omega_l^+)^{-1}+\left((C_k)_{1,2}+(C_k)_{2,2}\right)\frac{\tilde{\beta}_2}{\tilde{\beta}_1}\prod_{l=k}^{N-1}(\omega_l^+)^{-1} .
  \end{equation}
  By \eqref{eq:ass_betatilde} we have $\left| \frac{\tilde{\beta}_2}{\tilde{\beta}_1} \right|=c$ and thus 
  $|C(k)|\leq 4(1+c)\varepsilon_1.$
  Moreover, use \eqref{eq:estiamte ck} to rewrite \eqref{eq:beta 1} as
  \begin{equation}\label{eq:beta 2}
    \beta_{k} =  \frac{\tilde{\beta}_1\prod_{l=k}^{N-1}\omega_l^+}{\omega_{N-1}^--\omega_{N-1}^+} \left( 1+\frac{\tilde{\beta}_2}{\tilde{\beta}_1}\prod_{l=k}^{N-1}\frac{\omega_l^-}{\omega_l^+}+C(k) \right) .
  \end{equation}
  Then plug \eqref{eq:beta 2} into \eqref{eq:difference_beta 1} to obtain that for $j\leq N_1< k$
  \begin{align}
    \left(J_N^{-1}\right)_{j,k} 
    & = (-1)^{k-N_1}\left( J_N^{-1} \right)_{j,N_1}\left( \prod_{l=N_1}^{k-1}\omega_l^+ \right)^{-1} \frac{1+\frac{\tilde{\beta}_2}{\tilde{\beta}_1}\prod_{l=k}^{N-1}\frac{\omega_l^-}{\omega_l^+}+C(k)}{1+\frac{\tilde{\beta}_2}{\tilde{\beta}_1}\prod_{l=N_1}^{N-1}\frac{\omega_l^-}{\omega_l^+}+C(N_1)} .
  \end{align}
  Note that $\left|\frac{\omega_l^-}{\omega_l^+}  \right|<1$. Thus whenever $4(1+c)\varepsilon_1+c\varepsilon_2<1$, we have, by triangle inequality,
  \begin{equation}
    \left| \left(J_N^{-1}\right)_{j,k}  \right|
    \leq \frac{1+4(1+c)\varepsilon_1+c}{1-4(1+c)\varepsilon_1-c\varepsilon_2}\left\| J_N^{-1} \right\|_{\infty} \prod_{l=N_1}^{k-1}\left|\omega_l^+ \right|^{-1} 
  \end{equation}
\end{proof}

\subsection{Proof of Proposition~\ref{thm: overview F entry estimation}}\label{sec:proofs overview CT}
In this subsection, we use Proposition~\ref{prop: tri inverse} to prove Proposition~\ref{thm: overview F entry estimation}, whose assumptions are the same as Theorem~\ref{thm: main 2}.	

Let $\beta=\frac{\alpha}{2}+\frac{\varepsilon}{3}$ such that  $0<\frac{\alpha}{2}<\beta<\frac{\alpha+1}{3}<1$ and $m\in \mathbb{N}$. 
We consider the tri-diagonal matrix $J_N$, where
\begin{equation}\label{eq:setup1}
  N=n+2mn^\beta,\quad b_{j}= b_{j,n}-x_0, \quad a_j = a_{j,n}, \quad z=\frac{\eta}{n^\alpha}, 
\end{equation}
where $x_0$ is near the edge, i.e., \eqref{eq:def edges} is assumed. Indeed, we are looking at the case where $J_N=J_{n+2mn^\beta}^{(r)}$ as defined as \eqref{eq:truncation 1}. Since the subscript $r$ of $\eta_r$ will not play any role in what follows, we suppress it in this section. To consider the case $J_{n\pm2mn^\beta}^{(r)}$, we will point out at the end of the proof of Proposition~\ref{thm: overview F entry estimation} that the proof is the same up to relabelling of the indices. 

Now let
\begin{equation}\label{eq:def I}
  I_{n,m}^{(\beta)}\coloneqq\{j\in\mathbb{N}:N_0\leq j\leq N\}, \quad \text{where } N_0=n-2mn^{\beta}, N=n+2mn^{\beta}.
\end{equation}
Note that $I_{n,m}^{(\beta)}\subset I_n^{(\alpha,\varepsilon)}$, where $I_n^{(\alpha,\varepsilon)}$ is defined as \eqref{eq:index set}. The size of set $ I_{n,m}^{(\beta)}$ is denoted by $\left| I_{n,m}^{(\beta)} \right|$ and is of order $O(n^{\beta})$.

In Theorem~\ref{thm: main 2}, our main assumptions are about the coefficients $a_{j,n}\in\mathbb{R}$ and $b_{j,n}\in\mathbb{R}$. For reader's convenience, we list them in terms of $a_j$ and $b_j$ (under the setup \eqref{eq:setup1}) as follows.
\begin{Condition}[Theorem~\ref{thm: main 2}]\label{ass:slowly varying}
  There exist some constants $c_0,c_1>0$ such that  for all $j\in I_{n,m}^{(\beta)}$, where $I_{n,m}^{(\beta)}$ is defined as \eqref{eq:def I}, we have
  \begin{align}
    c_0<|a_j|<c_1 &,\quad  |b_j|<c_1, \label{ass:recurrence bound1}\\
    |a_j-a_{j-1}|\leq \frac{c_1}{n} &, \quad |b_j-b_{j-1}|\leq \frac{c_1}{n}. \label{eq:assumption_recurrance1} 
  \end{align}    
\end{Condition}

\begin{Condition}[Theorem~\ref{thm: main 2}] \label{ass:rate controlling}
  As $n\to\infty$,  
  \begin{equation}
    \max_{j\in I_{n,m}^{(\beta)}}	\left|a_{j}a_{j-2}-a_{j-1}^2\right| = o(n^{-3\beta+\frac{\alpha}{2}}) \label{ass:convex1} 
  \end{equation}
  \begin{equation}
    \max_{j\in I_{n,m}^{(\beta)}}\left|(b_{j-1}-a_{j})a_{j-2}-(b_{j-2}-a_{j-1})a_{j-1}\right|=o(n^{-3\beta}) , \quad \text{for the left edge } \label{ass:convex2}. 
  \end{equation}
  \begin{equation}
      \max_{j\in I_{n,m}^{(\beta)}}	\left| b_{j-1}-2\sqrt{a_{j}a_{j-1}} \right|=o(n^{-\alpha}), \quad \text{for the left edge } \label{ass:edge}. 
  \end{equation}
\end{Condition}
For the right edge, we should replace  $a_j$ in \eqref{ass:convex2} to be $-a_j$ for all j and replace \eqref{ass:edge} by $	b_{j-1}=-2\sqrt{a_{j}a_{j-1}}+o(n^{-\alpha})$. In the rest of this section, we will only consider the left edge.  The discussion about the right edge of $\mathcal{J}$ is the same as that of the left edge of $-\mathcal{J}$, whose off-diagonals are $-a_{j,n}$ which is the only reason we do not assume $a_{j,n}>0$  in this paper. Note that by Condition~\ref{ass:rate controlling}, the assumption \eqref{ass:convex2} for both left and right edges can be compressed to be $\max_{j\in I_{n,m}^{(\beta)}}\left|(|b_{j-1}|-a_{j})a_{j-2}-(|b_{j-2}|-a_{j-1})a_{j-1}\right|=o(n^{-3\beta}) $.

Also note that the indexing set $I_{n,m}^{(\beta)}$ is a subset of that in the Theorem~\ref{thm: main 2} and hence we make a slightly weaker assumption. The difference is only in reducing the notation in both the statement and the proof.

We continue to explain the implications of these conditions by the following lemmas.

\begin{lemma}\label{lemma: left edge locating}
  Assume conditions of Theorem~\ref{thm: main 2} are satisfied for the left edge, i.e., given Conditions~\ref{ass:slowly varying} and~\ref{ass:rate controlling}. 
  Then given $0<\frac{\alpha}{2}<\beta<1$, we have,
  \begin{equation}
    \liminf_{n\to\infty}\min_{j\in I_{n,m}^{(\beta)}}b_{j}>0  \label{eq:assumption_recurrance_b},
  \end{equation}
  and as $n\to\infty$
  \begin{equation}\label{ass:edge locating}
    \max_{j\in I_{n,m}^{(\beta)}}\left| b_{j-1}^2- 4a_{j}a_{j-1} \right|=o(n^{-\alpha}).
  \end{equation}
\end{lemma}
\begin{proof}
  Recall that $N$ and $N_0$ are defined in \eqref{eq:def I}. Then the size of the set $|I_{n,m}^{(\beta)}|=N-N_0=4mn^\beta$, with $0<\beta<1$. By \eqref{eq:assumption_recurrance1}, we have
  \begin{equation}
    \min_{j\in I_{n,m}^{(\beta)}}b_{j}\geq b_{n-1}-4c_1mn^{\beta-1}.
  \end{equation}
  By \eqref{ass:edge}, we have
  \begin{equation}
    b_{n-1}=2\sqrt{a_{n-1}a_{n}} +o(n^{-\alpha}), \quad \text{as }n\to\infty.
  \end{equation}
  Moreover, by \eqref{ass:edge} and the lower bound of $|a_j|$ in \eqref{ass:recurrence bound1}, we have $\sqrt{a_{n-1}a_{n}}>2c_0>0$. Consequently, we have \eqref{eq:assumption_recurrance_b}.

  Now we prove the second statement. By \eqref{ass:convex2}, together with the lower bound of $a_j$ in Condition~\ref{ass:rate controlling}, we have that for all $j\in I_{n,m}^{(\beta)}$
  \begin{equation}
    \left( \frac{b_{j-1}-a_j}{a_{j-1}} \right)-\left( \frac{b_{j-2}-a_{j-1}}{a_{j-2}} \right)=o(n^{-3\beta}), \quad \text{as }n\to\infty.
  \end{equation}
  Note that the size of the index set $|I_{n,m}^{(\beta)}|=O(n^{\beta})$. Then for all $j\in I_{n,m}^{(\beta)}$
  \begin{equation}\label{eq:edge locating proof 1}
    \left( \frac{b_{j-1}-a_j}{a_{j-1}} \right)-\left( \frac{b_{n-1}-a_{n}}{a_{n-1}} \right)=o\left(n^{-2\beta}\right), \quad \text{as }n\to\infty.
  \end{equation}
  Using the fact that $\left( \sqrt{\frac{a_j}{a_{j-1}}}-1 \right)^2= \frac{a_j}{a_{j-1}}+1-2\sqrt{\frac{a_j}{a_{j-1}}}$, we can rewrite for $j\in I_{n,m}^{(\beta)}$
  \begin{equation}\label{eq:edge locating proof 2}
    \frac{b_{j-1}-a_j}{a_{j-1}}  =  \frac{b_{j-1}}{a_{j-1}} -\left( \sqrt{\frac{a_j}{a_{j-2}}}-1 \right)^2-2\sqrt{\frac{a_j}{a_{j-1}}}+1.
  \end{equation}
  Note that by the Condition~\ref{ass:slowly varying} we have $\left(\sqrt{\frac{a_j}{a_{j-1}}}-1 \right)^2=O(n^{-2})$. Also note that \eqref{eq:edge locating proof 2} holds for all $j\in I_{n,m}^{(\beta)}$ including the special case $j=n$.
  Hence plugging \eqref{eq:edge locating proof 2} into \eqref{eq:edge locating proof 1} gives
  \begin{equation}\label{eq:edge locating proof 3}
    \frac{b_{j-1}}{a_{j-1}} -2\sqrt{\frac{a_j}{a_{j-1}}}- \frac{b_{n-1}}{a_{n-1}} +2\sqrt{\frac{a_n}{a_{n-1}}}=o(n^{-2\beta})+O(n^{-2}), \quad \text{as }n\to\infty.
  \end{equation}
  By the assumption \eqref{ass:edge} in Condition~\ref{ass:rate controlling}, we have $\frac{b_{n-1}}{a_{n-1}} -2\sqrt{\frac{a_n}{a_{n-1}}}=o(n^{-\alpha})$, as $n\to\infty$. Further, by the assumption $0<\frac{\alpha}{2}<\beta<1$, we estimate \eqref{eq:edge locating proof 3} to be
  
  \begin{equation}\label{eq:edge locating proof 4}
    \frac{b_{j-1}}{a_{j-1}} -2\sqrt{\frac{a_j}{a_{j-1}}}=o(n^{-\alpha}), \quad \text{as }n\to\infty.
  \end{equation}
  Multiply $\frac{b_{j-1}}{a_{j-1}} +2\sqrt{\frac{a_j}{a_{j-1}}}$ on both sides of \eqref{eq:edge locating proof 4} to obtain that for all $j\in I_{n,m}^{(\beta)}$
  \begin{equation}
    \left( \frac{b_{j-1}}{a_{j-1}} \right)^2=\frac{4a_j}{a_{j-1}}+o(n^{-\alpha}), \quad \text{as }n\to\infty,
  \end{equation}
  which completes the proof. 
\end{proof}

Note that \eqref{eq:assumption_recurrance_b} implies that $|\omega_j^-|\leq |\omega_j^+|$, since we are taking the principal square root. However, for the right edge, by the same argument, we have $b_j<0$ for all $j\in I_{n,m}^{(\beta)}$. In this case, we can swap the definition of $\omega_j^+$ and $\omega_j^-$ and all conclusions in this section follow with the same argument. Alternatively, for the right edge, we can simply consider $-J_N $ whose diagonals are $-b_j+\frac{\eta}{n^{\alpha}}$ and off diagonals are $-a_j$.  Hence, all discussions in this section can be directly applied to $ (-J_N)^{-1} $, since $-b_j>0$ for the right edge.

\begin{lemma} \label{lemma: finner est omega}
  Consider $0<\frac{\alpha}{2}<\beta<1$. Assume the Conditions~\ref{ass:slowly varying} and~\ref{ass:rate controlling} are satisfied. Then we have
  \begin{align}
      \max_{j\in I_{n,m}^{(\beta)}}	\left| \omega_j^{-} -\left(  1-\mysqrt{\frac{-\eta}{n^\alpha |a_{N-1}|}} \right) sgn(a_{N-1}) \right|=o(n^{-\frac{\alpha}{2}}), \label{eq:omega j-}\\
      \max_{j\in I_{n,m}^{(\beta)}}	\left| \omega_j^{+} -\left(  1+\mysqrt{\frac{-\eta}{n^\alpha |a_{N-1}|}} \right) sgn(a_{N-1}) \right|=o(n^{-\frac{\alpha}{2}}),
  \end{align}
  where $sgn(x)=1$ if $x\geq 0$ and $sgn(x)=-1$ if $x< 0$.   The choice of $a_{N-1}$ in the formulas above is arbitrary. It can be replaced by any $a_j$ for $j\in I_{n,m}^{(\beta)}$. 
\end{lemma}
\begin{proof}
  Recall that $\omega_j^-$ is given by 
  \begin{align}\label{eq:omega exact}
    \omega_j^{-}=\frac{b_{j-1}-\frac{\eta}{n^\alpha}-\mysqrt{\left( b_{j-1}-\frac{\eta}{n^\alpha} \right)^2-4a_{j-1}a_{j}}}{2a_{j-1}}.
  \end{align}
  
  By Condition \ref{ass:slowly varying} we have $\frac{a_l}{a_{l-1}}=1+O(n^{-1})$. Further, by Lemma~\ref{lemma: left edge locating}, we obtain 
  \begin{equation}\label{eq:estimate recurrence 1}
     \left( \frac{b_{j-1} }{2a_{j-1}}\right)^2-\frac{a_{j}}{a_{j-1}} = o(n^{-\alpha}), \quad \frac{b_{j-1}}{2|a_{j-1}|} = 1+o(n^{-\alpha}) +O(n^{-1}), \quad \text{as } n\to\infty.
  \end{equation}
  Condition \ref{ass:slowly varying} also implies $a_{j-1}=a_{N-1}+O(n^{\beta-1})$. We can then further estimate as $n\to\infty$
  \begin{multline*}
    \mysqrt{\left( \frac{b_{j-1}-\frac{\eta}{n^\alpha} }{2a_{j-1}}\right)^2-\frac{a_{j}}{a_{j-1}}}= 	\mysqrt{\left( \frac{b_{j-1} }{2a_{j-1}}\right)^2-\frac{a_{j}}{a_{j-1}}-\frac{b_{j-1}}{2a_{j-1}} \frac{\eta}{n^\alpha a_{j-1}}+\left(  \frac{\eta}{2n^\alpha a_{j-1}} \right)^2} \\
    =\mysqrt{o(n^{-\alpha})-\left( 1+ o(n^{-\alpha}) +O(n^{-1}) \right)\frac{\eta}{2n^\alpha\left(  |a_{N-1}| +O(n^{\beta-1})\right)} +O(n^{2\alpha})} 
    = \mysqrt{\frac{-\eta}{n^\alpha |a_{N-1}|}} +o(n^{-\frac{\alpha}{2}}) .
  \end{multline*}
  
  Hence we proved \eqref{eq:omega j-} for $\omega_j^{-}$. The proof for $\omega_j^{+}$ follows from the same argument. 
\end{proof}
Note that for the right edge where $b_j<0$, we will get the following by the same argument, 
\begin{align}
  \omega_j^{-} =\left(  -1-\mysqrt{\frac{\eta}{n^\alpha |a_{N-1}|}} \right) sgn(a_{N-1})+o(n^{-\frac{\alpha}{2}}), \\
  \omega_j^{+} =\left(  -1+\mysqrt{\frac{\eta}{n^\alpha |a_{N-1}|}} \right) sgn(a_{N-1})+o(n^{-\frac{\alpha}{2}}).
\end{align}

\begin{lemma}\label{lemma:small M} Let $0<\frac{\alpha}{2}<\beta<\frac{\alpha+1}{3}<1$.
  Assume Conditions~\ref{ass:slowly varying}, and~\ref{ass:rate controlling} are satisfied. Then we have as $n\to\infty$
  \begin{equation}
    \max_{j\in I_{n,m}^{(\beta)}}\left| \omega_j^+-\omega_{j-1}^+ \right| = o(n^{-3\beta+\frac{\alpha}{2}}),\quad 	\max_{j\in I_{n,m}^{(\beta)}}\left| \omega_j^--\omega_{j-1}^- \right| = o(n^{-3\beta+\frac{\alpha}{2}}).
  \end{equation}
  Consequently we also have as $n\to\infty$
  \begin{equation}
    \max_{j\in I_{n,m}^{(\beta)}}\|M_j\|_\infty = o(n^{-3\beta+\alpha}) .
  \end{equation}
\end{lemma}
\begin{proof}
  First note that the difference of equation \eqref{ass:convex1} and \eqref{ass:convex2} amounts to
  \begin{equation}
    \max_{j\in I_{n,m}^{(\beta)}}\left| b_{j-1}a_{j-2}-b_{j-2}a_{j-1}  \right|= o(n^{-3\beta+\frac{\alpha}{2}}), \quad \text{as } n\to\infty \label{ass:convex3} .
  \end{equation}
  Recall, from the definition \eqref{eq:pre eigenvalue_difference_eq}, we have $\omega_j^+-\omega_j^-=\frac{\mysqrt{(b_{j-1}-\frac{\eta}{n^{\alpha}})^2-4a_ja_{j-1}}}{a_{j-1}}$, where the square root is taken as the principal branch.  By Lemma~\ref{lemma: left edge locating} and \eqref{ass:recurrence bound1} we have $|\omega_j^+-\omega_j^-|^{-1}=O(n^\frac{\alpha}{2})$ for all $j\in I_{n,m}^{(\beta)}$. Moreover, 
  \begin{multline*}
    \omega_j^+-\omega_{j-1}^+=\frac{b_{j-1}}{2a_{j-1}}-\frac{b_{j-2}}{2a_{j-2}}-\frac{\eta}{2n^{\alpha}}\left( \frac{1}{a_{j-1}}-\frac{1}{a_{j-2}}\right) 
    +\mysqrt{\left( \frac{b_{j-1}-\frac{\eta}{n^{\alpha}}}{2a_{j-1}} \right)^2-\frac{a_{j}}{a_{j-1}}}-\mysqrt{\left( \frac{b_{j-2}-\frac{\eta}{n^{\alpha}}}{2a_{j-2}} \right)^2-\frac{a_{j-1}}{a_{j-2}}} .
  \end{multline*}
  The first part  can be estimated as $\frac{b_{j-1}}{2a_{j-1}}-\frac{b_{j-2}}{2a_{j-2}}=o(n^{-3\beta+\frac{\alpha}{2}})$, by \eqref{ass:convex3}. The second part can be estimated as $\frac{\eta}{2n^{\alpha}}\left( \frac{1}{a_{j-1}}-\frac{1}{a_{j-2}}\right)=O(n^{-\alpha-1})$, by Condition~\ref{ass:slowly varying}. For the third part, we have
  \begin{multline}
    \mysqrt{\left( \frac{b_{j-1}-\frac{\eta}{n^{\alpha}}}{2a_{j-1}} \right)^2-\frac{a_{j}}{a_{j-1}}}-\mysqrt{\left( \frac{b_{j-2}-\frac{\eta}{n^{\alpha}}}{2a_{j-2}} \right)^2-\frac{a_{j-1}}{a_{j-2}}} \\
    = \frac{\left( \frac{b_{j-1}-\frac{\eta}{n^{\alpha}}}{2a_{j-1}} -\frac{b_{j-2}-\frac{\eta}{n^{\alpha}}}{2a_{j-2}}\right)\left( \frac{b_{j-1}-\frac{\eta}{n^{\alpha}}}{2a_{j-1}} +\frac{b_{j-2}-\frac{\eta}{n^{\alpha}}}{2a_{j-2}} \right)-\frac{a_{j}}{a_{j-1}}+\frac{a_{j-1}}{a_{j-2}}}{\mysqrt{\left( \frac{b_{j-1}-\frac{\eta}{n^{\alpha}}}{2a_{j-1}} \right)^2-\frac{a_{j}}{a_{j-1}}}+\mysqrt{\left( \frac{b_{j-2}-\frac{\eta}{n^{\alpha}}}{2a_{j-2}} \right)^2-\frac{a_{j-1}}{a_{j-2}}}}.  \label{eq:difference of omega}
  \end{multline}
  By the same argument as before, we have $1$ over the denominator is also of order $O(n^{\frac{\alpha}{2}})$. For the numerator, Lemma~\ref{lemma: left edge locating} together with  \eqref{ass:recurrence bound1} implies that $\frac{b_{j-1}-\frac{\eta}{n^{\alpha}}}{2a_{j-1}}=\left| \frac{a_j}{a_{j-1}} \right|^{\frac{1}{2}}+O(n^{-\alpha})$ . Moreover, Condition~\ref{ass:slowly varying} implies $\frac{a_j}{a_{j-1}}=1+O(n^{-1})$. Since $\alpha\in(0,2)$, we have $\frac{b_{j-1}-\frac{\eta}{n^{\alpha}}}{2a_{j-1}} +\frac{b_{j-2}-\frac{\eta}{n^{\alpha}}}{2a_{j-2}}=2+O(n^{-\frac{\alpha}{2}})$. The numerator of \eqref{eq:difference of omega} now becomes
  \begin{equation}\label{eq:numerator}
    \frac{b_{j-1}-a_j}{a_{j-1}}-\frac{b_{j-2}-a_{j-1}}{a_{j-2}}-\frac{\eta}{n^{\alpha}}\left( \frac{1}{a_{j-1}}-\frac{1}{a_{j-2}}\right) +\left( \frac{b_{j-1}}{a_{j-1}}-\frac{b_{j-2}}{a_{j-2}} \right)O(n^{-\frac{\alpha}{2}}) .
  \end{equation}
  The sum of the first two terms is $o(n^{-3\beta})$ by \eqref{ass:convex2}. The third term is order $O(n^{-\alpha-1})$, by Condition~\ref{ass:slowly varying}. For the rest term, we have $\frac{b_{j-1}}{a_{j-1}}-\frac{b_{j-2}}{a_{j-2}}=o(n^{-3\beta+\frac{\alpha}{2}})$, by  \eqref{ass:convex3}. Then \eqref{eq:numerator} is of order $o(n^{-3\beta})$, by the assumption $\beta<\frac{\alpha+1}{3}$. Combining the arguments above, we have $\omega_j^+-\omega_{j-1}^+=o(n^{-3\beta+\frac{\alpha}{2}})$. Similarly, we also have $\omega_j^--\omega_{j-1}^-=o(n^{-3\beta+\frac{\alpha}{2}})$. Then by the definition of $M_j$ \eqref{eq:eigenvalue_difference_eq_M_j}, we have $\|M_j\|_\infty=o(n^{-3\beta+\alpha})$.
\end{proof}

  Now we are ready to verify the Proposition~\ref{thm: overview F entry estimation}. Recall that $\mathcal{J}$ is the Jacobi matrix with entries $\{a_{j,n},b_{j,n}\}_{j}$. The truncated operators $J_{n+2mn^\beta}^{(r)}$,  $J_{n\pm2mn^\beta}^{(r)}$, $F_{n+2mn^\beta}$ and $F_{n\pm 2mn^\beta} $ are defined in \eqref{eq:truncation 1}, \eqref{eq:truncation 2} and \eqref{eq:truncation 3}.
  
  \begin{proof}[Proof of Proposition~\ref{thm: overview F entry estimation}] 
    First, consider $$J_N = J_{n+2mn^\beta}^{(r)}.$$ 
    It is sufficient to verify the conditions in Proposition~\ref{prop: tri inverse}.
    
    For all $l\in I_{n,m}^{(\beta)}$, by Lemma~\ref{lemma: left edge locating} we have $b_l>0$. By Lemma ~\ref{lemma: finner est omega} we have $\omega_l^{-}=1+O(n^{-\frac{\alpha}{2}})$ and $\omega_l^{-}=1+O(n^{-\frac{\alpha}{2}})$. Hence 
    \begin{equation}
      \left|\frac{-\omega_{n+2mn^\beta-1}^+a_{n+2mn^\beta-1}+b_{n+2mn^\beta-1}-z}{\omega_{n+2mn^\beta-1}^-a_{n+2mn^\beta-1}-b_{n+2mn^\beta-1}+z}\right|  =1+ O(n^{-\frac{\alpha}{2}}), \quad \text{as } n\to\infty.
    \end{equation}
    This shows \eqref{eq:ass_betatilde} is bounded. 
    
    Let $M_j$ be defined as \eqref{eq:eigenvalue_difference_eq_M_j}. By Lemma~\ref{lemma:small M} we have as $n\to\infty$
    \begin{equation}
      \left| I_{n,m}^{(\beta)} \right|\max_{l\in I_{n,m}^{(\beta)}}\|M_l\|_\infty =o(n^{-2\beta+\alpha})
    \end{equation}
    This shows \eqref{eq:ass epsilon1} to be $\varepsilon_1=o(n^{-2\beta+\alpha})$. Moreover, use Lemma ~\ref{lemma: finner est omega}  to estimate that there exist constants $d>0$ and $n_0\in\mathbb{N}$ such that for all $n>n_0$ and all integer $N_1$ with $N_0\leq N_1<N$
    \begin{equation} 
      \prod_{l=N_1}^{N-1}\left| \frac{\omega_l^-}{\omega_l^+} \right| \leq \left( 1-\frac{d}{n^{\frac{\alpha}{2}}} \right)^{N-N_1}\leq e^{-d n^{-\frac{\alpha}{2}}(N-N_1)}.
    \end{equation}
    This shows \eqref{eq:ass epsilon2} to be $\varepsilon_2 =O(e^{-d'n^{\beta-\frac{\alpha}{2}}})$ which is exponentially small. 
    
    Then all assumptions of Proposition~\ref{prop: tri inverse} are cleared. Hence, for all $j\leq N_1<k$ we have 
    \begin{equation}
      \left| \left(J_N^{-1}\right)_{j,k}  \right|
      \leq c\left\| J_N^{-1} \right\|_{\infty}e^{-d n^{-\frac{\alpha}{2}}(k-N_1)} .
    \end{equation} 
    for some constant $c>0$. Then take $N_1=k-n^\beta\geq N_0$ to obtain \eqref{eq:overview improve CT J1}.
    
    Note that above argument is taken as $b_j>0$ as assumed in \eqref{eq:assumption_recurrance_b}. For $b_j<0$, we consider $-J_N $ whose diagonals are $-b_j+\frac{\eta}{n^{\alpha}}$ and off diagonals are $-a_j$. Apply the result above, and we reach the same estimate for $ -J_N^{-1} $.
    
    Note that  the entry $\left( J_{n\pm2mn^\beta}^{(r)} \right)_{j,k}$ is the same as $\left( J_{n+2mn^\beta}^{(r)} \right)_{j,k}$ for all $j,k \in I_{n,m}^{(\beta)}$  and the rest of the entries are in the trivial blocks. Hence, the result \eqref{eq:overview improve CT J2} for $J_N = J_{n\pm2mn^\beta}^{(r)}$ follows the same approach by relabelling the indices appropriately.
    
    Then, the estimates \eqref{eq:overview improve CT 1} and \eqref{eq:overview improve CT 2} follow from the definition of $F_{n+2mn^\beta}$ and $F_{n\pm 2mn^\beta}$ in \eqref{eq:truncation 1} and \eqref{eq:truncation 3} and triangle inequality. 
  \end{proof}

  \subsection{Comparison in the Trace Norm}
  
  The following lemma is essential to prove Proposition~\ref{prop: overview trimming}.
  
  \begin{lemma}\label{cutoff}
    Let $0<\frac{\alpha}{2}<\beta<1$.  Let $m\in\mathbb{N}$. Assume conditions in Proposition~\ref{prop: overview trimming} are satisfied. Then for any $l\leq m$ and $n>n_0$ for some constant $n_0\in\mathbb{N}$ we have 
    \begin{align}
      \left\|Q_{n+ln^\beta}F_{n+2mn^\beta}^{l}P_n\right\|_1\leq & C_m e^{-d'n^{\beta-\frac{\alpha}{2}}} ,\label{eq:ess_banded}\\
      \left\|Q_{n+ln^\beta}F_{n\pm2mn^\beta}^{l}P_n\right\|_1\leq & C_m e^{-d'n^{\beta-\frac{\alpha}{2}}} ,\label{eq:ess_bandedPM}\\
      \left\| \left( n^{\alpha l}\left(f_{\alpha,x_0}^{(n)}\left( \mathcal{J} \right)\right)^{(l)}-F_{n+2mn^\beta}
      ^{l} \right)P_n\right\|_1\leq & C_m e^{-d'n^{\beta-\frac{\alpha}{2}}}, \label{eq:cutoff 1}
    \end{align}
    where $C_m>0$ is a constant that depends on $m$ and  $d'>0$ is also an universal constant.  
  \end{lemma}	
  
  \begin{proof}To shorten the notation in the proof, we take $$N=n+2mn^\beta. $$ 
    
    Recall the trace norm inequality $\|A\|_1\leq \sum_{k_1,k_2}|(A)_{k_1,k_2}|$ for any matrix $A$.  By Proposition~\ref{thm: overview F entry estimation}, we have
    \begin{equation}
      \left\| Q_{n+n^\beta}F_NP_n\right\|_1\leq  \sum_{j=n+n^\beta+1}^N\sum_{k=1}^n C_0n^{1+\alpha}e^{-d_0n^{\beta-\frac{\alpha}{2}}} = C_0(2m-1)n^{2+\alpha+\beta}e^{-d_0n^{\beta-\frac{\alpha}{2}}}, \label{eq:indunction l is 0}
    \end{equation}
    similarly, 
    \begin{multline}\label{eq:induction stepl+1}
      \left\| Q_{n+(l+1)n^\beta}F_NP_{n+ln^\beta}\right\|_1\leq  \sum_{j=n+(l+1)n^\beta+1}^N\sum_{k=1}^{n+ln^\beta}C_0n^{1+\alpha}e^{-d_0n^{\beta-\frac{\alpha}{2}}} 
      =C_0(2m-l-1)n^{1+\alpha+\beta}(n+ln^\beta)e^{-d_0n^{\beta-\frac{\alpha}{2}}} .
    \end{multline}
    We will prove the following statement for all $l\leq m$ by induction:
    \begin{align}
      \left\|Q_{n+ln^\beta}F_N^lP_n\right\|_1\leq & C_0(2m-1)ln^{1+\alpha+\beta}(n+mn^{\beta})\|F_N\|_\infty^{l-1}e^{-d_0n^{\beta-\frac{\alpha}{2}}} . \label{eq:ess_banded precise}
    \end{align}
    For $l=1$ the statement follows by \eqref{eq:indunction l is 0}. Then write $F_N^{l+1}=F_N(P_{n+ln^\beta}+Q_{n+ln^\beta})F_N^{l}$. With the triangle inequality and trace norm inequality $\|AB\|_1\leq \|A\|_1\|B\|_\infty$, we get
    \begin{equation}
      \left\| Q_{n+(l+1)n^\beta}F_N^{l+1}P_n\right\|_1\leq \left\| Q_{n+(l+1)n^\beta} F_NP_{n+l n^\beta}\right\|_1\|F_N\|^l_\infty+\|F_N\|_\infty \left\| Q_{n+ln^\beta}F_N^{l}P_n\right\|_1.
    \end{equation}
    Then apply \eqref{eq:induction stepl+1} to the first trace norm and the induction hypothesis \eqref{eq:ess_banded precise} for the second trace norm above, and we obtain
    \begin{equation}
      \left\| Q_{n+(l+1)n^\beta}F_N^{l+1}P_n\right\|_1	\leq C_0(2m-1)(l+1)n^{1+\alpha+\beta}(n+mn^\beta)\|F_N\|_\infty^{l}  e^{-d_0n^{\beta-\frac{\alpha}{2}}}.
    \end{equation}
    This concludes the induction.

 Moreover, since $\mathcal{J}$ is real symmetric and $\Im(\eta)/n^{\alpha}\neq 0$,  we have $\left\|\left( J_N^{(r)} \right)^{-1}\right\|_\infty\leq \frac{n^{\alpha}}{|\Im(\eta_r)|}$. Hence $\|F_N\|_\infty\leq \sum_r\left| \frac{c_r}{\Im(\eta_r)} \right| n^{\alpha}$. By inserting this in \eqref{eq:ess_banded precise}, we see \eqref{eq:ess_banded} holds.
    
    The second statement \eqref{eq:ess_banded} follows in the same way, since Proposition~\ref{thm: overview F entry estimation} also holds for $F_{n\pm 2mn^\beta}$.

 Now we prove the third statement \eqref{eq:cutoff 1}. By Lemma~\ref{lemma:boundary} we have, for $l=1,2,\dots$, 
     \begin{equation}\label{eq:FtolTruncationProof}
         \left\|n^{\alpha l}P_n\left(f_{\alpha,x_0}^{(n)}\left( \mathcal{J} \right)\right)^{(l)} -  P_nF_N^{l}\right\|_1\leq \left|a_{N,n}\right|\sum_{j=0}^{l-1}n^{\alpha(l-1-j)}\sum_{r=1}^M \left\|P_nF_N^j\left(J_N^{(r)}\right)^{-1}E_N\right\|_1\left\|B_{r,n}^{(l-1-j)}\right\|_\infty.
    \end{equation}
    where $E_N$ is a semi-infinite matrix with only one entry non-zero $(E_N)_{i,j} \coloneqq \mathrm{1}_{i=N, j=N+1} $. We further estimate the trace norm in \eqref{eq:FtolTruncationProof} to be 
    \begin{multline}\label{eq:FtolTruncationProof2}
     \left\|P_nF_N^j\left(J_N^{(r)}\right)^{-1}E_N\right\|_1\leq \left\|P_nF_N^jP_{n+ln^\beta}\left(J_N^{(r)}\right)^{-1}E_N\right\|_1 + \left\|P_nF_N^jQ_{n+ln^\beta}\left(J_N^{(r)}\right)^{-1}E_N\right\|_1 \\
     \leq \left\|F_N\right\|_\infty^j\left\|P_{n+ln^\beta}\left(J_N^{(r)}\right)^{-1}E_N\right\|_1 + \left\|P_nF_N^jQ_{n+ln^\beta}\right\|_1\left\|\left(J_N^{(r)}\right)^{-1}\right\|_\infty.
    \end{multline}
    By Proposition~\ref{thm: overview F entry estimation} and trace norm estimate, we have the first trace norm in \eqref{eq:FtolTruncationProof2} to be$\left\|P_{n+ln^\beta}\left(J_N^{(r)}\right)^{-1}E_N\right\|_1  \leq C_0(n+ln^\beta)n^\alpha e^{-d_0n^{\beta-\frac{\alpha}{2}}}$. We also estimate the second trace norm in \eqref{eq:FtolTruncationProof2} to be exponentially small by \eqref{eq:ess_banded precise}.  Note that the operator norms in \eqref{eq:FtolTruncationProof2} are of order $n^{\alpha}$. Hence, the lemma is concluded.
  \end{proof}

    \begin{lemma}\label{cutoff 2}
    Let $0<\frac{\alpha}{2}<\beta<1$.  Let $m\in\mathbb{N}$.   Then for any $l\leq m$ we have 
    \begin{align}
      \left\|P_{n-ln^\beta}F_{n\pm 2mn^\beta}^{l}Q_{n}\right\|_1 \leq & C_m e^{-d'n^{\beta-\frac{\alpha}{2}}}  \label{eq:ess_banded2} \\
      \left\|\left( F_{n+2mn^\beta}^{l}-F_{n\pm 2mn^\beta}^{l} \right)Q_{n-mn^\beta}\right\|_1 \leq & C_m e^{-d'n^{\beta-\frac{\alpha}{2}}}  \label{eq:cutoff2} 
    \end{align}
    where $a_{n-2mn^\beta}$ is the $n-2mn^\beta+1,n-2mn^\beta$-th entry of $\mathcal{J}$,  $C_m>0$ is a constant that depends on $m$ and  $d'>0$ is also a universal constant.  
    
  \end{lemma}
  
  \begin{proof}
    
    This proof follows the same recipe as the one of Lemma~\ref{cutoff}. Using the trace norm inequality $\|A\|_1\leq \sum_{j,k}|(A)_{j,k}|$ for any matrix $A$ and applying the second statement of Proposition~\ref{thm: overview F entry estimation}, we obtain
    
    \begin{align}
      \left\| P_{n-n^\beta} F_{n\pm 2mn^\beta}Q_{n}\right\|_1\leq  \sum_{j=n-2mn^\beta+1}^{n-(1+m)n^\beta}\sum_{k=n-mn^\beta+1}^{n+2mn^\beta}C_0n^{1+\alpha}e^{-d_0n^{\beta-\frac{\alpha}{2}}}= C_0(2m-1)2mn^{1+\alpha+2\beta}e^{-d_0n^{\beta-\frac{\alpha}{2}}}, \label{eq:prop2 estimate1}
    \end{align}
    similarly, 
    \begin{multline}\label{eq:induction stepl+l 2}
      \left\| P_{n-(l+1)n^\beta} F_{n\pm 2mn^\beta}Q_{n-ln^\beta}\right\|_1\leq  \sum_{j=n-2mn^\beta+1}^{n-(l+1)n^\beta}\sum_{k=n-ln^\beta+1}^{n+2mn^\beta}C_0n^{1+\alpha}e^{-d_0n^{\beta-\frac{\alpha}{2}}} \\
      =C_0(2m-l-1)(2m+l)n^{1+\alpha+2\beta}e^{-d_0n^{\beta-\frac{\alpha}{2}}}.
    \end{multline}
    
    We will prove the following statement for all $l\leq m$ by induction:
    \begin{equation}
      \|P_{n-ln^\beta}F_{n\pm 2mn^\beta}^lQ_{n}\|_1 \leq  3C_0(2m-1)mn^{1+\alpha+2\beta}l\|F_{n\pm mn^\beta}\|_\infty^{l-1} e^{-d_0n^{\beta-\frac{\alpha}{2}}}. \label{eq:ess_banded2 precise} 
    \end{equation}

    For $l=1$ the statement follows by  \eqref{eq:prop2 estimate1}. Then write $F_{n\pm 2mn^\beta}^{l+1}=F_{n\pm 2mn^\beta}(P_{n+ln^\beta}+Q_{n+ln^\beta})F_{n\pm 2mn^\beta}^{l}$. With the triangle inequality and trace norm inequality $\|AB\|_1\leq \|A\|_1\|B\|_\infty$, we get
    \begin{multline}
      \left\| P_{n-(l+1)n^\beta} F_{n\pm 2mn^\beta}^{l+1}Q_{n}\right\|_1 \\ 
      \leq \left\| P_{n-(l+1)n^\beta}F_{n\pm 2mn^\beta}Q_{n-ln^\beta}\right\|_1\|F_{n\pm 2mn^\beta}\|^l_\infty  
      +\|F_{n\pm 2mn^\beta}\|_\infty \left\|P_{n-ln^\beta} F_{n\pm 2mn^\beta}Q_{n}\right\|_1 
    \end{multline}
    Then apply \eqref{eq:induction stepl+l 2} to the first trace norm and the induction hypothesis \eqref{eq:ess_banded2 precise} for the second trace norm above, and we obtain
    \begin{equation}
      \left\| P_{n-(l+1)n^\beta} F_{n\pm 2mn^\beta}^{l+1}Q_{n}\right\|_1	\leq 3C_0(2m-1)mn^{1+\alpha+2\beta}(l+1)\|F_{n\pm mn^\beta}\|_\infty^{l} e^{-d_0n^{\beta-\frac{\alpha}{2}}}.
    \end{equation}
    This concludes the induction.
    
    For the second statement, to shorten the notation, let us denote 
    $$\tilde{N}= n-2mn^\beta.$$
    Let
    \begin{align}
      E \coloneqq J_{n+2mn^\beta}^{(r)}-P_{\tilde{N}}J_{n+2mn^\beta}^{(r)}P_{\tilde{N}}-Q_{\tilde{N}}J_{n+2mn^\beta}^{(r)}Q_{\tilde{N}}.
    \end{align}
    Note that $E$ is independent of $r$. It is because that $E$ is a semi-infinite matrix with only two entries non zero, i.e. $(E)_{\tilde{N}+1,\tilde{N}}=(E)_{\tilde{N},\tilde{N}+1} = a_{\tilde{N}}$. This also implies that $E=EP_{\tilde{N}+1}Q_{\tilde{N}-1}$. Moreover,  $ \left( J_{n\pm 2mn^\beta}^{(r)} \right)^{-l}Q_{n-mn^\beta}=\left(J_{n+2mn^\beta}^{(r)}-E\right)^{-l}Q_{n-mn^\beta}$. 
    
    Hence 
    \begin{equation*}
      F_{n+2mn^\beta}-F_{n\pm 2mn^\beta}=\sum_r c_r\left( J_{n+2mn^\beta}^{(r)} \right)^{-1} \left(-EP_{\tilde{N}+1}Q_{\tilde{N}-1}-P_{\tilde{N}}J_{n+2mn^\beta}^{(r)}P_{\tilde{N}} \right)\left( J_{n\pm2mn^\beta}^{(r)} \right)^{-1} .
    \end{equation*}
    
    Further by telescopic sum, we have 
    \begin{multline*}
      \left(F_{n+2mn^\beta}^l-F_{n\pm2mn^\beta}
      ^l \right)Q_{n-mn^\beta} 
      =-\sum_{j=0}^{l-1}\sum_rF_{n+2mn^\beta}^{j}  c_r\left( J_{n+2mn^\beta}^{(r)} \right)^{-1} EP_{\tilde{N}+1}Q_{\tilde{N}-1}\left( J_{n\pm2mn^\beta}^{(r)} \right)^{-1} F_{n\pm2mn^\beta}
      ^{l-j-1}Q_{n-mn^\beta},
    \end{multline*}
    which follows from $P_{\tilde{N}} \left( J_{n\pm2mn^\beta}^{(r)} \right)^{-1}F_{n\pm2mn^\beta}
    ^{l-j-1}Q_{n-mn^\beta}=0$.
    
    By triangle inequality and trace norm inequality, we have 
    \begin{multline}
      \left\| 	\left(F_{n+2mn^\beta}^l-F_{n\pm2mn^\beta}
      ^l \right)Q_{n-mn^\beta}\right\|_1\\
      \leq \sum_{j=0}^{l-1}\sum_r|c_r|\|F_{n+2mn^\beta}\|_\infty^{j}  \left\|\left( J_{n+2mn^\beta}^{(r)} \right)^{-1}\right\|_\infty \|E\|_\infty\left\|P_{\tilde{N}+1}Q_{\tilde{N}-1}\left( J_{n\pm2mn^\beta}^{(r)} \right)^{-1} F_{n\pm2mn^\beta}
      ^{l-j-1}Q_{n-mn^\beta}\right\|_1 . \label{eq:telescpoc 2}
    \end{multline}
    
    Write $\left( J_{n\pm2mn^\beta}^{(r)} \right)^{-1}=\left( J_{n\pm2mn^\beta}^{(r)} \right)^{-1}\left(P_{n-(l-j-1+m)n^\beta}+Q_{n-(l-j-1+m)n^\beta}\right)$, use triangle inequality and we have 
    \begin{multline*}
      \left\|P_{\tilde{N}+1}Q_{\tilde{N}-1}\left( J_{n\pm2mn^\beta}^{(r)} \right)^{-1} F_{n\pm2mn^\beta}
      ^{l-j-1}Q_{n-mn^\beta}\right\|_1 
      \leq \left\| P_{\tilde{N}+1}Q_{\tilde{N}-1}\left( J_{n\pm2mn^\beta}^{(r)} \right)^{-1}Q_{n-(l-j-1+m)n^\beta}\right\|_1  \left\|F_{n+2mn^\beta}\right\|_\infty^{l-j-1}+\tilde{R}_{l-j-1}
    \end{multline*}
    where $\tilde{R}_{l-j-1}=\left\|\left( J_{n\pm2mn^\beta}^{(r)} \right)^{-1} \right\|_\infty\left\|P_{n-(l-j-1+m)n^\beta}F_{n\pm2mn^\beta}^{l-j-1}Q_{n-mn^\beta}\right\|_1$. Note that $\tilde{R} _0=0$. For $j<l-1$ we can estimate $\tilde{R}_{l-j-1}$ via  \eqref{eq:ess_banded2 precise} to be 
    \begin{equation*}
      \tilde{R}_{l-j-1}\leq 4C_0(m-1)mn^{1+\alpha+2\beta}(l-j-1)\|J_{n\pm mn^\beta}^{-1}\|_\infty^{l-j-1}e^{-d_0n^{\beta-\frac{\alpha}{2}}}.
    \end{equation*}
    Further using the trace norm inequality $\|A\|_1\le \sum_{j,k}|(A)_{j,k}|$ for any matrix $A$ and estimating the entries via Proposition~\ref{thm: overview F entry estimation}, we have 
    \begin{multline*}
      \left\| P_{\tilde{N}+1}Q_{\tilde{N}-1}J_{n\pm2mn^\beta}^{-1}Q_{n-(l-j-1+m)n^\beta}\right\|_1  \leq  \sum_{k=n-(l-j-1+m)n^\beta+1}^{n+2mn^\beta} C_0n^{1+\alpha}n^{-d_0n^{\beta-\frac{\alpha}{2}}} 
      \leq C_0(3m+l)n^{1+\alpha+\beta}e^{-d_0n^{\beta-\frac{\alpha}{2}}}.
    \end{multline*}
    Assembling the above estimates into \eqref{eq:telescpoc 2}, we have for all $l\leq m$ 
    \begin{align*}
      \left\| \left( F_{n+2mn^\beta}^{l}-F_{n\pm2mn^\beta}
      ^{l} \right)Q_{n-mn^\beta}\right\|_1 \leq \|E\|_\infty C_m'e^{-d'n^{\beta-\frac{\alpha}{2}}}
    \end{align*}
    where $d'\in(0,d_0)$ and $C_m'>0$ some universal constant.  Note that  $\|E\|_\infty=|a_{\tilde{N}}|<c_1$. This concludes the lemma. 
  \end{proof}	

  \subsection{Proof of Proposition~\ref{prop: overview trimming}}\label{sec:proof overview trim}
  Recall that for any linear operator $\mathcal{A}$ we write by \eqref{eq:cumulant operator} for any $m\geq 2$
    \begin{equation*}
      \mathcal{C}_m^{(n)}(\mathcal{A}) = m!\sum_{j=2}^{m}\frac{(-1)^{j+1}}{j}\sum_{l_1+\dots +l_j=m, l_i\geq 1}\frac{\Tr(\mathcal{A})^{l_1}P_n\dots (\mathcal{A})^{l_j}P_n-\Tr(\mathcal{A}^mP_n)}{l_1!\dots l_j!}.
    \end{equation*}
We first explore some algebraic properties of the trace above. 
\begin{lemma}\label{lemma:traceAlgebras}
     Let $m\geq 2$, $j=2,\dots, m$ and $l_i\geq 1$ with $l_1+\dots +l_j=m$. Let $\mathcal{A}$ and $\mathcal{B}$ be bounded linear operators. Let $P_n$ be the projection operator onto the first $n$ coordinates and $Q_n = \mathrm{Id}-P_n$. Then
    \begin{multline}\label{eq:cumulant telescopic}
      \left( \Tr\left(\mathcal{A}^{l_1}P_n\cdots \mathcal{A}^{l_j}P_n\right)-\Tr\left(\mathcal{A}^mP_n\right) \right)-\left( \Tr\left(\mathcal{B}^{l_1}P_n\cdots \mathcal{B}^{l_j}P_n\right)-\Tr\left(\mathcal{B}^mP_n\right) \right) \\
      = - \sum_{k=2}^{j}\Bigg( \Tr\left((\mathcal{A}^{l_k}-\mathcal{B}^{l_k})P_n\mathcal{A}^{l_{k+1}}\cdots P_n\mathcal{A}^{l_j}P_n\mathcal{A}^{l_1+\cdots+l_{k-1}}Q_n\right)\\
      +\sum_{i=k}^{j-1}\Tr\left(\mathcal{B}^{l_k}P_n\cdots P_n\mathcal{B}^{l_{i}}P_n(\mathcal{A}^{l_{i+1}}-\mathcal{B}^{l_{i+1}})P_n\cdots \mathcal{A}^{l_{j}}P_n\mathcal{A}^{l_1+\cdots+l_{k-1}}Q_n\right)\\
      + \Tr\left(\mathcal{B}^{l_{k}}P_n\cdots P_n\mathcal{B}^{l_j}P_n(\mathcal{A}^{l_1+\cdots+l_{k-1}}-\mathcal{B}^{l_1+\cdots+l_{k-1}})Q_n\right) \Bigg),
    \end{multline}
    here we take the convention that $\sum_{j=a}^b\equiv 0$ for any integers $b<a$.
     
\end{lemma}
\begin{proof}
    We write 
    \begin{equation*}
      \mathcal{A}^mP_n=\mathcal{A}^{l_1}(P_n+Q_n)\mathcal{A}^{l_2}(P_n+Q_n)\cdots \mathcal{A}^{l_j}P_n.
    \end{equation*} 
    By expanding the formula above, we find
    \begin{multline}\label{eq:telescopic sum 1}
      \mathcal{A}^{l_1}P_n\cdots \mathcal{A}^{l_j}P_n-\mathcal{A}^mP_n \\
      = -	\mathcal{A}^{l_1}Q_n\mathcal{A}^{l_2}P_n\cdots \mathcal{A}^{l_j}P_n-	\mathcal{A}^{l_1+l_2}Q_n\mathcal{A}^{l_3}P_n\cdots \mathcal{A}^{l_j}P_n-\cdots -	\mathcal{A}^{l_1+\cdots+l_{j-1}}Q_n\mathcal{A}^{l_j}P_n.
    \end{multline}
    Using the cyclic property of the trace, we obtain \begin{equation}\label{eq:cumulant telescopicPn}
      \Tr\left(\mathcal{A}^{l_1}P_n\cdots \mathcal{A}^{l_j}P_n\right)-\Tr\left(\mathcal{A}^mP_n\right)	= - \sum_{k=2}^j\Tr\left(\mathcal{A}^{l_{k}}P_n\cdots P_n\mathcal{A}^{l_j}P_n\mathcal{A}^{l_1+\cdots+l_{k-1}}Q_n\right) . 
    \end{equation}
    
    Then by \eqref{eq:cumulant telescopicPn}, we further write 
    \begin{multline}
      \left( \Tr\left(\mathcal{A}^{l_1}P_n\cdots \mathcal{A}^{l_j}P_n\right)-\Tr\left(\mathcal{A}^mP_n\right) \right)-\left( \Tr\left(\mathcal{B}^{l_1}P_n\cdots \mathcal{B}^{l_j}P_n\right)-\Tr\left(\mathcal{B}^mP_n\right) \right) \\ 
      =  \sum_{k=2}^j-Tr\left(\mathcal{A}^{l_{k}}P_n\cdots P_n\mathcal{A}^{l_j}P_n\mathcal{A}^{l_1+\cdots+l_{k-1}}Q_n\right)+ \Tr\left(\mathcal{B}^{l_{k}}P_n\cdots P_n\mathcal{B}^{l_j}P_n\mathcal{B}^{l_1+\cdots+l_{k-1}}Q_n\right)
    \end{multline}

    Similarly, using the telescoping sum and cyclic property of the trace operator, we conclude \eqref{eq:cumulant telescopic}.
\end{proof}

\begin{lemma}\label{lemma:traceCompare}
    Let $m\geq 2$, $j=2,\dots, m$ and $l_i\geq 1$ with $l_1+\dots +l_j=m$. Take $\mathcal{A} = F_{n\pm 2mn^\beta}$ and $\mathcal{B}$  be any bounded operator. Assume there exists a constant $C_{op}>1$ such that 
    \begin{equation}
        \max\{\|\mathcal{A}\|_\infty,\|\mathcal{B}\|_\infty\}\leq n^\alpha C_{op}.
    \end{equation}
    Also assume for any $l\leq m$
    \begin{equation}\label{eq:ass:RightCornerDecay}
        \left\|P_n\mathcal{B}^{l}Q_{n+ln^\beta}\right\|_1\leq C_m e^{-d'n^{\beta-\frac{\alpha}{2}}}. 
    \end{equation}
    Define 
\begin{equation}
        R_n\coloneqq 2mC_m \left(n^\alpha C_{op} \right)^{m}e^{-d'n^{\beta-\frac{\alpha}{2}}},
      \end{equation}
      where $C_m$ is the same constant in the upper bound in Lemmas~\ref{cutoff} and~\ref{cutoff 2}.
      Let $P_n$ be the projection operator onto the first $n$ coordinates and $Q_n = \mathrm{Id}-P_n$. We have
    \begin{multline}\label{eq:cumulant telescopic bound}
      \left|\left( \Tr\left(\mathcal{A}^{l_1}P_n\cdots \mathcal{A}^{l_j}P_n\right)-\Tr\left(\mathcal{A}^mP_n\right) \right)-\left( \Tr\left(\mathcal{B}^{l_1}P_n\cdots \mathcal{B}^{l_j}P_n\right)-\Tr\left(\mathcal{B}^mP_n\right) \right)\right|\\
      \leq (m+1)\left(\left(n^\alpha C_{op}\right)^{m+1} \left\|(\mathcal{A}-\mathcal{B})P_{n+mn^\beta}Q_{n-mn^\beta}\right\|_1 +R_n\right).
    \end{multline}
\end{lemma}
\begin{proof}
    Apply \eqref{eq:ess_bandedPM} in Lemma~\ref{cutoff}, and apply the trace norm inequality, $\|AB\|_1\leq\|A\|_\infty\|B\|_1$, and the fact that, for $i<m$ to obtain
      \begin{equation}\label{eq:step3 approx 10}
        \|Q_{n+mn^\beta}\mathcal{A}^{i}P_n\|_1=\|Q_{n+mn^\beta}Q_{n+in^\beta}\mathcal{A}^{i}P_n\|_1\leq \|Q_{n+mn^\beta}\|_\infty\|Q_{n+in^\beta}\mathcal{A}^{i}P_n\|_1\leq C_me^{-d'n^{\beta-\frac{\alpha}{2}}}.
      \end{equation}
      Use the fact $\mathcal{A}^{l_k}-\mathcal{B}^{l_k}=\sum_{i=0}^{l_k-1}\mathcal{B}^{l_k-1-i}(\mathcal{A}-\mathcal{B})\mathcal{A}^{i}$, together with \eqref{eq:step3 approx 10}, to obtain
      \begin{multline}\label{eq:step 3 approx 1}
        \left\| \left( \mathcal{A}^{l_k}-\mathcal{B}^{l_k} \right)P_n -\sum_{i=0}^{l_k-1}\mathcal{B}^{l_k-1-i}(\mathcal{A}-\mathcal{B})P_{n+mn^\beta}\mathcal{A}^{i}P_n\right\|_1
        =\left\|\sum_{i=0}^{l_k-1}\mathcal{B}^{l_k-1-i}(\mathcal{A}-\mathcal{B})Q_{n+mn^\beta}\mathcal{A}^{i}P_n\right\|_1\\
        \leq  \sum_{i=0}^{l_k-1}	\left\|\mathcal{B}^{l_k-1-i}(\mathcal{A}-\mathcal{B})\right\|_\infty\|Q_{n+mn^\beta}\mathcal{A}^{i}P_n\|_1\leq 2l_k\left( n^\alpha C_{op}\right)^{l_k-i} C_me^{-d'n^{\beta-\frac{\alpha}{2}}}.
      \end{multline}
      To shorten the notation for the rest of the proof, we define 
    \begin{equation}
    \mathcal{V}^{[k]}\coloneqq \mathcal{A}^{l_1+\cdots+l_{k}}Q_n, \qquad \mathcal{W}_{\mathcal{A}}^{[k,j]}\coloneqq \mathcal{A}^{l_{k}}P_n\cdots P_n\mathcal{A}^{l_j}P_n, \text{ for $k\leq j$,} \quad \mathcal{W}_{\mathcal{A}}^{[k,j]} \equiv 1 \text{ for $k>j$}. \label{eq:def:traceCompress}
    \end{equation}
    Then, for $k=2,\dots, j-1$, use \eqref{eq:step 3 approx 1} to estimate the first summand in \eqref{eq:cumulant telescopic} to be
      \begin{equation}\label{eq:step 3 0}
        \left| \Tr\left((\mathcal{A}^{l_k}-\mathcal{B}^{l_k})P_n\mathcal{W}_{\mathcal{A}}^{[k+1,j]}\mathcal{V}^{[k-1]}\right) \right|         \leq 	\left\| \sum_{i=0}^{l_k-1}\mathcal{B}^{l_k-1-i}(\mathcal{A}-\mathcal{B})P_{n+mn^\beta}\mathcal{A}^{i}P_n\mathcal{W}_{\mathcal{A}}^{[k+1,j]}\mathcal{V}^{[k-1]}\right\|_1+R_n.
      \end{equation}
      We write $P_n=\mathrm{Id}-Q_n$, use the telescopic sum, and get 
      \begin{equation}\label{eq:step 3 11}
        \mathcal{A}^{i}P_n\mathcal{W}_{\mathcal{A}}^{[k+1,j]}\mathcal{V}^{[k-1]} 
      =  -\mathcal{A}^{i}Q_n\mathcal{W}_{\mathcal{A}}^{[k+1,j]}\mathcal{V}^{[k-1]} -\sum_{p=k+1}^{j}\mathcal{A}^{i+l_{k+1}+\cdots+ l_{p}}Q_n\mathcal{W}_{\mathcal{A}}^{[p+1,j]}\mathcal{V}^{[k-1]} + \mathcal{A}^{i+m-l_k}Q_n. 
      \end{equation}
      Note that there are $j-k+2$ terms on the right-hand side of \eqref{eq:step 3 11} and $j-k+2\leq m$. 
      
      By \eqref{eq:ess_banded2} in Lemma~\ref{cutoff 2} we have, for any $l\in\mathbb{N}$ with $l\leq m$, 
      \begin{equation}\label{eq:step 3 12}
        \|\mathcal{A}^{l}Q_n-Q_{n-mn^\beta}\mathcal{A}^{l}Q_n\|_1=\|P_{n-mn^\beta}P_{n-ln^\beta}\mathcal{A}^{l}Q_n\|_1\leq C_me^{-d'n^{\beta-\frac{\alpha}{2}}}. 
      \end{equation} 
      Hence, estimate \eqref{eq:step 3 11} by \eqref{eq:step 3 12} to obtain
      \begin{equation}\label{eq:step 3 111}
      \left\| 	\mathcal{A}^{i}P_n\mathcal{W}_{\mathcal{A}}^{[k+1,j]}\mathcal{V}^{[k-1]}-	Q_{n-mn^\beta}\mathcal{A}^{i}P_n\mathcal{W}_{\mathcal{A}}^{[k+1,j]}\mathcal{V}^{[k-1]}  \right\|_1
      \leq (j-k+2)C_m \left( n^\alpha C_{op} \right)^{m-l_k-i}e^{-d'n^{\beta-\frac{\alpha}{2}}}.
      \end{equation}
      Plug \eqref{eq:step 3 111} into the right-hand side of \eqref{eq:step 3 0} to obtain
      \begin{equation}\label{eq:step 3 13}
        \left\| \mathcal{B}^{l_k-1-i}(\mathcal{A}-\mathcal{B})P_{n+mn^\beta}\mathcal{A}^{i}P_n\mathcal{W}_{\mathcal{A}}^{[k+1,j]}\mathcal{V}^{[k-1]}\right\|_1
        \leq \left(n^\alpha C_{op}\right)^{m-1}\|(\mathcal{A}-\mathcal{B})P_{n+mn^\beta}Q_{n-mn^\beta}\|_1 +R_n .
      \end{equation}
      Plug \eqref{eq:step 3 13} into the right-hand side of \eqref{eq:step 3 0}, and we get
\begin{equation}\label{eq:step 3 1}
        \left| \Tr\left((\mathcal{A}^{l_k}-\mathcal{B}^{l_k})P_n\mathcal{W}_{\mathcal{A}}^{[k+1,j]}\mathcal{V}^{[k-1]}\right) \right| 
        \leq l_k(n^\alpha C_{op})^{m-1}\|(\mathcal{A}-\mathcal{B})P_{n+mn^\beta}Q_{n-mn^\beta}\|_1 +(l_k+1)R_n.
      \end{equation}
      Now we estimate the second summand in \eqref{eq:cumulant telescopic}. Define $\mathcal{W}_{\mathcal{B}}^{[k,j]}$ analogously as \eqref{eq:def:traceCompress}. For $i=k,\dots, j-1$,
      \begin{equation}\label{eq:step 3 2}
        \left|  \Tr\left(\mathcal{W}_{\mathcal{B}}^{[k,i]}(\mathcal{A}^{l_{i+1}}-\mathcal{B}^{l_{i+1}})P_n\mathcal{W}_{\mathcal{A}}^{[k+1,j]}\mathcal{V}^{[k-1]}\right)\right|
        \leq l_{i+1}(n^\alpha C_{op})^{m-1}\|(\mathcal{A}-\mathcal{B})P_{n+mn^\beta}Q_{n-mn^\beta}\|_1 +(l_{i+1}+1)R_n.
      \end{equation}
      
      Lastly, we estimate the last summand in \eqref{eq:cumulant telescopic}. 
     
      By \eqref{eq:ass:RightCornerDecay} we have
      \begin{equation}
        \left|  \Tr\left(\mathcal{W}_{\mathcal{B}}^{[k,j]}(\mathcal{A}^{l_1+\cdots+l_{k-1}}-\mathcal{B}^{l_1+\cdots+l_{k-1}})Q_n\right) \right|
        \leq \left|  \Tr\left(\mathcal{W}_{\mathcal{B}}^{[k,j]}(\mathcal{A}^{l_1+\cdots+l_{k-1}}-\mathcal{B}^{l_1+\cdots+l_{k-1}})P_{n+mn^\beta}Q_n\right) \right| +R_n.
      \end{equation}
      Repeat the argument similar to \eqref{eq:step 3 1} to obtain
      \begin{multline}\label{eq:step 3 3}
        \left|  \Tr\left(\mathcal{W}_{\mathcal{B}}^{[k,j]}(\mathcal{A}^{l_1+\cdots+l_{k-1}}-\mathcal{B}^{l_1+\cdots+l_{k-1}})Q_n\right) \right|\\
        \leq (l_1+\cdots+l_{k-1}) (n^\alpha C_{op})^{m-1}\|(\mathcal{A}-\mathcal{B})P_{n+mn^\beta}Q_{n-mn^\beta}\|_1 +((l_1+\cdots+l_{k-1})+1)R_n.
      \end{multline}
      
      Hence, plugging three estimates \eqref{eq:step 3 1}, \eqref{eq:step 3 2} and \eqref{eq:step 3 3} into the formula \eqref{eq:cumulant telescopic}, we conclude \eqref{eq:cumulant telescopic bound}. 
\end{proof}

  Now we are ready to prove the Proposition~\ref{prop: overview trimming} by Lemmas~\ref{cutoff} and~\ref{cutoff 2}.
  
  \begin{proof}[Proof of Proposition~\ref{prop: overview trimming}]
    \subsubsection*{Step 1 Consider $(F)^{(l)}\coloneqq n^{\alpha l}\left(f_{\alpha,x_0}^{(n)}\left( \mathcal{J} \right)\right)^{(l)}$ as in \eqref{eq:mesoscopic ftol} and $\mathcal{B}=F_{n+2mn^\beta}$.}
   By Lemma~\ref{cutoff} we have for any $l\leq m$ we have $\left\| \left( (F)^{(l)}-\mathcal{B}
    ^{l} \right)P_n\right\|_1\leq C_m e^{-d'n^{\beta-\frac{\alpha}{2}}}$. Note that we have $\|(F)^{(l)}\|_\infty,\|\mathcal{B}
    ^{l} \|_\infty\leq\left(\sum_r \left| \frac{c_r}{\Im(\eta_r)} \right|n^\alpha\right)^l$. Then use the telescopic sum and the trace norm inequality to obtain
    \begin{multline}\label{eq:cumulant_comparison 1}     \left|\Tr(F)^{(l_1)}P_n\cdots (F)^{(l_j)}P_n-\Tr((F)^{(m)}P_n)-\Tr(\mathcal{B})^{l_1}P_n\cdots (\mathcal{B})^{l_j}P_n+\Tr(\mathcal{B}^mP_n)\right|\\
    \leq \sum_{k=1}^{j} \left\| \left( (F)^{(l_k)}-\mathcal{B}
    ^{l_k} \right)P_n\right\|_1\left(\sum_r \left| \frac{c_r}{\Im(\eta_r)} \right|n^\alpha\right)^{m-l_k} + \left\| \left( (F)^{(m)}-\mathcal{B}
    ^{m} \right)P_n\right\|_1
    \leq C'' e^{-d''n^{\beta-\frac{\alpha}{2}}},  
    \end{multline}
    for some constant $C''>0, d''>0$. Hence plug the estimate \eqref{eq:cumulant_comparison 1} into the  cumulant formula \eqref{eq:cumulant operator} to get
    \begin{equation}\label{eq:step 1 result}
      \left| 	n^{\alpha m}\mathcal{C}_m(X^{(n)}_{f,\alpha,x_0})-	\mathcal{C}_m^{(n)}(F_{n+2mn^\beta})  \right| \leq C_m''e^{-d''n^{\beta-\frac{\alpha}{2}}}.
    \end{equation}
    
    \subsubsection*{Step 2 Consider $\mathcal{A}=F_{n\pm2mn^\beta}$ and $\mathcal{B}=F_{n+2mn^\beta}$}
    Then by Lemma~\ref{lemma:traceCompare} and \eqref{eq:cutoff2}, we have
    \begin{align}
      \left|\Tr(\mathcal{A})^{l_1}P_n\cdots (\mathcal{A})^{l_j}P_n-\Tr(\mathcal{A}^mP_n)-\Tr(\mathcal{B})^{l_1}P_n\cdots (\mathcal{B})^{l_j}P_n+\Tr(\mathcal{B}^mP_n)\right|\leq C'' e^{-d''n^{\beta-\frac{\alpha}{2}}},  \label{eq:cumulant_comparison}
    \end{align}
    for some constant $C''>0, d''>0$. Hence plug the estimate \eqref{eq:cumulant_comparison} into the definition \eqref{eq:cumulant operator} and we have 
    \begin{equation}\label{eq:step 2 result}
      \left| 	\mathcal{C}_m^{(n)}(F_{n+2mn^\beta}) -	\mathcal{C}_m^{(n)}(F_{n\pm2mn^\beta})  \right| \leq C_m''e^{-d''n^{\beta-\frac{\alpha}{2}}}.
    \end{equation}
   Combining \eqref{eq:step 1 result} from Step~$1$ and \eqref{eq:step 2 result} from Step~$2$,  we conclude \eqref{eq:overview cutoff} in Proposition~\ref{prop: overview trimming}.
  \end{proof}

    \section{Proof of Propositions~\ref{prop: overview tri-inverse} and~\ref{prop: overview trimming 2}}\label{sec:cumulant truncation}
    In this section, we  prove Propositions~\ref{prop: overview tri-inverse} and~\ref{prop: overview trimming 2}. Essentially, we are studying the inverse of the blocked operator $J_{n\pm 2mn^\beta}^{(r)}$ defined in \eqref{eq:truncation 2}, with $\beta=\frac{\alpha}{2}+\frac{\varepsilon}{3}$.  It is the middle block, 
    $P_{n+2mn^\beta}Q_{n-2mn^\beta}J^{(r)}Q_{n-2mn^\beta}P_{n+2mn^\beta}$,
    that is non-trivial and matters in the analysis. 
    
    In Section~\ref{sec:Inverse of A Tri-diagonal Matrix: Part II}, we will compute the resolvent for a general $N\times N$ matrix with $N\in\mathbb{N}$, (cf.Proposition~\ref{prop: tri-inverse}). Later,  in Sections~\ref{sec:proof tri-inverse} and~\ref{sec:proof overview trim 2}, we will apply Proposition\ref{prop: tri-inverse}  to the middle block the resolvent of $J_{n\pm2mn^\beta}^{(r)}$, by taking the size of the matrix to be $4mn^\beta$ with relabelling of the indexes.

    \subsection{Inverse of a Tri-diagonal Matrix: Part II} \label{sec:Inverse of A Tri-diagonal Matrix: Part II}
    Consider an $N\times N$ non singular symmetric matrix with $a_j\in\mathbb{R}, b_j\in\mathbb{R}$, $N\in\mathbb{N}$ and $z\in\mathbb{C}$ with $\Im z\neq 0$.
    \begin{align} \label{eq:assumption_J_inverse copy}
      J_{N}= \begin{pmatrix}b_{0} & a_{1} &  &  & &  & \\a_{1} & b_{1} & a_{2} &  &  &  & \\& a_{2} & b_{2} & a_{3} &  & & \\ & & \ddots & \ddots & \ddots \\ & & & a_{N-2} & b_{N-2} & a_{N-1} \\ & & & & a_{N-1} & b_{N-1}  \end{pmatrix} -z \mathrm{Id}
    \end{align}
    
    The goal of this subsection is to continue to develop a theory to estimate the inverse of $J_N$ entry-wise by imposing stronger assumptions than in Section~\ref{sec:improve CB}.

\begin{proposition}[Almost Toeplitz]\label{prop: tri-inverse}
      Consider an $N\times N$ non singular symmetric matrix $J_N$ with $a_j\in\mathbb{R}, b_j-\Re z>0$ and $N\in\mathbb{N}$ as defined as  \eqref{eq:assumption_J_inverse copy}. Let $\omega_j^+, \omega_j^-$ as defined as  \eqref{eq:pre eigenvalue_difference_eq}. Let $M_j$ as defined as  \eqref{eq:eigenvalue_difference_eq_M_j}. Define
      \begin{equation}
        c_0=\min_j|a_j|, \quad c_1=\max\left\{ \max_j|a_j|,\max_j\left|b_{j-1}-z\right| \right\} \label{ass:recurrence bound} ,
      \end{equation}
      \begin{equation}
        c_2=\max\left\{1, \left|\frac{-a_{1}+\omega_2^-(b_{0}-z)}{a_{1}-\omega_2^+(b_{0}-z)}\right|, \left|\frac{-\omega_{N-1}^+a_{N-1}+b_{N-1}-z}{\omega_{N-1}^-a_{N-1}-b_{N-1}+z}\right|   \right\}\label{eq:assumption_betatilde} ,
      \end{equation}
      \begin{equation}
        \varepsilon_1=N \max_{j}\|M_j\|_\infty \label{eq:assuption_tech2},
      \end{equation}
      \begin{equation}
        \varepsilon_2=\prod_{l=2}^{N-1}\left| \frac{\omega_l^-}{\omega_l^+} \right|.
      \end{equation}
      Assume that $c_0>0$,  and both $\varepsilon_1, \varepsilon_2$ are sufficiently small. \footnote{ More precisely, we assume $0<\varepsilon_1< 3^{-1}\left( \frac{(1+\sqrt{5})c_1}{2c_0}\right)^{-2},0<\varepsilon_2$ and 
        $12\left(1+\left( \frac{(1+\sqrt{5})c_1}{2c_0}\right)^2 c_2^2\right)\varepsilon_1+ c_2\varepsilon_2<\left( \frac{(1+\sqrt{5})c_1}{2c_0}\right)^{-2}$.}
      Then we have
      \begin{align}
        J_N^{-1} = T_N(\eta) + H_N(\eta),
      \end{align} 
      such that  for all $ j\leq k$
      \begin{align}
        \left| (T_N(\eta))_{j,k}   \right|
        \leq & Constant \left| \frac{\prod_{l=j}^{k-1}\omega_l^-}{a_{k-1}\omega_j^-\left( \omega_{N-1}^--\omega_{N-1}^+  \right)}  \right|, \label{eq:T estimate} \\
        \left| (H_N(\eta))_{j,k} \right| \leq & Constant \left| \frac{\prod_{l=j}^{k-1}\omega_l^-\left( \prod_{l=2}^j\frac{\omega_l^-}{\omega_l^+} +\prod_{l=k}^{N-1}\frac{\omega_l^-}{\omega_l^+} \right)}{a_{k-1}\omega_j^-\left( \omega_{N-1}^--\omega_{N-1}^+  \right)} \right| \label{eq:H estimate},
      \end{align}
      where $Constant$ is given by \eqref{eq:constant} that only depends on $c_0,c_1,c_2,\varepsilon_1$ and $\varepsilon_2$ and remains bounded as $\varepsilon_1,\varepsilon_2\to0$.
      
      Further we have  for any $j,k,l$, as $\varepsilon_1,\varepsilon_2\to 0$, 
      \begin{equation}
        (T_N(\eta))_{j,k} 
        =  \frac{(-1)^{k-j}\omega_{N-1}^-}{\omega_{N-1}^+-\omega_{N-1}^-} \frac{\prod_{m=j}^{k-1}\omega_m^-}{a_{k-1}\omega_j^-}(1+O(\varepsilon_1+\varepsilon_2)) \label{eq:T}
      \end{equation}
      and
      \begin{equation}
        \frac{(T_N(\eta))_{j,k}  }{(T_N(\eta))_{j+l,k+l} }
        = \left( \prod_{m=j}^{k-1}\frac{\omega_m^-}{\omega_{m+l}^-} \right)\frac{\omega_{l+j}^-a_{k+l-1}}{\omega_{j}^-a_{k-1}}  \left( 1+O(\varepsilon_1) \right) \label{eq:T ratio estimate} .
      \end{equation}
      By convention we take $\prod_{l=k}^{k-1}\equiv 1$. Note that $J_N$ is symmetric and so is $J_N^{-1}$. Therefore we have the approximation of each entry of the inverse. 
      
    \end{proposition}
    
    \begin{proof}
      Applying the inverse formula for a tri-diagonal matrix \eqref{eq:pre  difference_beta 2}, we have  for $j\leq k$
      \begin{equation}\label{eq:difference_beta}
        (J_N^{-1})_{j,k} = \frac{(-1)^{k-j}\gamma_{j}\beta_{k}}{\beta_{N}a_{N}\gamma_{N+1}}  .
      \end{equation}
      Choose any $ a_{N}\neq 0 $ (e.g. pick $a_N=a_{N-1}$) and
      Note that the inverse formula is independent on the choice of $a_{N}$.  Recall that in the proof of Proposition \ref{prop: tri inverse} we have computed $\beta_k$ to be \eqref{eq:beta 2}, i.e.,  for $k=1,\dots, N$
      
      \begin{equation}
        \beta_{k} =  \frac{\tilde{\beta}_1\prod_{l=k}^{N-1}\omega_l^+}{\omega_{N-1}^--\omega_{N-1}^+} \left( 1+\frac{\tilde{\beta}_2}{\tilde{\beta}_1}\prod_{l=k}^{N-1}\frac{\omega_l^-}{\omega_l^+}+C(k) \right)
      \end{equation}
      where  $C(k)$ is defined in \eqref{eq:estiamte ck} with $|C(k)|\leq 4(1+c_2)\varepsilon_1$ and 
      \begin{equation}
        \tilde{\beta}_1 \coloneqq \omega_{N-1}^-a_{N-1}-b_{N-1}+z, \qquad \tilde{\beta}_2 \coloneqq -\omega_{N-1}^+a_{N-1}+b_{N-1}-z. \label{eq:def:beta12}
      \end{equation}
      Note that we have $\beta_N=a_{N-1},$ $\left| \frac{\tilde{\beta}_2}{\tilde{\beta}_1}\prod_{l=k}^{N-1}\frac{\omega_l^-}{\omega_l^+} \right|\leq c_2\varepsilon_2$. 
      
      Recall that transfer matrices  $B_j$ is defined as  \eqref{eq:pre transfer matrix B}, and $\gamma_j$ is defined recursively to be
      \begin{equation}\label{eq:gamma 1}
        \begin{pmatrix}
          \gamma_j \\ \gamma_{j+1} 
        \end{pmatrix} =  B_jB_{j-1}\dots B_{2} \begin{pmatrix}
          \gamma_{1} \\ \gamma_{2} 
        \end{pmatrix}, \quad \gamma_1=a_1, \quad \gamma_2=b_{0}-z. 
      \end{equation}
      
      Recall that the eigenvalues $\lambda_j^+$ and $\lambda_j^-$ of $B_j$ are defined as \eqref{eq:pre eigenvalue_difference_eq}. $B_j$ can be diagonalised as $	B_j = W_j \Lambda_j W_j^{-1}$, where $W_j \coloneqq \begin{pmatrix}
        1 & 1 \\ \lambda_j^+ & \lambda_j^- 
      \end{pmatrix}$ and $ \Lambda_j  \coloneqq \begin{pmatrix}
        \lambda_j^+ & 0 \\0 &  \lambda_j^- 
      \end{pmatrix}$.

      Further, define 
      \begin{equation}
        D_j\coloneqq W_j^{-1}\left(B_jB_{j-1}\dots B_{2}- W_j\left(\prod_{l=2}^{j} \Lambda_l\right)W_{2}^{-1}\right)W_{2}, \quad j=2,3,\dots, N, \quad \text{and } D_1\coloneqq 0. 
      \end{equation} 
      We can rewrite \eqref{eq:gamma 1}  for all $j=2,3,\dots, N$ to be 
      \begin{equation}
        \begin{pmatrix}
          \gamma_j \\ \gamma_{j+1} 
        \end{pmatrix}  = W_j\left(D_j+\prod_{l=2}^j \Lambda_l\right)W_{2}^{-1}\begin{pmatrix}
          \gamma_{1} \\ \gamma_{2} 
        \end{pmatrix} .
      \end{equation}
      Let $\tilde{\gamma}_1\coloneqq \lambda_{2}^-a_{1}-b_{0}+z$, $\tilde{\gamma}_2\coloneqq -\lambda_{2}^+a_{1}+b _{0}-z$ such that $\frac{1}{\lambda_2^--\lambda_2^+}\begin{pmatrix}
        \tilde{\gamma}_1\\\tilde{\gamma}_2
      \end{pmatrix}= W_{2}^{-1}\begin{pmatrix}
      \gamma_{1} \\ \gamma_{2} \end{pmatrix} $. Then, we have for $j=2,\dots, N$
      \begin{align}
        \gamma_{j} = & \frac{\left((D_j)_{1,1}+(D_j)_{2,1}+\prod_{l=2}^{j}\lambda_l^+\right)\tilde{\gamma}_1+\left((D_j)_{1,2}+(D_j)_{2,2}+\prod_{l=2}^{j}\lambda_l^-\right)\tilde{\gamma}_2}{\lambda_{2}^--\lambda_{2}^+} .
      \end{align}
      By convention, let $D_1\equiv 0 $ and $\prod_{l=2}^{1}\equiv 1$ and we can allow $j=1,\dots, N$ for the formula $\gamma_j$ above.

      For $\gamma_{N+1}$, we keep  $B_{N}$ unchanged, estimate from $B_{N-1}$ and write 
      \begin{equation}
        \begin{pmatrix}
          \gamma_N \\ \gamma_{N+1} 
        \end{pmatrix}  = B_{N}W_{N-1}\left(D_{N-1}+\prod_{l=2}^{N-1} \Lambda_l\right)W_{2}^{-1}\begin{pmatrix}
          \gamma_{1} \\ \gamma_{2} 
        \end{pmatrix} .\label{eq:gamma Nplus1}
      \end{equation}
      Note that 
      \begin{equation}
        B_{N}W_{N-1}=\begin{pmatrix}
          \lambda_{N-1}^+ & \lambda_{N-1}^- \\ \frac{\left( b_{N-1} -z\right)\lambda_{N-1}^+-a_{N-1}}{a_N} & \frac{\left( b_{N-1} -z\right)\lambda_{N-1}^--a_{N-1}}{a_N} 
        \end{pmatrix} .
      \end{equation}
      Denote 
      \begin{equation}
        \tilde{\delta}_{1} \coloneqq \left( b_{N-1} -z\right)\lambda_{N-1}^+-a_{N-1} , \qquad 	\tilde{\delta}_{2} \coloneqq \left( b_{N-1} -z\right)\lambda_{N-1}^--a_{N-1} . \label{eq:def:lambda12}
      \end{equation}
      Compute the second entry of \eqref{eq:gamma Nplus1} to obtain
      \begin{equation}
        \gamma_{N+1} =  \frac{\tilde{\delta}_{1} \left((D_{N-1})_{1,1}+\frac{\tilde{\delta}_{2}}{\tilde{\delta}_{1} }(D_{N-1})_{2,1}+\prod_{l=2}^{N-1}\lambda_l^+\right)\tilde{\gamma}_1+\tilde{\delta}_{2} \left(\frac{\tilde{\delta}_{1}}{\tilde{\delta}_{2} }(D_{N-1})_{1,2}+(D_{N-1})_{2,2}+\prod_{l=2}^{N-1}\lambda_l^-\right)\tilde{\gamma}_2}{a_N\left( \lambda_{2}^--\lambda_{2}^+ \right)} .
      \end{equation}
      By convention we let $D_1\equiv 0 $ and $\prod_{l=2}^1\equiv 1$ and we can allow $j=1,\dots, N$ for all above. Moreover, by the construction of $E_j$  in \eqref{eq:eigenvalue_difference_eq_E_j}, we have $E_j \equiv W_j^{-1}W_{j-1}-\mathrm{Id}$. Now we estimate $D_j$ via the expansion 
      \begin{equation}
        B_jB_{j-1}\dots B_{2}  = W_j \Lambda_j(\mathrm{Id}+E_j) \Lambda_{j-1}(\mathrm{Id}+E_{j-1})\dots  \Lambda_{3}(\mathrm{Id}+E_{3}) \Lambda_{2}W_{2}^{-1}.
      \end{equation}
      Also, note that by assumption $b_j>0$, we have $|\lambda_j^+|> |\lambda_j^-|$. Hence, $\|W_j\|_\infty\leq |\lambda_j^+|$. Then, we have
      \begin{equation}\label{eq:DjupperBound}
        \|D_j\|_\infty \leq \left( \prod_{l=2}^{j}\left( 1+\|E_l\|_\infty \right) -1\right)\prod_{l=2}^{j}|\lambda_l^+|.
      \end{equation}

      Now we are about to estimate $E_j$. We will do so by revealing some connections between $\lambda_j^{\pm}$ and $\omega_j^{\pm}$ and between $E_j$ and $M_j$.
      
      By \eqref{ass:recurrence bound}, let $c_{01} = \frac{(1+\sqrt{5})c_1}{2c_0}$ then
      \begin{align}
        |\omega_j^+|\leq & \frac{\left|b_{j-1}-z\right|+\mysqrt{\left|b_{j-1}-z\right|^2+4|a_{j-1}a_j|}}{2|a_{j-1}|}\leq c_{01}, \label{eq:omega+bound}\\
        |\omega_j^-|\geq & \frac{2|a_{j}|}{\left|b_{j-1}-z\right|+\mysqrt{\left|b_{j-1}-z\right|^2+4|a_{j-1}a_j|}}\geq \frac{1}{c_{01}} \label{eq:omega-bound} .
      \end{align}
      Note that each entry of the $2$ by $2$ matrix  $E_j$ is bounded by its operator norm. 
      Since $\lambda_j^-\omega_j^+=\lambda_j^+\omega_j^-=1$, by \eqref{eq:omega+bound} and \eqref{eq:omega-bound}, we have each entry of $E_j$ is bounded by $c_{01}^2\|M_j\|_{\infty}$. Hence, $\|E_j\|_\infty\leq 3c_{01}^2\|M_j\|_\infty$. By \eqref{eq:assuption_tech2} and $0<3c_{01}^2\varepsilon_1< 1$, by the same engagement as in \eqref{eq:ck_epsilon1}, we further estimate \eqref{eq:DjupperBound} to be
      \begin{equation}
        \|D_j\|_\infty \leq 6 c_{01}^2\varepsilon_1\prod_{l=2}^{j}|\lambda_l^+| . \label{eq:dj_epsilon1}
      \end{equation}
      Then, let 
      \begin{align}
        D(j)\coloneqq &  \left((D_j)_{1,1}+(D_j)_{2,1}\right)\prod_{l=2}^{j}(\lambda_l^+)^{-1}+\left((D_j)_{1,2}+(D_j)_{2,2}\right)\frac{\tilde{\gamma}_2}{\tilde{\gamma}_1}\prod_{l=2}^{j}(\lambda_l^+)^{-1}, \\
        \tilde{D}(N-1)\coloneqq &  \left((D_{N-1})_{1,1}+\frac{\tilde{\delta}_{2}}{\tilde{\delta}_{1} }(D_{N-1})_{2,1}\right)\prod_{l=2}^{N}(\lambda_l^+)^{-1}+\left((D_{N-1})_{1,2}+\frac{\tilde{\delta}_2}{\tilde{\delta}_1}(D_{N-1})_{2,2}\right)\frac{\tilde{\gamma}_2}{\tilde{\gamma}_1}\prod_{l=2}^{N-1}(\lambda_l^+)^{-1} .
      \end{align}
      Hence, by \eqref{eq:dj_epsilon1}
      \begin{equation}
          |D(j)|\leq 12 \left( 1+\left| \frac{\tilde{\gamma}_2}{\tilde{\gamma}_1} \right| \right)c_{01}^2\varepsilon_1, \quad |\tilde{D}(N-1)|\leq 6 \left( 1+\left| \frac{\tilde{\delta}_2}{\tilde{\delta}_1} \right| \right)\left( 1+\left| \frac{\tilde{\gamma}_2}{\tilde{\gamma}_1} \right| \right)c_{01}^2\varepsilon_1.
      \end{equation}
      Note that $\lambda_j^-\omega_j^+=\lambda_j^+\omega_j^-=1$ and we have $\frac{\tilde{\delta}_2}{\tilde{\delta}_1}=\frac{\left( b_{N-1} -z\right)\lambda_{N-1}^--a_{N-1} }{\left( b_{N-1} -z\right)\lambda_{N-1}^+-a_{N-1} }=\frac{\omega_{N-1}^-}{\omega_{N-1}^+} \frac{\omega_{N-1}^+a_{N-1}-b_{N-1}+z}{\omega_{N-1}^-a_{N-1}-b_{N-1}+z}$ and  $\frac{\tilde{\gamma}_2}{\tilde{\gamma}_1}=\frac{-\lambda_2^+a_1+b_0-z}{\lambda_2^-a_1-b_0+z}=\frac{\omega_2^+}{\omega_2^-}\frac{-a_1+\omega_2^-(b_0-z)}{a_1-\omega_2^+(b_0-z)}$. Then by \eqref{eq:omega+bound} \eqref{eq:omega-bound}, we estimate, 
      \begin{equation}
        \left|  \frac{\tilde{\delta}_2}{\tilde{\delta}_1}\right| \leq c_2, \quad	\left|  \frac{\tilde{\gamma}_2}{\tilde{\gamma}_1}\right| \leq c_{01}^2 c_2.
      \end{equation}
      Recall that $c_2\geq 1$. Further define 
      \begin{equation}
        \tilde{\varepsilon_1} \coloneqq 6(1+c_2)\left(1+c_{01}^2 c_2^2\right)c_{01}^2\varepsilon_1.
      \end{equation}
      Then by \eqref{eq:dj_epsilon1} we have $|C(k)|, |D(j)|, |\tilde{D}(N-1)|<\tilde{\varepsilon_1}$ for all $j,k$. Also, note that by definition we have the following handy relations  
      \begin{equation}
        \lambda_j^-=(\omega_j^+)^{-1}, \qquad \lambda_j^+=(\omega_j^-)^{-1}, \qquad \omega_j^+\omega_j^-=\frac{a_j}{a_{j-1}}, \qquad   \frac{\tilde{\beta}_1}{\tilde{\delta}_1} = -\omega_{N-1}^-.
      \end{equation}
   
      Hence for $j\leq k$ 
      \begin{multline}
        (J_N^{-1})_{j,k}  =  \frac{(-1)^{k-j}\beta_{k}}{\beta_{N}a_{N}} \frac{\gamma_{j}}{\gamma_{N+1}} \\
        =   \frac{(-1)^{k-j}\beta_k}{a_{N-1}a_{N}} \frac{\left((D_j)_{1,1}+(D_j)_{2,1}+\prod_{l=2}^{j}\lambda_l^+\right)\tilde{\gamma}_1+\left((D_j)_{1,2}+(D_j)_{2,2}+\prod_{l=2}^{j}\lambda_l^-\right)\tilde{\gamma}_2}{\tilde{\delta}_{1} \left((D_{N-1})_{1,1}+\frac{\tilde{\delta}_{2}}{\tilde{\delta}_{1} }(D_{N-1})_{2,1}+\prod_{l=2}^{N-1}\lambda_l^+\right)\tilde{\gamma}_1+\tilde{\delta}_{2} \left(\frac{\tilde{\delta}_{1}}{\tilde{\delta}_{2} }(D_{N-1})_{1,2}+(D_{N-1})_{2,2}+\prod_{l=2}^{N-1}\lambda_l^-\right)\tilde{\gamma}_2}\\
        =
         \frac{(-1)^{k-j}\omega_{N-1}^-}{a_{k-1}} \frac{\prod_{l=j}^{k-1}\omega_l^-}{\left( \omega_{N-1}^+-\omega_{N-1}^-  \right)\omega_j^-} 
        \frac{\left( 1+\frac{\tilde{\gamma}_2}{\tilde{\gamma}_1}\prod_{l=2}^j\frac{\omega_l^-}{\omega_l^+}+D(j) \right)\left( 1+\frac{\tilde{\beta}_2}{\tilde{\beta}_1}\prod_{l=k}^{N-1}\frac{\omega_l^-}{\omega_l^+}+C(k) \right)}{ \left( 1+\frac{\tilde{\delta}_2\tilde{\gamma}_2}{\tilde{\delta}_1\tilde{\gamma}_1}\prod_{l=2}^{N-1}\frac{\omega_l^-}{\omega_l^+}+\tilde{D}(N-1)  \right)}  .
      \end{multline}
      Let 
      \begin{align}
        (T_N(\eta))_{j,k}  
        \coloneqq & \frac{(-1)^{k-j}\omega_{N-1}^-}{a_{k-1}} \frac{\prod_{l=j}^{k-1}\omega_l^-}{\left( \omega_{N-1}^+-\omega_{N-1}^-  \right)\omega_j^-}   \frac{\left( 1+D(j) \right)\left( 1+C(k) \right)}{ \left( 1+\frac{\tilde{\delta}_2\tilde{\gamma}_2}{\tilde{\delta}_1\tilde{\gamma}_1}\prod_{l=2}^{N-1}\frac{\omega_l^-}{\omega_l^+}+\tilde{D}(N-1)  \right)},   \\
        H_N(\eta)\coloneqq & J_N^{-1}-T_N(\eta) .\label{eq:def H}
      \end{align}
      Hence as $\varepsilon_1\to 0$ we have 
      \begin{align}
        \frac{(T_N(\eta))_{j,k}  }{(T_N(\eta))_{j+l,k+l} }
        = &\left( \prod_{m=j}^{k-1}\frac{\omega_m^-}{\omega_{m+l}^-} \right)\frac{\omega_{l+j}^-a_{k+l-1}}{\omega_{j}^-a_{k-1}}  \frac{\left( 1+D(j) \right)\left( 1+C(k) \right)}{\left( 1+D(j+l) \right)\left( 1+C(k+l) \right)} \\
        = &\left( \prod_{m=j}^{k-1}\frac{\omega_m^-}{\omega_{m+l}^-} \right)\frac{\omega_{l+j}^-a_{k+l-1}}{\omega_{j}^-a_{k-1}} \left( 1+O(\varepsilon_1) \right) .
      \end{align}
      
      Hence, whenever $\tilde{\varepsilon_1}+c_{01}^2 c_2\varepsilon_2<1$ 
      \begin{equation}\label{eq:constant}
        \left|  \frac{\omega_{N-1}^-\left( 1+D(j) \right)\left( 1+C(k) \right)}{ \left( 1+\frac{\tilde{\gamma}_2}{\tilde{\gamma}_1}\prod_{l=2}^{N-1}\frac{\omega_l^-}{\omega_l^+}+\tilde{D}(N-1)  \right)} \right|\leq \frac{c_{01}(1+\tilde{\varepsilon}_1)^2}{1-c_{01}^2 c_2\varepsilon_2-\tilde{\varepsilon_1}}  \eqqcolon Constant' . 
      \end{equation}

      Hence, 
      \begin{align}
        \left| (T_N(\eta))_{j,k}   \right|
        \leq & Constant' \left| \frac{\prod_{l=j}^{k-1}\omega_l^-}{a_{k-1}\omega_j^-\left( \omega_{N-1}^--\omega_{N-1}^+  \right)}  \right|, \\
        \left| (H_N(\eta))_{j,k} \right| \leq & 2 Constant' \left| \frac{\left( \prod_{l=j}^{k-1}\omega_l^- \right)\left( \prod_{l=2}^j\frac{\omega_l^-}{\omega_l^+} +\prod_{l=k}^{N-1}\frac{\omega_l^-}{\omega_l^+} \right)}{a_{k-1}\left( \omega_{N-1}^--\omega_{N-1}^+  \right)} \right|. 
      \end{align}
    Take the $Constant=2Constant'$ and we conclude the proof of Proposition~\ref{prop: tri-inverse}. 
    \end{proof}

    \subsection{Proof of Proposition~\ref{prop: overview tri-inverse} }\label{sec:proof tri-inverse}

    \begin{proof}[Proof of Proposition~\ref{prop: overview tri-inverse} ] \label{proof: prop2}
      
      Recall that we define $J_{n\pm 2mn^\beta}^{(r)}$ as \eqref{eq:truncation 2}, which can be regarded as a three-block diagonal matrix. The non-trivial block of $J_{n\pm 2mn^\beta}^{(r)}$ is where the entries have the indices both belonging to this set $I_{n,m}^{(\beta)}$, defined as \eqref{eq:def I}, .i.e., for all $j,k\in I_{n,m}^{(\beta)}$ 
      \begin{equation}\label{eq:middle block}
        \left(J_{n\pm 2mn^\beta}^{(r)}  \right)_{j,k} = \left( P_{n+2mn^\beta}Q_{n-2mn^\beta}J^{(r)}Q_{n-2mn^\beta}P_{n+2mn^\beta} \right)_{j,k}.
      \end{equation}
      In the following, we apply Proposition~\ref{prop: tri-inverse} to this non-trivial block of $J_{n\pm 2mn^\beta}^{(r)}$, i.e., \eqref{eq:middle block}.
      To shorten the notation, let
      \begin{equation}
        b_{j}= b_{j,n}-x_0, \quad a_j= a_{j,n}, \quad  \text{and } z=\frac{\eta}{n^\alpha}.
      \end{equation}
      Then, at the left edge, for all $j\in I_{n,m}^{(\beta)}$, $a_j$ and $b_j$ satisfy Conditions~\ref{ass:slowly varying} and~\ref{ass:rate controlling}. Hence, the quantities in \eqref{ass:recurrence bound} are bounded below and above as desired. 
      
      By Lemma~\ref{lemma: left edge locating}, we have $b_l>0$, for all $l\in I_{n,m}^{(\beta)}$. By Lemma ~\ref{lemma: finner est omega}, we have as $n\to\infty$
      \begin{equation}
        \max_{l\in I_{n,m}^{(\beta)}}|\omega_l^{-}-1|=O(n^{-\frac{\alpha}{2}}),\quad \max_{l\in I_{n,m}^{(\beta)}}|\omega_l^{+}-1|=O(n^{-\frac{\alpha}{2}})
      \end{equation}
      Hence, as $n\to\infty$
      \begin{multline*}
        \left|\frac{-a_{n-2mn^\beta}+\omega_{n-2mn^\beta+1}^-(b_{n-2mn^\beta-1}-z)}{a_{n-2mn^\beta}-\omega_{n-2mn^\beta+1}^+(b_{n-2mn^\beta-1}-z)}\right|  =1+ O(n^{-\frac{\alpha}{2}}), \quad
        \left|\frac{-\omega_{n+2mn^\beta-1}^+a_{n+2mn^\beta-1}+b_{n+2mn^\beta-1}-z}{\omega_{n+2mn^\beta-1}^-a_{n+2mn^\beta-1}-b_{n+2mn^\beta-1}+z}\right|  =1+ O(n^{-\frac{\alpha}{2}}).
      \end{multline*}
      This shows $c_2=1+O(n^{-\alpha/2})$ in \eqref{eq:assumption_betatilde}.

      Let $M_j$ be defined as \eqref{eq:eigenvalue_difference_eq_M_j}. By Lemma~\ref{lemma:small M} we have as $n\to\infty$
      \begin{equation}
        4mn^\beta\max_{l\in I_{n,m}^{(\beta)}}\|M_l\|_\infty=o(n^{-2\beta+\alpha}).
      \end{equation}
      This shows $\varepsilon_1=o(n^{-2\beta+\alpha})$ in \eqref{eq:assuption_tech2}. Moreover, use Lemma ~\ref{lemma: finner est omega}  again to obtain that there exists a constant $d>0$ such that for all $n>n_0$ 
      \begin{equation}
        \prod_{l=n-2mn^\beta+2}^{n+2mn^\beta-1}\left| \frac{\omega_l^-}{\omega_l^+} \right| \leq \left( 1-\frac{d}{n^{\frac{\alpha}{2}}} \right)^{4mn^\beta-2}\leq e^{-d n^{-\frac{\alpha}{2}}(4mn^\beta-2)}.
      \end{equation}
      This shows $\varepsilon_2 =O(e^{-d'n^{\beta-\frac{\alpha}{2}}})$ in \eqref{eq:ass epsilon2}, which is exponentially small. 
      
      Hence, the assumptions of Proposition~\ref{prop: tri-inverse} hold for the middle block of $J_{n\pm 2mn^\beta}^{(r)} $. Hence, we obtain the estimate of $T_{n\pm2mn\beta}(\eta_r)$ to be \eqref{eq:T}. Further, by $a_{j}=a_{n}+O(n^{\beta-1})$, \eqref{eq:T} and Lemma~\ref{lemma: finner est omega}, we have as $n\to\infty$
      \begin{equation}
        (T_{n\pm 2mn^\beta}(\eta_r))_{j,k} 
        =  \frac{(-sgn(a_{N-1}) )^{k-j}}{2a_{n}\mysqrt{\frac{-\eta_r}{n^\alpha |a_{n}|}}}\prod_{l=j}^{k-1}\left( 1- \mysqrt{\frac{-\eta_r}{n^\alpha |a_{n}|}}+\xi_l \right)(1+o(1)), 
      \end{equation}
      where $\xi_l\coloneqq 	\omega_l^{-}sgn(a_{N-1})  -\left(  1-\mysqrt{\frac{-\eta}{n^\alpha |a_{N-1}|}} \right) =o\left(n^{-\frac{\alpha}{2}}\right)$. Note that this formula is for the left edge and for the right edges one should replace $a_j$ by $-a_j$,  as is described in Condition~\ref{ass:rate controlling}.

      To prove \eqref{eq:overview tri-inverse T ratio}.  Note that since each entry of $M_l$ is bounded by the operator norm, we have for all $l=n-2mn^\beta,\dots, n+2mn^\beta$,
      \begin{equation*}
        |\omega_{l}^--\omega_{l+1}^-| \leq \|M_l\|_\infty 	|\omega_{l}^--\omega_{l}^+| .
      \end{equation*}
      By  Lemma~\ref{lemma:small M} we have $\|M_l\|_\infty=o(n^{-3\beta+\alpha})$, given that $|\omega_{l}^--\omega_{l}^+| $ is of order $n^{-\frac{\alpha}{2}}$. Hence we have  
      \begin{equation}
        \max_{l\in I_{n,m}^{(\beta)}}|\omega_{l}^--\omega_{l+1}^-| =o(n^{-3\beta+\frac{\alpha}{2}}),
        \quad \text{ as } n\to\infty.
      \end{equation}	
      Lemma~\ref{lemma: finner est omega} says that $\omega_m^-$ is bounded away from zero, hence we estimate \eqref{eq:T ratio estimate} to be 
      \begin{align}
        \frac{(T_{n\pm 2mn^\beta}(\eta_r))_{j,k}  }{(T_{n\pm 2mn^\beta}(\eta_r))_{j+l,k+l} }
        = &\left(1+o(n^{-3\beta+\frac{\alpha}{2}}) \right)^{(k-j)l}\left( 1+O(n^{\beta-1}) \right) \left( 1+O(\varepsilon_1)\right) ,
        \quad \text{ as } n\to\infty.
      \end{align}
      Note that $1+O(n^{\beta-1})=\frac{a_{k+l-1}}{a_{k-1}}$ comes from  Condition~\ref{ass:slowly varying}. Also note that $\varepsilon_1=o(n^{-2\beta+\alpha})$. Moreover, $(k-j)l=O(n^{2\beta})$, since $1\leq k-j,l\leq 4mn^\beta$. Note that all the of the big-$O$s and small-$o$ are uniform in the indices by assumptions. Together with the assumption $0<\frac{\alpha}{2}<\beta<\frac{\alpha+1}{3}<1$, we have 
      
      \begin{equation*}
        \frac{(T_{n\pm 2mn^\beta}(\eta_r))_{j,k}  }{(T_{n\pm 2mn^\beta}(\eta_r))_{j+l,k+l} }
        = 1+o(n^{-\beta+\frac{\alpha}{2}}),
        \quad \text{ as } n\to\infty.
      \end{equation*}
      
      This shows the statement \eqref{eq:overview tri-inverse T ratio}.
      
      Note that above argument is taken at the left edge i.e., $b_j>0$ as assumed in \eqref{eq:assumption_recurrance_b}. For the right edge, we have $b_j<0$. Then we consider $-J_{n\pm 2mn^\beta} $ whose diagonals are $-b_j+\frac{\eta_r}{n^{\frac{\alpha}{2}}}$ and off diagonals are $-a_j$. Apply the result above by replacing $a_{n}$ by $-a_{n}$ and $\eta_r$ by $-\eta_r$ to estimate entries of $\left(- J_{n\pm 2mn^\beta}^{(r)} \right)^{-1}$ as in the statement. Then the result at the right edge is obtained by the simple fact $\left(J_{n\pm 2mn^\beta}^{(r)} \right)^{-1}=-\left(- J_{n\pm 2mn^\beta}^{(r)} \right)^{-1}$.
      
      This completes the proof. 
    \end{proof}

    \subsection{Proof of Proposition~\ref{prop: overview trimming 2}}\label{sec:proof overview trim 2}
    Now we are ready to prove the Proposition~\ref{prop: overview trimming 2} by Lemmas~\ref{cutoff},  ~\ref{cutoff 2} and~\ref{lemma:traceAlgebras} and Proposition~\ref{prop: overview tri-inverse}.

    \begin{proof}[Proof of Proposition~\ref{prop: overview trimming 2}]
      In the following we consider $\mathcal{A}=F_{n\pm2mn^\beta}$ and $\mathcal{B}=\sum_{r}c_rT_{n\pm 2mn^\beta}\left(\eta_r \right)$. Recall that the operator norm of a linear operator can be estimated by entries, i.e., $\|\mathcal{B}\|_\infty\leq \sum_{j=-\infty}^{\infty}\sup_k|(\mathcal{B})_{k,k+j}|$. Together with \eqref{eq:overview T Toeplitz} in Proposition~\ref{prop: overview tri-inverse}, we compute that $n^{-\alpha}\|\mathcal{B}\|_{\infty}$ is uniformly bounded for all $n$. That is there exists a constant $C_{op}>0$ such that $\|\mathcal{A}\|_\infty,\|\mathcal{B}\|_\infty \leq n^\alpha C_{op}$.
      
      By \eqref{eq:overview T Toeplitz} in Proposition~\ref{prop: overview tri-inverse}, we have for all $|j-k|\geq n^{\beta}$
      \begin{equation}
        \left| \left( \mathcal{B} \right)_{j,k} \right|\leq  C_0n^{\frac{\alpha}{2}}e^{-d_0n^{\beta-\frac{\alpha}{2}}} .
      \end{equation}
      Then use the same proof as in Lemma~\ref{cutoff}, we have for any $l\in \mathbb{N}$ with $l\leq m$, 
      
      \begin{equation}
        \|P_n\mathcal{B}^lQ_{n+mn^\beta}\|_1\leq C_me^{-d'n^{\beta-\frac{\alpha}{2}}}.
      \end{equation}
      
      Hence, by Lemma~\ref{lemma:traceCompare}, we get
      \begin{multline}\label{eq:step 3}
        \left| \Tr(\mathcal{A})^{l_1}P_n\cdots (\mathcal{A})^{l_j}P_n-\Tr(\mathcal{A}^mP_n)-\Tr(\mathcal{B})^{l_1}P_n\cdots (\mathcal{B})^{l_j}P_n+\Tr(\mathcal{B}^mP_n) \right| \\ 
        \leq  (m+1)\left(  \left( n^\alpha C_{op}\right)^{m-1} \|(\mathcal{A}-\mathcal{B})P_{n+mn^\beta}Q_{n-mn^\beta}\|_1 + R_n  \right).
      \end{multline}
        
      Note that,
      \begin{equation*}
        \|(\mathcal{A}-\mathcal{B})P_{n+mn^\beta}Q_{n-mn^\beta}\|_1 \leq \sum_r|c_r|\|H_{n\pm 2mn^\beta}(\eta_r)P_{n+mn^\beta}Q_{n-mn^\beta}\|_1,
      \end{equation*}
      where $H_{n\pm 2mn^\beta}(\eta_r)$ is defined in \eqref{eq:def H}. Thus  by Proposition~\ref{prop: overview tri-inverse}, we have 
      \begin{multline}
        \|H_{n\pm 2mn^\beta}(\eta_r)P_{n+mn^\beta}Q_{n-mn^\beta}\|_1
        \leq \sum_{j = n-2mn^\beta+1}^{n+2mn^\beta}\sum_{k=n-mn^\beta +1}^{n+mn^\beta} \\
        Constant\quad  n^{\frac{\alpha}{2}}\left( \left( 1-\frac{d}{n^{\frac{\alpha}{2}}} \right)^{\max\{j,k\} -n+2mn^\beta}+\left( 1-\frac{d}{n^{\frac{\alpha}{2}}} \right)^{n+2mn^\beta-\min\{j,k\}} \right) \\
        \leq Constant \quad 16m n^{2\beta+\frac{\alpha}{2}}e^{-dmn^{\beta-\frac{\alpha}{2}}}.
      \end{multline}
      This shows that \eqref{eq:step 3} is exponentially small. Hence plug this estimate into the definition \eqref{eq:cumulant operator} and we have $ \left| 	\mathcal{C}_m^{(n)}(\mathcal{A}) -	\mathcal{C}_m^{(n)}(\mathcal{B})  \right| \leq C_m''e^{-d''n^{\beta-\frac{\alpha}{2}}}$.
   Combining with Proposition~\ref{prop: overview trimming}, we conclude \eqref{eq:trim_cumulant}. 
      
    \end{proof}

    \section{Proof of Proposition~\ref{thm: overview strong szego}}\label{sec:strong szego}
    Before we prove the Proposition~\ref{thm: overview strong szego}, we need to study the limiting behaviour of
    \begin{equation}\label{eq:sec7:logdet}
      \log \det \left(\mathrm{Id}+P_n\left(e^{t n^{-\alpha}\sum_{r}c_rT_{n\pm 2mn^\beta}\left(\eta_r \right)}-\mathrm{Id}\right)P_n\right)e^{-t \Tr P_n n^{-\alpha}\sum_{r}c_rT_{n\pm 2mn^\beta}\left(\eta_r \right)},
    \end{equation}
    where $T_{n\pm 2mn^\beta}$ comes from Proposition~\ref{prop: overview tri-inverse}. In Subsection~\ref{sec:detAbstract}, we will first give an abstract determinant estimate cf. Proposition~\ref{prop:logdetOrderBound}. Then apply it to prove Proposition~\ref{thm: overview strong szego} in Subsection~\ref{sec proof overview strong szego}.

    \subsection{Limiting Behaviour of Fredholm Determinants}\label{sec:detAbstract}

In the following we assume $N = N_n$ is an even integer. Let $\vartheta = \vartheta_n >0$ and $w = w_n>0$. Consider $A = A(n),B= B(n)\in\mathbb{C}^{N\times N}$ whose operator norms are bounded by a uniform constant $M_0$. In this subsection we suppress the parameter $n$ when this causes no confusion. Assume that $A$ and $B$ satisfy the following entry bound, with constants $C_0,d_0>0$,
\begin{equation}\label{eq:ass:abstract:improvedCT}
 |A_{i,j}|+|B_{i,j}|\leq \frac{C_0}{w}e^{-d_0|i-j|/w}.
\end{equation}

\begin{equation}\label{eq:ass:abstract:diagonal}
|A_{i+h,j+h}-A_{i,j}|
   +|B_{i+h,j+h}-B_{i,j}|\leq \frac{\vartheta}{w}e^{-d_0|i-j|/w},
\end{equation}
whenever all indices belong to $\{1,\dots,N\}$. These assumptions mean that $A$ is close enough to a Toeplitz matrix with a small error. Namely, define $T_A$ to be a Toeplitz matrix of size $N$ with symbol $u^A$ such that
\begin{equation}
    u^A_i\coloneqq \begin{cases}
        A_{i+1,1}, \quad i\geq 0\\
        A_{1,1-i}, \quad i<0
    \end{cases}, 
    \qquad (T_A)_{i,j}\coloneqq u^A_{i-j} .
\end{equation}
A direct consequence of \eqref{eq:ass:abstract:improvedCT} is that $|u_i^A|\leq \frac{C_0}{w}e^{-d_0|i|/w}$. Define $T_B$ and $u^B_i$ analogously.

\begin{lemma}\label{lemma:norm}
Fix $0<b<d_0$ and define a matrix norm $\|\cdot\|_{\infty,b}$ such that
\begin{equation}\label{eq:weighted-norm}
 \|X\|_{\infty,b}=
 \max\left\{\sup_i\sum_j|X_{i,j}|e^{b|i-j|/w},
             \sup_j\sum_i|X_{i,j}|e^{b|i-j|/w}\right\}.
\end{equation}
The norm is well-defined and submultiplicative, and $\|X\|_{\infty}\leq \|X\|_{\infty,b}$.
\end{lemma}

\begin{proof}
The triangle inequality $|i-j|\le|i-k|+|k-j|$ proves
submultiplicativity. Domination of the operator norm comes from the operator norm bound $\|X\|_\infty\leq \sup_i\sum_j|X_{i,j|}$. 
\end{proof}

\begin{lemma}\label{lemma:communatorDiff} Assume $A$ and $B$ satisfy \eqref{eq:ass:abstract:improvedCT} and \eqref{eq:ass:abstract:diagonal}. Then there exists a universal constant $c>0$ such that 
\begin{equation}\label{eq:communtator_toeplitz}
    \left\|[A,B]-[T_A,T_B] \right\|_1 \leq c\frac{\vartheta N}{w}.
\end{equation}
Moreover, define four Hankel matrices $(H_{\pm,A})_{i,j} = u^A_{\pm(i+j-1)} $ and $(H_{\pm,B})_{i,j} = u^B_{\pm(i+j-1)} $. Then
\begin{equation}\label{eq:HankelDecomposition}
    [A,B] = -H_{+,A}H_{-,B}+H_{+,B}H_{-,A} + W( -H_{-,A}H_{+,B}+H_{-,B}H_{+,A})W + E,
\end{equation}
where $W$ is the reflection operator such that $We_j =e_{N+1-j}$ and the remainder $E$ is bounded as \eqref{eq:communtator_toeplitz}.
\end{lemma}
\begin{proof}Define $ D_A\coloneqq A - T_A$ and $D_B\coloneqq B - T_B$. Then by \eqref{eq:ass:abstract:diagonal} we have $\| D_A\|_\infty + \| D_B\|_\infty \leq c_D \vartheta $, for some universal constant $c_D>0$ only depending on $d_0$.
Use the definition of commutator to obtain
   \begin{equation}\label{eq:communtator_algebra_proof}
       [A,B]-[T_A,T_B] = [D_A,B] +[T_A,D_B].
   \end{equation}
   Then use entry-wise estimate and \eqref{eq:ass:abstract:improvedCT} we have there exists a universal constant $c_1>0$ such that
   \begin{equation}
       \|B\|_2 \leq \sqrt{\sum_{i,j=1}^N|B_{i,j}|^2}\leq c_1\sqrt{\frac{N}{w}}.
   \end{equation}
   Similarly, we also have $ \|T_A\|_2 \leq c_1\sqrt{\frac{N}{w}}$. For  $D_A$, we use \eqref{eq:ass:abstract:diagonal} to obtain that 
   \begin{equation}
       \|D_A\|_2 \leq \sqrt{\sum_{i,j=1}^N|A_{i,j}-{T_A}_{i,j}|^2}\leq c\vartheta  \sqrt{\frac{ N}{w}}.
   \end{equation}
   Similarly, we have $ \|D_B\|_2\leq c\vartheta  \sqrt{\frac{ N}{w}}$. Hence, using $\|XY\|_1\leq \|X\|_2\|Y\|_2$, we obtain 
   \begin{equation}
       \|[D_A,B]\|_1\leq 2\|D_A\|_2\|B\|_2 \leq 2c_1^2\frac{\vartheta N}{w}. 
   \end{equation}
   Similarly, we have $\|[T_A,D_B]\|_1 \leq c\frac{\vartheta N}{w}$. Then plug these estimates into \eqref{eq:communtator_algebra_proof} to conclude \eqref{eq:communtator_toeplitz}. Apply the finite Toeplitz and Hankel identity to $[T_A,T_B]$ to conclude \eqref{eq:HankelDecomposition}, where we take $E= [A,B]-[T_A,T_B]$. 
\end{proof}

\begin{lemma}\label{lemma:commutator1}
    Assume $A$ and $B$ satisfy \eqref{eq:ass:abstract:improvedCT} and \eqref{eq:ass:abstract:diagonal}. For any matrices $X$ and $Y$, we have 
    \begin{equation}
\|P_{N/2}X[A,B]YQ_{N/2}\|_1+\|Q_{N/2}X[A,B]YP_{N/2}\|_1\leq\|X\|_{\infty,b}\|Y\|_{\infty,b}\left(\frac{c\vartheta N}{w}+\frac{C_0^2}{b^2}e^{-b N/(2w)}\right),
    \end{equation}
    where $P_{N/2}$ is the projection onto the first $N/2$ coordinates and $Q_{N/2} =\mathrm{Id}- P_{N/2}$. 
\end{lemma}
\begin{proof}
Let $H$ be any of the Hankel matrices in
\eqref{eq:HankelDecomposition}, then by definition we have $|(H)_{k,j}|\leq \frac{C_0}{w}e^{-d_0(k+j-1)/w}$.
Since $i+j-1\leq|i-k|+k+j-1$ and $b<d_0$, then we have
\begin{equation*}
    |(XH)_{ij}|
 \leq\frac{C_0}{w}\sum_k|X_{ik}|e^{-b(k+j-1)/w}
 \leq\frac{C_0\|X\|_{\infty,b}}{w}e^{-b(i+j-1)/w}.
\end{equation*}

Summing squares, first over all entries and then over the rows or columns
beyond $N/2$, gives
\begin{equation}
    \|XH\|_2\leq \frac{C_0\|X\|_{\infty,b}}{2b}, \quad
 \|XHQ_{N/2}\|_2\leq \frac{C_0\|X\|_{\infty,b}}{2b}e^{-bN/(2w)}.
\end{equation}
Consequently, for  $H_1$ and $H_2$ to be any of the Hankel matrices in
\eqref{eq:HankelDecomposition}, we have 
\begin{equation}
    \|P_{N/2}XH_1H_2YQ_{N/2}\|_1
 \leq\|P_{N/2}XH_1\|_2\|H_2YQ_{N/2}\|_2
 \leq \frac{C_0^2\|X\|_{\infty,b}\|Y\|_{\infty,b}}{4b^2}e^{-bN/(2w)}.
\end{equation}
Similarly, we obtain the same bound for $\|Q_{N/2}XH_1H_2YP_{N/2}\|_1$. 

For the conjugation $W$, we have
$\|WXW\|_{\infty,b}=\|X\|_{\infty,b}$, $WP_{N/2}W=Q_{N/2}$ and $W^2=\mathrm{Id}$. Hence, 
\begin{equation}
    \|P_{N/2}XWH_1H_2WYQ_{N/2}\|_1 
 \leq\|Q_{N/2}WXWH_1\|_2\|H_2WYWP_{N/2}\|_2
 \leq \frac{C_0^2\|X\|_{\infty,b}\|Y\|_{\infty,b}}{4b^2}e^{-bN/(2w)}.
\end{equation}

Finally, $\|XEY\|_1\leq \|X\|_{\infty,b}\|Y\|_{\infty,b}\frac{c\vartheta N}{w}$ by Lemma~\ref{lemma:communatorDiff}.
Adding these estimates concludes the lemma. 
\end{proof}

\begin{lemma} \label{lemma:RH} Assume $A$ and $B$ satisfy \eqref{eq:ass:abstract:improvedCT} and \eqref{eq:ass:abstract:diagonal}.
    Let 
    \begin{equation}
        R(t;n)\coloneqq e^{-tA}e^{t(A+B)}e^{-tB}, \quad H(t;n) \coloneqq R'(t;n)R(t;n)^{-1}, \quad \varepsilon_{w,\vartheta,N} \coloneqq \frac{c\vartheta N}{w}+\frac{C_0^2}{b^2}e^{-b N/(2w)}.
    \end{equation}
    Then uniformly for all $|t|<t_0$
    \begin{equation}\label{eq:commutatorApprox}
        \left|\Tr P_{N/2}H(t;n)P_{N/2} +t\Tr P_{N/2}[A,B]P_{N/2}\right|\leq |t|^2m_{t_0,M_0}\varepsilon_{w,\vartheta,N} ,
    \end{equation}
   for some universal constant $m_{t_0,M_0}>0$ only depends on $t_0,M_0,C_0,d_0,b$. Moreover, 
    \begin{equation} \label{eq:RHApprox}
    \|P_{N/2}R(t;n)Q_{N/2}\|_1+\|Q_{N/2}R(t;n)P_{N/2}\|_1\leq |t|^2m_{t_0,M_0}\varepsilon_{w,\vartheta,N}, \quad \|P_{N/2}H(t;n)Q_{N/2}\|_1\leq |t|m_{t_0,M_0}\varepsilon_{w,\vartheta,N}  .
    \end{equation}
\end{lemma}
\begin{proof} Direct differentiation gives $R'=e^{-tA}[B,e^{t(A+B)}]e^{-tB}$.
Since $[B,A+B]=[B,A]$, by Duhamel's identity
\begin{equation}\label{eq:R1}
 R'(t)=t\int_0^1
 e^{-tA}e^{ut(A+B)}[B,A]e^{t(1-u)(A+B)}e^{-tB}\,du.
\end{equation}
Multiplication by $R(t)^{-1}$ gives
\begin{equation}\label{eq:H}
 H(t)=t\int_0^1 J_u(t)\,du,\quad \text{ where }
 J_u(s)\coloneqq V_u(s)[B,A]V_u(s)^{-1},\quad
 V_u(s) \coloneqq  e^{-sA}e^{us(A+B)}.
\end{equation}
By Lemma~\ref{lemma:norm} $\max\{\|V_u(s)\|_{\infty,b},\|V_u(s)^{-1}\|_{\infty,b}\}\leq e^ {t_0(\|A\|_{\infty,b}+\|B\|_{\infty,b})}\eqqcolon C_{t_0,M_0}$. Define 
\begin{equation}
     M_u(s)\coloneqq V_u'(s)V_u(s)^{-1}
       =-A+ue^{-sA}(A+B)e^{sA}.
\end{equation}
Then by direct computation we have $ M_u(s) = -V_u(s)(V_u(s)^{-1})'$. Hence 
\begin{equation}
     J_u'(s)=[M_u(s),J_u(s)].
\end{equation}
Insert $P_{N/2}+Q_{N/2}=\mathrm{Id}$ and take the trace to obtain
\begin{equation}\label{eq:s7-trace-cancellation}
 \Tr(P_{N/2}[M_u(s),J_u(s)]P_{N/2})=\Tr(P_{N/2}M_u(s)Q_{N/2}J_u(s)P_{N/2}-P_{N/2}J_u(s)Q_{N/2}M_u(s)P_{N/2}).
\end{equation}
Note that $\|M_u(s)\|_\infty\leq \|A\|_\infty+e^{2t_0\|A\|_\infty}\|A+B\|_\infty<c_{t_0,M_0}$ for some universal constant $c_{t_0,M_0}$ . Lemma~\ref{lemma:commutator1}, with
$X=V_u(s)$ and $Y=V_u(s)^{-1}$, now gives
\begin{equation}
     \left|\frac{d}{ds}\Tr(P_{N/2}J_u(s)P_{N/2})\right|
 \leq \|M_u(s)\|_\infty
       \bigl(\|Q_{N/2}J_u(s)P_{N/2}\|_1+\|P_{N/2}J_u(s)Q_{N/2}\|_1\bigr)
 \leq c_{t_0,M_0}C_{t_0,M_0}^2\varepsilon_{w,\vartheta,N}.
\end{equation}
Since $J_u(0)=[B,A]$, integration from zero to $t$ implies
\begin{equation}
    \left|\Tr P_{N/2}H(t)P_{N/2}-t\Tr P_{N/2}[B,A]P_{N/2}\right| = \left|t \int_0^1\int_0^t\Tr P_{N/2} J_u'(s)P_{N/2} ds du\right| \leq  c_{t_0,M_0}C_{t_0,M_0}^2 |t|^2 \varepsilon_{w,\vartheta,N},
\end{equation}
which proves \eqref{eq:commutatorApprox}. 

For the rest of the statement, note that $R(0) = \mathrm{Id}$. Hence similarly by \eqref{eq:R1} and Lemma~\ref{lemma:commutator1} we have 
\begin{equation}
    \|P_{N/2}R(t)Q_{N/2}\|_1\leq \int_0^t\|P_{N/2} R'(s)Q_{N/2}\|_1ds \leq |t|^2 m_{t_0,M_0}'\varepsilon_{w,\vartheta,N},
\end{equation}
for some uniform constant $m_{t_0,M_0}'>0$. Similarly we have the same upper bound for $\|Q_{N/2}R(t)P_{N/2}\|_1$. By \eqref{eq:H} and Lemma~\ref{lemma:commutator1}, $ \|P_{N/2}H(t)Q_{N/2}\|_1\leq |t| m_{t_0,M_0}'\varepsilon_{w,\vartheta,N}$, which concludes the proof. 
\end{proof}

\begin{proposition}\label{prop:logdetOrderBound} Fix $t_0>0$.
Suppose that $A,B\in\mathbb C^{N\times N}$
satisfy \eqref{eq:ass:abstract:improvedCT} and
\eqref{eq:ass:abstract:diagonal}. Then there exist constants $\varepsilon_0>0$ and $C_{t_0,m_0}>0$ 
depending only on $t_0,M_0,C_0,d_0$ and $b$, such that, whenever $\varepsilon_{w,\vartheta,N}\leq\varepsilon_0$
the matrix
\begin{equation}
    F(t)\coloneqq
Q_{N/2}
+P_{N/2}e^{-tA}e^{t(A+B)}e^{-tB}P_{N/2}
\end{equation}
is invertible for every $|t|\leq t_0$, and 
 \begin{equation}
        \left|\log \det F(t)+\frac{t^2}{2}\Tr P_{N/2}[A,B]P_{N/2}\right| \leq C_{t_0,m_0}(|t|^3\varepsilon_{w,\vartheta,N}+|t|^4\varepsilon_{w,\vartheta,N}^2),
    \end{equation}
 
\end{proposition}
\begin{proof} Set $R(t)$ and $H(t)$ as in Lemma~\ref{lemma:RH}. Note that 
\begin{equation}
F(t)\left(Q_{N/2}+P_{N/2}R(t)^{-1}P_{N/2}\right) = \mathrm{Id} - P_{N/2}R(t)Q_{N/2}R(t)^{-1}P_{N/2}.
\end{equation}
Then by \eqref{eq:RHApprox} in Lemma~\ref{lemma:RH} we can choose $\varepsilon_0>0$, depending only on
$t_0,M_0,C_0,d_0$ and $b$, sufficiently small that $C_{t_0}\varepsilon_0\leq1/2$.
Then for all $|t|<t_0$, we have the matrix $F(t)$ is invertible and
\begin{equation}
    \|F(t)^{-1}\|_\infty\leq \frac{\|R(t)^{-1}\|_\infty}{1-\|R(t)\|_\infty |t|^2 m_{t_0}\varepsilon_{w,\vartheta,N}}\leq c_{t_0}.
\end{equation}

Use the derivative of log determinant 
\begin{equation}
    \frac{d}{dt}\log \det F(t) = \Tr F'(t)F(t)^{-1} 
    = \Tr P_{N/2}H(t)P_{N/2} + \Tr P_{N/2}H(t)Q_{N/2}R(t)P_{N/2}F(t)^{-1}
\end{equation}
where the second equality uses the definition of $F(t)$ and $H(t)$. Then by \eqref{eq:RHApprox} in Lemma~\ref{lemma:RH} we have
 \begin{equation}
     \left|\Tr P_{N/2}H(t)Q_{N/2}R(t)P_{N/2}F(t)^{-1}\right|\leq |t|^3c_{t_0}m_{t_0,M_0}^2\varepsilon_{w,\vartheta,N}^2 .
 \end{equation}
 Together  with \eqref{eq:commutatorApprox} in Lemma~\ref{lemma:RH} we have 
 \begin{equation}
    \left|\frac{d}{dt}\log \det F(t) +t\Tr P_{N/2}[A,B]P_{N/2} \right| \leq |t|^2m_{t_0,M_0}\varepsilon_{w,\vartheta,N} + |t|^3  c_{t_0} m_{t_0,M_0}^2\varepsilon_{w,\vartheta,N}^2.
\end{equation}
Then integrating above from $0$ to $t$ implies 
\begin{equation}
    \left|log \det F(t) +\frac{t^2}{2}\Tr P_{N/2}[A,B]P_{N/2}\right|\leq |t|^3m_{t_0,M_0}\varepsilon_{w,\vartheta,N} + |t|^4  c_{t_0} m_{t_0,M_0}^2\varepsilon_{w,\vartheta,N}^2.
\end{equation}
Then we can define the universal constant $C_{t_0,M_0} = \max\{m_{t_0,M_0},  c_{t_0} m_{t_0,M_0}^2\}$ to conclude the Proposition. 
\end{proof}

\begin{corollary}\label{corollary: varianceApprox}
    \begin{equation}
        \left|\Tr P_{N/2}[A,B]P_{N/2} - \sum_{k=1}^{N-1}k(u_k^Bu_{-k}^A-u_k^Au_{-k}^B)\right| \leq c\frac{\vartheta  N}{w}+ \frac{C_0^2}{d_0^2}e^{-d_0N/w}.
    \end{equation}
\end{corollary}
\begin{proof}
    By Lemma~\ref{lemma:communatorDiff}, it is sufficient to compute 
    \begin{equation}
        \Tr P_{N/2}[T_A,T_B]P_{N/2} = \sum_{i=1}^{N/2}\sum_{j=N/2+1}^N u_{i-j}^Au_{j-i}^B-u_{i-j}^Bu_{j-i}^A.
    \end{equation}
 Then, taking $k=j-i$ and relabelling the sum above, we get 
     \begin{equation}\label{eq:traceCommuntatorProof1}
        \Tr P_{N/2}[T_A,T_B]P_{N/2} 
        = \sum_{k=1}^{N-1}k \left( u_{-k}^Au_{k}^B-u_{-k}^Bu_{k}^A \right)  +  \sum_{k=N/2+1}^{N-1} (N-2k)\left( u_{-k}^Au_{k}^B-u_{-k}^Bu_{k}^A \right) .
    \end{equation}
    Then by assumption \eqref{eq:ass:abstract:improvedCT} we can bound the second sum on the right-hand side of \eqref{eq:traceCommuntatorProof1} to be 
    \begin{equation}
      \sum_{k=N/2+1}^{N-1} 2|N-2k| \frac{C_0^2}{w^2}e^{-2d_0k/w} = \sum_{j=1}^{N/2-1} 4 j \frac{C_0^2}{w^2}e^{-2d_0(N/2+j)/w}\leq \frac{C_0^2}{d_0^2}e^{-d_0N/w},
    \end{equation}
    where for the inequality above we use the estimate $\sum_{j=1}^\infty \frac{j}{x^2}e^{-2d_0j/x} = \frac{d}{d x}\frac{1}{2d_0}\sum_{j=1}^\infty e^{-2d_0j/x}\leq \frac{1}{4d_0^2}$.
\end{proof}

In \cite{breuer2024mesoscopic},  a generalisation of
the Helton–Howe–Pincus formula is used to obtain the central limit theorem for the mesoscopic bulk case. This generalisation was first proved by Ehrhardt \cite{ehrhardt2003generalization}. That is for bounded operators $A, B$ such that the commutator $[A,B]$ is trace class, then
    \begin{equation}\label{eq:ehrhardt}
        \det \left(e^{-A}e^{A+B}e^{-B}\right) = e^{-\frac{1}{2}\Tr[A,B]}.
    \end{equation}
    However, for finite dimensional matrices $A$ and $B$, the formula \eqref{eq:ehrhardt} is equivalent the constant $1$. Indeed, Proposition~\ref{prop:logdetOrderBound} reveals that the boundary terms caused by the projection $P_{N/2}$ result in a non-trivial limit as $\varepsilon_{w,\vartheta,N}\to 0$. 

\subsection{Proof of Proposition~\ref{thm: overview strong szego}}\label{sec proof overview strong szego}
Fix $m\geq2$ and define
\begin{equation} \label{def:T}
    T \coloneqq \sum_{r=1}^{2M}c_rT_{n\pm2mn^\beta}(\eta_r)=T_-+T_+,
\end{equation}
where $T_-$ is strictly upper triangular and $T_+$ is lower triangular.

\begin{corollary} \label{cor:detExpansion} As $n\to\infty$ uniformly for all $t$ in a small disc, we have
    \begin{equation}
         e^{-tn^{-\alpha}\Tr(P_nTP_n)}
 \det\left(\mathrm{Id}+P_n(e^{tn^{-\alpha}T}-\mathrm{Id})P_n\right) = 
 \exp{\left(-\frac{t^2}{2n^{2\alpha}}\Tr P_n[T_+,T_-]P_n + |t|^3o(1)\right)}.
    \end{equation}
\end{corollary}
\begin{proof}
Since $T_+$ is lower triangular, we have $P_ne^{T_+}P_n = P_ne^{T_+}$ and similarly, $P_ne^{T_-}P_n = e^{T_-}P_n$. Then
\begin{equation}\label{eq:detFactorisedProof}
    e^{-tn^{-\alpha}\Tr(P_nTP_n)}
 \det\left(\mathrm{Id}+P_n(e^{tn^{-\alpha}T}-\mathrm{Id})P_n\right) = 
 \det\left(\mathrm{Id}+P_n(e^{-tn^{-\alpha}T_+}e^{tn^{-\alpha}T}e^{-tn^{-\alpha}T_-}-\mathrm{Id})P_n\right).
\end{equation}

The entry estimates of Proposition~\ref{prop: overview tri-inverse} give
\begin{equation*}
    n^{-\alpha}|T_{j,k}|\leq C n^{-\alpha/2}
       e^{-d_0|j-k|/n^{\alpha/2}},\quad  n^{-\alpha}|T_{j,k}-T_{j+h,k+h}|\leq \frac{\vartheta}{n^{\alpha/2}} e^{-d_0|j-k|/n^{\alpha/2}} ,
\end{equation*}
where $\delta =o(n^{-\beta+\alpha/2})$, uniformly for the indices in the consideration. To apply Proposition~\ref{prop:logdetOrderBound}, we need to reorganise the coordinates. Restrict to $\{n-2mn^\beta+1,\ldots,n+2mn^\beta\}$ and re-index from $1$ to $N$. Then apply Proposition~\ref{prop:logdetOrderBound} with
\begin{equation}\label{eq:s7-application-AB}
 N=4mn^\beta,\qquad w=n^{\alpha/2},\qquad
 A=n^{-\alpha}T_+,\qquad B=n^{-\alpha}T_-,
\end{equation}
and with $\varepsilon_{w,\vartheta,N} = o(1)$.  Then we have \eqref{eq:detFactorisedProof} is estimated to be as $n\to\infty$ uniformly for $t$ in a small disc, 
\begin{equation}
    \exp{\left(-\frac{t^2}{2n^{2\alpha}}\Tr P_n[T_-,T_+]P_n + O(|t|^3\varepsilon_{w,\vartheta,N})\right)}.
\end{equation}
\end{proof}

\begin{lemma}\label{lemma: variance}
      Consider the test function $f(x)=\sum_{r=1}^{2M}\frac{c_r}{x-\eta_r}$ with $c_j,\eta_r$ defined as \eqref{eq:resolvent fun}. Then we have 
      \begin{equation}
        -\lim_{n\to\infty}\frac{1}{n^{2\alpha}}\Tr P_n[T_+,T_-]P_n= \sigma_f^2 \equiv \begin{cases}\frac{1}{8\pi^2}\int\int_{\mathbb{R}^2}\left( \frac{f(x^2)-f(y^2)}{x-y} \right)^2dxdy & \quad \text{left edge,}\\
             \frac{1}{8\pi^2}\int\int_{\mathbb{R}^2}\left( \frac{f(-x^2)-f(-y^2)}{x-y} \right)^2dxdy & \quad \text{right edge.}
        \end{cases}
      \end{equation}
    \end{lemma}
    \begin{proof}
      Let $N_0 = n-2mn^\beta+1$. By Corollary~\ref{corollary: varianceApprox} we have       \begin{equation}\label{eq:variance limit}
          -\frac{1}{n^{2\alpha}}\Tr P_n[T_+,T_-]P_n =\frac1{n^{2\alpha}}\sum_{r,s=1}^{2M}c_rc_s
 \sum_{k=1}^{4mn^\beta-1}k
 \left(T_{n\pm2mn^\beta}(\eta_r)\right)_{k+N_0,N_0}
 \left(T_{n\pm2mn^\beta}(\eta_s)\right)_{N_0,k+N_0} + o(1).
      \end{equation}
			
      We estimate each entry of the three-fold sum \eqref{eq:variance limit}, by Proposition~\ref{prop: overview tri-inverse}, that is for any $r,s$ and $j,l$
\begin{equation}
\left( T_{n\pm 2mn^\beta}\left(\eta_r \right) \right)_{j,l} 	\left( T_{n\pm 2mn^\beta}\left(\eta_{s} \right) \right)_{l,j}
= \frac{n^{\alpha}(1+o(1))}{4a_{n,n}^2\mysqrt{\frac{-\eta_r}{|a_{n,n}|}}\mysqrt{\frac{-\eta_s}{|a_{n,n}|}}}\prod\limits_{i=j}^{l-1}\left( 1-\mysqrt{-\frac{\eta_r}{n^\alpha |a_{n,n}|}}-\mysqrt{-\frac{\eta_s}{n^\alpha |a_{n,n}|}} +\xi_{i}^{(r,s)}\right),
\end{equation}
as $n\to\infty$, where $\xi_{i}^{(r,s)}\coloneqq \xi_{i}^{(r)}+\xi_{i}^{(s)}+\xi_{i}^{(r)}\xi_{i}^{(s)}-\xi_{i}^{(r)}\mysqrt{-\frac{\eta_s}{n^\alpha |a_{n,n}|}}-\xi_{i}^{(s)}\mysqrt{-\frac{\eta_r}{n^\alpha |a_{n,n}|}}+\frac{\mysqrt{-\eta_r}\mysqrt{-\eta_s}}{|a_{n,n}|n^{\alpha}}=o(n^{-\frac{\alpha}{2}})$.

Note that for all $0<\frac{\alpha}{2}<\beta$ we have $\left( 1-zn^{-\frac{\alpha}{2}} \right)^{n^\beta} = O(e^{-\Re z n^{\beta-\frac{\alpha}2}})$ for any $\Re z> 0$, which is exponentially small. Hence, adding exponentially small terms does not change the sum in the limit, i.e., as $n\to\infty$,
\begin{multline}\label{eq:variance estimate 1 }
 \frac{1}{n^{2\alpha}}\Tr P_n[T_+,T_-]P_n 
=\frac{(1+o(1))}{4n^\alpha a_{n,n}^2\mysqrt{\frac{-\eta_r}{|a_{n,n}|}}\mysqrt{\frac{-\eta_s}{|a_{n,n}|}}}\sum_{r,s=1}^{2M}c_rc_s
\sum_{k=1}^{\infty}k\left( 1-\mysqrt{-\frac{\eta_r}{ n^\alpha |a_{n,n}|}}-\mysqrt{-\frac{\eta_s}{n^\alpha |a_{n,n}|}} \right)^k +o(1)\\ =\sum_{r,s=1}^{2M}c_rc_s\frac{1}{4\mysqrt{-\eta_r}\mysqrt{-\eta_s}\left( \mysqrt{-\eta_r}+\mysqrt{-\eta_s} \right)^2}+o(1).
\end{multline}
      
The rest is to use the residue theorem to recover the double integral formula. By definition of $f$ as \eqref{eq:resolvent fun} 
      \begin{equation*}
        \int \int_{\mathbb{R}^2}\left(\frac{f(x^2)-f(y^2)}{x-y}\right)^2dxdy=\sum_{r,s}c_rc_s \int \int_{\mathbb{R}^2}\left(\frac{\frac{1}{x^2-\eta_r}-\frac{1}{y^2-\eta_r}}{x-y}\right)\left(\frac{\frac{1}{x^2-\eta_s}-\frac{1}{y^2-\eta_s}}{x-y}\right)dxdy.
      \end{equation*}
      Each of the double integrals can then be rewritten in the following way 
      \begin{multline*}
        \int \int_{\mathbb{R}^2}\left(\frac{\frac{1}{x^2-\eta_r}-\frac{1}{y^2-\eta_r}}{x-y}\right)\left(\frac{\frac{1}{x^2-\eta_s}-\frac{1}{y^2-\eta_s}}{x-y}\right)dxdy \\
        =  2\int_{\mathbb{R}}\frac{x^2}{(x^2-\eta_r)(x^2-\eta_s)}dx\int_{\mathbb{R}}\frac{1}{(y^2-\eta_r)(y^2-\eta_s)}dy+
        2\left(\int_{\mathbb{R}}\frac{x}{(x^2-\eta_r)(x^2-\eta_s)}dx\right)^2.
      \end{multline*}
      Apply the residue theorem to the three integrals above. On the principal square root branch, we have $\Im(i\mysqrt{-\eta_r})>0$. So each integral has two residues to consider i.e. $i\mysqrt{-\eta_r}$ and $i\mysqrt{-\eta_s}$. Hence, 
      \begin{align*}
        \frac{1}{2\pi i}\int_\mathbb{R}\frac{x^2}{(x^2-\eta_r)(x^2-\eta_s)}dx = & \frac{1}{2i(\mysqrt{-\eta_r}+\mysqrt{-\eta_s})}, \\
        \frac{1}{2\pi i} \int_{\mathbb{R}}\frac{1}{(y^2-\eta_r)(y^2-\eta_s)}dy = &\frac{1}{2i\mysqrt{-\eta_r}\mysqrt{-\eta_s}(\mysqrt{-\eta_r}+\mysqrt{-\eta_s})},\\
        \frac{1}{2\pi i}\int_{\mathbb{R}}\frac{x}{(x^2-\eta_r)(x^2-\eta_s)}dx = & 0.
      \end{align*}
      
      Consequently,
      \begin{equation*}
        \frac{1}{(2\pi i)^2}\int \int_{\mathbb{R}^2}\left(\frac{f(x^2)-f(y^2)}{x-y}\right)^2dxdy = \sum_{r,s=1}^{2M}c_rc_s\frac{-1}{2\mysqrt{-\eta_r}\mysqrt{-\eta_s}\left( \mysqrt{-\eta_r}+\mysqrt{-\eta_s} \right)^2}.
      \end{equation*}

      Combining \eqref{eq:variance estimate 1 } and Lemma \ref{lemma: finner est omega}, we conclude that, at the left edge,
      
      \begin{equation}
        \lim_{n\to\infty}\frac{1}{n^{2\alpha}}\Tr P_n[T_+,T_-]P_n = -\frac{1}{8\pi^2}\int\int_{\mathbb{R}^2}\left( \frac{f(x^2)-f(y^2)}{x-y} \right)^2dxdy.
      \end{equation}

      Note that at the right edge instead of \eqref{eq:omega exact} we have 
      \begin{align*}
        \omega_l^{(r)-}=\frac{x_0-b_{l,n}+\frac{\eta_r}{n^\alpha}-\mysqrt{\left( x_0-b_{l,n}+\frac{\eta_r}{n^\alpha} \right)^2-4a_{l-1,n}a_{l,n}}}{-2a_{l-1,n}}.
      \end{align*}
      Then by the same proof of Lemma \ref{lemma: finner est omega} we get 
      as $n\to\infty$, for $l\in\{n-2mn^{\beta},n-2mn^{\beta}+1,\dots,n+2mn^{\beta} \}$, 
      \begin{align}
        \omega_l^{(r)-} =\left(  1-\mysqrt{\frac{\eta_r}{n^\alpha |a_{n,n}|}} \right) sgn(-a_{n,n})+o(n^{-\frac{\alpha}{2}}) .
      \end{align}
      
      We can now follow the same computation and obtain that 
      \begin{equation}
\lim_{n\to\infty}\frac{1}{n^{2\alpha}}\Tr P_n[T_+,T_-]P_n = \sum_{r,s=1}^{2M}c_rc_s\frac{1}{4\mysqrt{\eta_r}\mysqrt{\eta_s}\left( \mysqrt{\eta_r}+\mysqrt{\eta_s} \right)^2}.
      \end{equation}
      
      From here the rest is to use the residue theorem to obtain the double integral formula. Recall the formulation of $f$ and we have 
      \begin{equation*}
        \int \int_{\mathbb{R}^2}\left(\frac{f(-x^2)-f(-y^2)}{x-y}\right)^2dxdy=\sum_{r,s}c_rc_s \int \int_{\mathbb{R}^2}\left(\frac{\frac{1}{x^2+\eta_r}-\frac{1}{y^2+\eta_r}}{x-y}\right)\left(\frac{\frac{1}{x^2+\eta_s}-\frac{1}{y^2+\eta_s}}{x-y}\right)dxdy.
      \end{equation*}
      Similar to the computation of the left edge, we rewrite each of the double integrals into multiplications of single integrals with poles at $i\mysqrt{\eta_r}$ and $i\mysqrt{\eta_s}$,  on the principal square root branch. Apply the residue theorem and we obtain
      \begin{align*}
        \frac{1}{(2\pi i)^2}\int \int_{\mathbb{R}^2}\left(\frac{f(-x^2)-f(-y^2)}{x-y}\right)^2dxdy = \sum_{r,s=1}^{2M}c_rc_s\frac{-1}{2\mysqrt{\eta_r}\mysqrt{\eta_s}\left( \mysqrt{\eta_r}+\mysqrt{\eta_s} \right)^2}.
      \end{align*}

      Combining \eqref{eq:variance estimate 1 } and Lemma \ref{lemma: finner est omega}, we get 
      
      \begin{align}
        \lim_{n\to\infty}\frac{1}{n^{2\alpha}}\Tr P_n[T_+,T_-]P_n = - \frac{1}{8\pi^2}\int\int_{\mathbb{R}^2}\left( \frac{f(-x^2)-f(-y^2)}{x-y} \right)^2dxdy.
      \end{align}		
    \end{proof}

    We are now ready to prove Proposition~\ref{thm: overview strong szego}.
      \begin{proof}[Proof of Proposition~\ref{thm: overview strong szego}]
        Recall that we consider $f$ defined as \eqref{eq:resolvent fun} and $F\coloneqq \sum_{r=1}^{2M}c_r\left( \mathcal{J}-x_0- \frac{\eta_r}{n^{\alpha}} \right)^{-1}$.

        Proposition \ref{prop: overview trimming} shows that the first $m$ coefficients in the expansion around small $t$ of below
        \begin{equation}\label{eq:Fredholm det 1}
          \log\mathbb{E}[\exp\left(tX_{f,\alpha,x_0}-t\mathbb{E}[X_{f,\alpha,x_0}]\right)] = \log \left( \det \left(\mathrm{Id} +P_n\left(e^{tn^{-\alpha}F}-\mathrm{Id} \right)P_n\right)e^{-tn^{-\alpha} \Tr \left( P_n F \right)} \right)
        \end{equation}
        are asymptotically the same as those of the following 
        \begin{multline}\label{eq:Fredholm det 2}
          \log \det \left(\mathrm{Id} +P_n\left(e^{t n^{-\alpha}\sum_{r}c_rT_{n\pm 2mn^\beta}\left(\eta_r \right)}-\mathrm{Id} \right)P_n\right)e^{-t n^{-\alpha} \Tr P_n \sum_{r}c_rT_{n\pm 2mn^\beta}\left(\eta_r \right)} \\=\sum_{k\geq 2} \frac{t^k}{k!}\mathcal{C}_k^{(n)}\left(\sum_{r}c_rn^{-\alpha}T_{n\pm 2mn^\beta}\left(\eta_r \right)\right).
        \end{multline}
        By Corollary~\ref{cor:detExpansion} and Lemma~\ref{lemma: variance}, we see that, as $n\to\infty$ 
        \begin{equation}
          \lim_{n\to\infty}\mathcal{C}_k^{(n)}\left(\sum_{r}c_rn^{-\alpha}T_{n\pm 2mn^\beta}\left(\eta_r \right)\right)
          =\begin{cases}
           \sigma_f^2, \quad &k=2\\
            0, \quad &k=3,\dots m.
          \end{cases}
        \end{equation}
        
        The consequence of Lemma~\ref{lemma: dominated} implies that the expansions of both of the right-hand sides of \eqref{eq:Fredholm det 1} and \eqref{eq:Fredholm det 2} are absolutely convergent. We conclude that, via the dominated convergence theorem,
        for any $0<\frac{\alpha}{2}<\beta<\frac{\alpha+1}{3}$ 
        around the edges, i.e. $x_0=b_{n-1,n} \pm 2\sqrt{a_{n,n}a_{n-1,n}}+o(n^{-\alpha})$
        \begin{equation}
          \lim_{n\to\infty}	\log\mathbb{E}\left[\exp\left(tX_{f,\alpha,x_0}^{(n)}-t\mathbb{E}[X_{f,\alpha,x_0}^{(n)}]\right)\right] =\frac{t^2\sigma_f^2}{2}. 
        \end{equation}        
        Then by the moments method,  $X_n(f_{\alpha,x_0})-\mathbb{E}[X_n(f_{\alpha,x_0})]$ has the Gaussian fluctuations in the limit. 
      \end{proof}
      
      Hence, together with Subsection~\ref{sec:Lipschitz extension} and we complete the proof of Theorem~\ref{thm: main 2}.

    \section{Proof of Theorems ~\ref{thm: regular}, ~\ref{thm: modify Jacobi}, ~\ref{thm: main 1}}\label{sec:proof of 123}
      We will use Theorem~\ref{thm: main 2} to prove Theorems ~\ref{thm: modify Jacobi} and ~\ref{thm: main 1}. We will prove Theorem~\ref{thm: regular} by Theorem~\ref{thm: main 1}.

      \begin{proof}[Proof of Theorem~\ref{thm: main 1} ]
        For $0<\alpha<\frac{2}{3}$, condition~\eqref{ass:slowly varying recurrence coe} implies that for all $j\in I^{(\alpha,\varepsilon)}_n$, as $n\to\infty$,
        \begin{equation*}
          a_{j,n}a_{j-2,n}-a_{j-1,n}^2 = O(n^{-1}), \quad
          (b_{j-1,n}-x_0-a_{j,n})a_{j-2,n}-(b_{j-2,n}-x_0-a_{j-1,n})a_{j-1,n}=O(n^{-1}).
        \end{equation*}
        Hence the conditions in Theorem~\ref{thm: main 2} are satisfied. Then apply Theorem~\ref{thm: main 2} and we conclude the proof of Theorem~\ref{thm: main 1}.
      \end{proof}
      
      \begin{proof}[Proof of Theorem~\ref{thm: modify Jacobi} ]
        
         The asymptotics of recurrence coefficients of the modified Jacobi polynomials with measure \eqref{eq:modified Jacobi} are given by the Riemann-Hilbert method in \cite{kuijlaars2004riemann} to be
        \begin{equation} \label{eq:recurrence modified Jacobi}
          a_{j} = \frac{1}{2} - \left( \frac{4\gamma_1^2-1}{32} + \frac{4\gamma_2^2-1}{32}  \right)\frac{1}{j^2} +O(j^{-3}), \quad 	b_{j} =  \frac{\gamma_2^2-\gamma_1^2}{4j^2}  +O(j^{-3}), \quad \text{as } j\to\infty, 
        \end{equation}
        where $\gamma_1$ and $\gamma_2$ are some constants explicitly given in \cite{kuijlaars2004riemann}. 
        Note that there is no scaling in this model and, hence, we remove the second subscript $n$ in the notations. Hence, 
        \begin{align*}
          |a_{j}-a_{j-1}| = O(j^{-3}), \quad 	|b_{j}-b_{j-1}| = O(j^{-3}), \qquad \text{as } j\to\infty. 
        \end{align*}
        Then the conditions \eqref{eq:ass:recurrence 1} and \eqref{eq:ass:recurrence 2} are satisfied for all $0<\frac{\alpha}{2}<\beta<1$ .Hence the limit of mesoscopic fluctuations holds for all $\alpha\in(0,2)$. Then apply Theorem~\ref{thm: main 2}, take $\beta\to1$ and we conclude the proof of Theorem~\ref{thm: modify Jacobi}.
      \end{proof}
      
      \begin{proof}[Proof of Theorem~\ref{thm: regular}]
       Let $j\in I_{n}^{(\alpha,\varepsilon)}$, which is defined in \eqref{eq:index set}. Let $c=\frac{j}{n}$. Then we rewrite 
        \begin{align}
          e^{-nV(x)}dx = e^{-\frac{j}{c}V(x)}dx. 
        \end{align}
        Then the conditions on $V$ and positivity about $c$ imply the existence and uniqueness  of the equilibrium measure of $c^{-1}V $, denoted by $\mu_{c^{-1}V}$ by the potential theory. By the convexity of $c^{-1}V$, the support of $\mu_{c^{-1}V}$ is a single interval say $[\alpha_1,\alpha_2]$, which is determined by the following set of equations, 
        \begin{align}
          \int_{\alpha_1}^{\alpha_2}\frac{V'(s)}{\sqrt{(s-\alpha_1)(\alpha_2-s)}}ds=0, 
          \quad 	\int_{\alpha_1}^{\alpha_2}\frac{sV'(s)}{\sqrt{(s-\alpha_1)(\alpha_2-s)}}ds=2\pi c . \label{eq:Mhaskar-Rakhmanov-Saff}
        \end{align}
        In the literature $\alpha_1, \alpha_2$ are called the Mhaskar-Rakhmanov-Saff numbers, (cf. \cite{saff2013logarithmic} pages  203-234)
        Let $F:\mathbb{R}^3\to \mathbb{R}^2$ be a real analytic function such that  $	F(x,y,c) \coloneqq  \begin{pmatrix} \phi(x,y)\\ \varphi(x,y,c) \end{pmatrix}$ where
        \begin{align*}
          \phi(x,y) \coloneqq &	\int_{-1}^{1}V'\left(\frac{(y-x)t+x+y}{2}\right)\frac{dt}{\sqrt{1-t^2}}, \\
          \varphi(x,y,c) \coloneqq & \int_{-1}^{1}\frac{(y-x)t}{2}V'\left(\frac{(y-x)t+x+y}{2}\right)\frac{dt}{\sqrt{1-t^2}}-2\pi c . 
        \end{align*}
        Then we have $F(\alpha_1,\alpha_2,c) = \mathbf{0}$ by \eqref{eq:Mhaskar-Rakhmanov-Saff} and a change of variable argument. We further define 
        \begin{align*}
          d\mathcal{F}(t) \coloneqq V''\left(\frac{(\alpha_2-\alpha_1)t+\alpha_1+\alpha_2}{2}\right)\frac{dt}{\sqrt{1-t^2}}.
        \end{align*} 
        Then by strict convexity of $V$, we have $V''\geq0$ and does not vanish on any non-empty open interval. Thus $\mathcal{F}(t)$ defines a positive measure. Now computing the Jacobian of $F$ at $(\alpha_1,\alpha_2,c)$, we get  
        \begin{equation}\label{eq:IFT Jacobian}
          D_{(x,y)}F(\alpha_1,\alpha_2,c) 
          = \begin{pmatrix} \frac{1}{2}\int_{-1}^1\left( 1-t \right)d\mathcal{F}(t) ,& \frac{1}{2}\int_{-1}^1\left( 1+t \right)d\mathcal{F}(t)
            \\-\frac{2\pi c}{\alpha_2-\alpha_1}+\frac{\alpha_2-\alpha_1}{4}\int_{-1}^1t\left( 1-t \right)d\mathcal{F}(t), & \frac{2\pi c}{\alpha_2-\alpha_1} + \frac{\alpha_2-\alpha_1}{4}\int_{-1}^1t\left( 1+t \right)d\mathcal{F}(t) 
          \end{pmatrix}.
        \end{equation}
        The determinant of \eqref{eq:IFT Jacobian} can be computed as 
        \begin{multline*}
          \frac{4\pi c}{\alpha_2-\alpha_1}\int_{-1}^1d\mathcal{F}(t)\\
          +\left( \alpha_2-\alpha_1 \right)\left( \int_{-1}^1\left( 1-t \right)d\mathcal{F}(t) \int_{-1}^1t\left( 1+t \right)d\mathcal{F}(t) -\int_{-1}^1\left( 1+t \right)d\mathcal{F}(t) \int_{-1}^1t\left( 1-t \right)d\mathcal{F}(t)\right)\\
          =\frac{4\pi c}{\alpha_2-\alpha_1}\int_{-1}^1d\mathcal{F}(t)+2\left( \alpha_2-\alpha_1 \right)\left( \int_{-1}^1d\mathcal{F}(t) \int_{-1}^1t^2d\mathcal{F}(t) -\left( \int_{-1}^1t d\mathcal{F}(t)  \right)^2\right).
        \end{multline*}
        Use Cauchy-Schwartz inequality and we deduce $\left( \int_{-1}^1td\mathcal{F}(t) \right)^2\leq \int_{-1}^1t^2d\mathcal{F}(t)\int_{-1}^1d\mathcal{F}(t)$. Subsequently, the Jacobian matrix under consideration is strictly positive and hence non-singular. Then apply the Implicit Function Theorem to obtain that $\alpha_1$ and $\alpha_2$ are continuously differentiable functions of $c$, denoted by  
        \begin{align*}
          \alpha_1 = \alpha_1(c), \quad  \alpha_2 = \alpha_2(c). 
        \end{align*}  
        Moreover, one can apply the same analysis in \cite{deift1999uniform}  for $c^{-1}V$ and obtain a similar result in Application 1 (1.64) (1.65) in \cite{deift1999uniform} for the regular case (also refer to the proof of Theorem 2.1 in \cite{deift1999strong}). Note that in their notation in the one-cut case $\tilde{M}^{(\infty)}(z) $ found as (4.71) of  \cite{deift1999uniform} is reduced to
        \begin{equation*}
          \tilde{M}^{(\infty)}(z) = \frac{1}{2}\begin{pmatrix}
            \gamma(z)+\gamma(z)^{-1}, & i(\gamma(z)^{-1}-\gamma(z))\\
            i(\gamma(z)-\gamma(z)^{-1}), & \gamma(z)+\gamma(z)^{-1}
          \end{pmatrix},
        \end{equation*}
        which is equivalent to $N(z)$ defined in (6.13-6.16) of \cite{deift1999strong} after rescaling the support from $[\alpha_1,\alpha_2]$ to $[-1,1]$. 
        This means that the recurrence coefficients can be estimated as
        \begin{equation}\label{eq:RHP recurrence estimate}
          a_{j,n} =\frac{\alpha_2(c)- \alpha_1(c)}{4} + O(n^{-1}), \quad b_{j,n} =\frac{\alpha_2(c)+\alpha_1(c)}{2} + O(n^{-1}),
        \end{equation}
        uniformly for all $j\in I_{n}^{(\alpha,\varepsilon)}$. Recall that $c=\frac{j}{n}$ and $\alpha_1(c), \alpha_2(c)$ are continuously differentiable functions. Then for all $j\in I_{n}^{(\alpha,\varepsilon)} $ we have 
        \begin{equation*}
          |a_{j,n}-a_{j-1,n}|=O(n^{-1}), \quad 	|b_{j,n}-b_{j-1,n}|=O(n^{-1}).
        \end{equation*}
        Moreover, by the normalisation we have $\alpha_1(1)=-1, \alpha_2(1)=1$. Hence, 
        \begin{equation*}
          \lim_{j/n\to 1}a_{j,n} = \frac{1}{2}, \quad \lim_{j/n\to 1}b_{j,n} =0.
        \end{equation*}
        From here we obtain that the edges of the fluctuations are at $x_0=-1$ or $1$.  
        Now we can employ Theorem~\ref{thm: main 1} and conclude Theorem~\ref{thm: regular}.
      \end{proof}
      
      Note that the strict convexity is not a necessary condition for the Jacobian to be non-singular nor the equilibrium measure to be supported on a single interval. Though it can be relaxed, convexity simplifies the argument. Furthermore, analyticity is only needed for the Riemann-Hilbert technique to obtain the asymptotics of recurrence coefficients. However, it is highly possible that \eqref{eq:RHP recurrence estimate} still holds for $V$ having some degree of smoothness, without assuming analyticity. In the literature, the Riemann–Hilbert–$\bar{\partial}$ method is able to deal with V with two Lipschitz continuous derivatives, see \cite{mclaughlin2008steepest}. Their result Theorem 1(1) indicates that the recurrence coefficients should be slowly varying with $O(n^{-\frac{1}{3}}\log n)$, in a more general setup.  Here we believe that Theorem~\ref{thm: regular} should still hold for $V$ being a convex (or the equilibrium measure being supported in a single interval) and  $C^{2+\varepsilon}(\mathbb{R})$ (for some $\varepsilon>0$) potential satisfying \eqref{eq:potential infinity}.  This requires a more detailed study about the potential theory and we leave it open.

      \section{Proof of Theorem~\ref{thm: main 3} }\label{sec:proof of 3}
      Assumption \eqref{eq:ass recurrence perturbation} allows us to approximate the Jacobi matrix $\mathcal{J}$ by a Toeplitz matrix. Compared with Theorem~\ref{thm: main 2}, we have more control of the tails of the diagonals in this case. Hence, a direct comparison with a Toeplitz operator is possible by exploring the algebra of the cumulant formula. In this section we are to prove Theorem~\ref{thm: main 3} by fulfilling such an idea.
      
      For clarity, throughout this section we let
      \begin{equation}
        \beta=\frac{\alpha+\varepsilon}{2}, \quad \text{for some } \varepsilon>0 \quad  \text{such that } 0<\frac{\alpha}{2}<\beta<1.
      \end{equation}
      As mentioned in the remark of Theorem~\ref{thm: main 3}, we will give a proof with a weaker assumption \eqref{eq:ass recurrence perturbation}. Without loss of generality we further assume 
      
        \begin{equation}.\label{eq:ass perturbation recurrence 1}
        \sup\limits_{j\geq n-n^{\beta+\varepsilon/2}}|a_{j,n} -1|=O(n^{-2\beta}),\quad \sup\limits_{j\geq n-n^{\beta+\varepsilon/2}}|b_{j,n} |= O(n^{-2\beta}).
      \end{equation}

      The general case follows by translating and scaling the point process by $(x-b)/a$, where $a\coloneqq\lim_{n}a_{n,n}$ and $b\coloneqq\lim_{n}b_{n-1,n}$.
      
      \subsection{Resolvent of a Toeplitz Operator}\label{sec:thm3 Toeplitz}
      
      We will first consider a block semi-infinite matrix where the first block has the same entries as $(a_{j,n}, b_{j,n})$ for all $j < n-n^{\beta+\varepsilon/2}$, and the second block has constant entries $(1,0)$, i.e., 
      
      \begin{equation} \label{eq:assumption_free_J_inverse}
        \widetilde{\mathcal{J}}\coloneqq P_{n-n^{\beta+\varepsilon/2}} \mathcal{J}P_{n-n^{\beta+\varepsilon/2}} +Q_{n-n^{\beta+\varepsilon/2}} \begin{pmatrix}0 & 1& & \\1& 0 & 1 &   \\&1 &0& 1&  \\ & &\ddots &\ddots &\ddots \end{pmatrix}Q_{n-n^{\beta+\varepsilon/2}} .
      \end{equation}
    
      Take $x_0=-2$ and define 
      \begin{equation}
        \omega_+\coloneqq \frac{2-\frac{\eta}{n^\alpha} +\mysqrt{\left( 2-\frac{\eta}{n^\alpha}  \right)^2-4}}{2}, \quad 	\omega_-\coloneqq \frac{2-\frac{\eta}{n^\alpha} -\mysqrt{\left( 2-\frac{\eta}{n^\alpha}  \right)^2-4}}{2},
      \end{equation}
      where the square root is taken on the principal branch so that $|\omega_+|>|\omega_-|$. The right edge ($x_0=2$) can be done in the exact same way by swapping the definitions of $\omega_+$ and $\omega_-$. 
      
      Since the second block of  $\widetilde{\mathcal{J}}$ is a Toeplitz operator, we can use the Wiener-Hopf factorisation to obtain the exact inverse formula entry-wise, for all $j,k\geq n-n^{\beta+\varepsilon/2}+1$
      
      \begin{equation}\label{eq:free J inverse}
        \left( \left(  \widetilde{\mathcal{J}}-\left( x_0+\frac{\eta}{n^\alpha}  \right)\mathrm{Id} \right)^{-1} \right)_{j,k} = \frac{\left(-\omega_+\right)^{-|j-k|}-\left(-\omega_-\right)^{j+k-2(n-n^{\beta+\varepsilon/2}+1)}}{\omega_+-\omega_-}.
      \end{equation}
      One should compare this formula with Proposition~\ref{prop: overview tri-inverse}.
      
      Note that 
      \begin{equation}\label{eq:free J omega approx}
        \omega_+ = 1+\mysqrt{-\frac{\eta}{n^{\alpha}}} + O(n^{-\alpha}), \quad 	\omega_- = 1-\mysqrt{-\frac{\eta}{n^{\alpha}}} + O(n^{-\alpha}).
      \end{equation}
      From \eqref{eq:free J inverse} and \eqref{eq:free J omega approx}, we see that for all $j,k\geq n-n^{\beta+\varepsilon/2}+1$
      \begin{equation}
        \left| 	\left( \left(  \widetilde{\mathcal{J}}-\left( x_0+\frac{\eta}{n^\alpha}  \right)\mathrm{Id} \right)^{-1} \right)_{j,k}  \right| \leq C n^{\frac{\alpha}{2}}e^{-dn^{-\frac{\alpha}{2}}|j-k|}. \label{eq:free Combes Thomas}
      \end{equation}
      
      One convenient consequence of this estimate is the following lemma. 
      \begin{lemma}\label{lemma: HS norm}
        There exist a constant $C_0>0$, such that for any integers $m_1, m_2\geq n-n^{\beta+\varepsilon/2}+1$, we have 
        \begin{equation}
          \left\|P_{m_1}Q_{m_2}  \left(  \widetilde{\mathcal{J}}-\left( x_0+\frac{\eta}{n^\alpha}  \right)\mathrm{Id} \right)^{-1}  \right\|_2\leq \sqrt{C_0 n^{\frac{3\alpha}{2}}|m_1-m_2|}.
        \end{equation} 
      \end{lemma}
      \begin{proof}
        For $m_1\leq m_2$, $P_{m_1}Q_{m_2} =0$ and the norm on left-hand side is zero. The inequality holds trivially. We now consider the case where $m_1>m_2$. 
        
        Recall for any linear Hilbert-Schmidt operator $A$, the Hilbert-Schmidt norm  $\|A\|_2^2= \sum_{j,k}|(A)_{j,k}|^2$. Then we can estimate by \eqref{eq:free Combes Thomas}
        
        \begin{multline}\label{eq:thm3 lemma 2 norm 1}
          \left\|P_{m_1}Q_{m_2}  \left(  \widetilde{\mathcal{J}}-\left( x_0+\frac{\eta}{n^\alpha}  \right)\mathrm{Id} \right)^{-1}  \right\|_2^2\leq \sum_{j=m_2+1}^{m_1}\sum_{k=1}^\infty C^2 n^{\alpha}e^{-2dn^{-\frac{\alpha}{2}}|j-k|} \\
          \leq \sum_{j=m_2+1}^{m_1}\sum_{l\in \mathbb{Z}} C^2 n^{\alpha}e^{-2dn^{-\frac{\alpha}{2}}|l|}	\leq 2\sum_{j=m_2+1}^{m_1}\sum_{l=0}^\infty C^2 n^{\alpha}e^{-2dn^{-\frac{\alpha}{2}}l} =\frac{2C^2 n^{\alpha}(m_1-m_2)}{1-e^{-2dn^{-\frac{\alpha}{2}}}}, 
        \end{multline}
        where the second line is by adding more terms in the sum to make the series a power series. Note that $1-e^{-x}\geq e^{-1}x$ for all $x\in[0,1]$, we can further estimate \eqref{eq:thm3 lemma 2 norm 1} to be  
        \begin{equation}
          \left\|P_{m_1}Q_{m_2}  \left(  \widetilde{\mathcal{J}}-\left( x_0+\frac{\eta}{n^\alpha}  \right)\mathrm{Id} \right)^{-1}  \right\|_2^2\leq \frac{2eC^2}{d} n^{\frac{3\alpha}{2}}(m_1-m_2).
        \end{equation}
        This completes the proof.
      \end{proof}
      
      \subsection{Lemmas of Proof of Theorem~\ref{thm: main 3}}
      
      We consider test function $f(x)=\sum_{r=1}^{2M}c_r \frac{1}{x-\eta_r}$ as described in \eqref{eq:resolvent fun}  for some $\eta\in\{x+iy | x\in\mathbb{R}, y \neq 0\}$. Let us define
      
      \begin{equation}\label{eq:def F tilde}
        \widetilde{F}\coloneqq \sum_{r=1}^{2M}c_r \left( \widetilde{\mathcal{J}}-x_0-\frac{\eta_r}{n^{\alpha}} \right)^{-1}.
      \end{equation}
      
      Denote the Jacobi matrix of the OPE with recurrence coefficients with $\{a_{j,n}, b_{j,n}\}$ satisfying \eqref{eq:ass perturbation recurrence 1} to be $\mathcal{J}$. Define 
      \begin{equation}\label{eq:def F}
        F\coloneqq \sum_{r=1}^{2M}c_r \left( \mathcal{J}-x_0-\frac{\eta_r}{n^{\alpha}} \right)^{-1}.
      \end{equation}
      
      \begin{lemma}\label{lemma: free exp} Let $0<\frac{\alpha}{2}<\beta<1$. Let $l\in\mathbb{N}$. Let $m_1\in \mathbb{N}$ such that $m_1>n-n^{\beta+\varepsilon/2}+1 $.  Then there exists a $C_l>0$ such that
        \begin{equation}
          \left\|Q_{m_1+ln^\beta}\widetilde{F}^{l}P_{m_1}\right\|_1\leq  C_l e^{-d'n^{\beta-\frac{\alpha}{2}}} ,\label{eq:free ess_banded}
        \end{equation}
        \begin{equation}
          \left\|P_{m_1-ln^\beta}\widetilde{F}^{l}Q_{m_1}\right\|_1 \leq C_l e^{-d'n^{\beta-\frac{\alpha}{2}}} \label{eq:free ess_banded2}
        \end{equation}
      \end{lemma}
      \begin{proof}
        Repeat the proof of \eqref{eq:ess_banded} in Lemma~\ref{cutoff} taking $N=\infty$, using \eqref{eq:free Combes Thomas} and one gets \eqref{eq:free ess_banded}. Repeat the proof of \eqref{eq:ess_banded2} in Lemma~\ref{cutoff 2} summing $j$ from $1$ and $k$ up to $\infty$, using \eqref{eq:free Combes Thomas} and one gets \eqref{eq:free ess_banded2}. 
      \end{proof}
      
      \begin{lemma} \label{lemma: thm3 step0}
        Let $0<\frac{\alpha}{2}<\beta<1$. Assume \eqref{eq:ass perturbation recurrence 1} holds for all $j>m_2-2n^\beta$. For $m_1,m_2\in \mathbb{N}$ such that $m_2\geq n-n^{\beta+\varepsilon/2} +2n^\beta+2$, $2n^\beta<m_2<m_1$ and $m_1-m_2=O(n^{\beta})$, we have, as $n\to\infty$
        \begin{equation}
          \left\| P_{m_1}\left( \widetilde{F}-F \right)Q_{m_2} \right\|_1=O\left( n^{\frac{3\alpha}{2}-\beta} \right).
        \end{equation}
      \end{lemma}
      \begin{proof}
        Let us define
        \begin{align}
          \widetilde{F}^{(r)} \coloneqq &\left( \widetilde{\mathcal{J}}-x_0-\frac{\eta_r}{n^{\alpha}} \right)^{-1}\\
          F^{(r)} \coloneqq &  \left( \mathcal{J}-x_0-\frac{\eta_r}{n^{\alpha}} \right)^{-1}.
        \end{align}
        By triangle inequality we have 
        \begin{equation}\label{eq:thm3 lemma 1}
          \left\| P_{m_1}\left( \widetilde{F}-F \right)Q_{m_2} \right\|_1\leq \sum_r|c_r|\left\| P_{m_1}\left( \widetilde{F}^{(r)}-F^{(r)} \right)Q_{m_2} \right\|_1.
        \end{equation}
        Now we are going to study the trace norm for any $r$. For a cleaner notation, let
        \begin{equation*}
          \mathcal{A} =\widetilde{\mathcal{J}}-x_0-\frac{\eta_r}{n^{\alpha}}, \quad 	\mathcal{B} = \mathcal{J}-x_0-\frac{\eta_r}{n^{\alpha}}.
        \end{equation*}
        It is sufficient to estimate $\left\| 	P_{m_1}\left( \mathcal{A}^{-1}-\mathcal{B}^{-1} \right)Q_{m_2}  \right\|_1$. Note that estimates about $\mathcal{A}^{-1}$ are discussed in Section~\ref{sec:thm3 Toeplitz}.
        
        We use resolvent identity and rewrite $\mathcal{B}^{-1}=\mathcal{A}^{-1}+\left( \mathcal{B}^{-1} -\mathcal{A}^{-1}\right)$ to get
        \begin{multline}\label{eq:thm3 lemma sum}
          P_{m_1}\left( \mathcal{A}^{-1}-\mathcal{B}^{-1} \right)Q_{m_2} = P_{m_1}\mathcal{B}^{-1}\left(\mathcal{B}-\mathcal{A} \right) \mathcal{A}^{-1}Q_{m_2} \\
          = P_{m_1}\mathcal{A}^{-1}\left(\mathcal{B}-\mathcal{A} \right) \mathcal{A}^{-1}Q_{m_2}+ P_{m_1} \left( \mathcal{B}^{-1} -\mathcal{A}^{-1} \right)\left(\mathcal{B}-\mathcal{A} \right)\mathcal{A}^{-1}Q_{m_2}.
        \end{multline}
        Note that, by the fact that $\mathrm{Id}=P_{m_2}+Q_{m_2}$ the second summand on the right-hand side can be written as $$ P_{m_1} \left( \mathcal{B}^{-1} -\mathcal{A}^{-1} \right)P_{m_2}\left(\mathcal{B}-\mathcal{A} \right)\mathcal{A}^{-1}Q_{m_2}+ P_{m_1} \left( \mathcal{B}^{-1} -\mathcal{A}^{-1} \right)Q_{m_2}\left(\mathcal{B}-\mathcal{A} \right)\mathcal{A}^{-1}Q_{m_2}.$$ Rearrange the formula \eqref{eq:thm3 lemma sum} and we get, 
        \begin{multline}\label{eq:thm3 lemma AB}
          P_{m_1}\left( \mathcal{A}^{-1}-\mathcal{B}^{-1} \right)Q_{m_2} \left( \mathrm{Id}+\left(\mathcal{B}-\mathcal{A} \right)\mathcal{A}^{-1}Q_{m_2} \right)\\
          = P_{m_1}\mathcal{A}^{-1}\left(\mathcal{B}-\mathcal{A} \right) \mathcal{A}^{-1}Q_{m_2}+ P_{m_1} \left( \mathcal{B}^{-1} -\mathcal{A}^{-1} \right)P_{m_2}\left(\mathcal{B}-\mathcal{A} \right)\mathcal{A}^{-1}Q_{m_2}.
        \end{multline}
        
        Note that, by Lemma~\ref{lemma: free exp}, we have $\left\| \mathcal{A}^{-1}Q_{m_2} - Q_{m_2-n^\beta}\mathcal{A}^{-1}Q_{m_2} \right\|_1=O(e^{-d'n^{\beta-\frac{\alpha}{2}}})$. Moreover, since $\mathcal{B}-\mathcal{A} $ is three diagonal, we have $\left(\mathcal{B}-\mathcal{A} \right)Q_{m_2-n^\beta}= Q_{m_2-n^\beta-1}\left(\mathcal{B}-\mathcal{A} \right)Q_{m_2-n^\beta}$. Also by Lemma~\ref{lemma: free exp}, we have $\left\| \mathcal{A}^{-1}Q_{m_2-n^\beta-1} - Q_{m_2-2n^\beta-1}\mathcal{A}^{-1}Q_{m_2-n^\beta-1} \right\|_1=O(e^{-d'n^{\beta-\frac{\alpha}{2}}})$. Also note that $\|\mathcal{A}^{-1}\|_\infty=O(n^{\alpha})$, $\|\mathcal{A}-\mathcal{B}\|_\infty=O(1)$. Hence, the first summand on the right-hand side of \eqref{eq:thm3 lemma AB} can be estimated to be
        \begin{multline}\label{eq:thm3 lemma Qcommute}
          \left\| P_{m_1}\mathcal{A}^{-1}\left(\mathcal{B}-\mathcal{A} \right) \mathcal{A}^{-1}Q_{m_2} \right\|_1\\
          =\left\| P_{m_1}Q_{m_2-2n^\beta-1}\mathcal{A}^{-1}Q_{m_2-n^\beta-1}\left(\mathcal{B}-\mathcal{A} \right) Q_{m_2-n^\beta}\mathcal{A}^{-1}Q_{m_2} \right\|_1+O\left(n^{\alpha}e^{-d'n^{\beta-\frac{\alpha}{2}}}\right).
        \end{multline}
        Similarly, we commute $P_{m_1}$ from left to right on the right-hand side of \eqref{eq:thm3 lemma Qcommute} with exponentially small error. Precisely, by Lemma~\ref{lemma: free exp}, we have $\left\| P_{m_1}Q_{m_2-2n^\beta-1}\mathcal{A}^{-1}-P_{m_1}Q_{m_2-2n^\beta-1}\mathcal{A}^{-1}P_{m_1+n^\beta} \right\|_1=O(e^{-d'n^{\beta-\frac{\alpha}{2}}}).$ Since $\mathcal{B}-\mathcal{A} $ is three diagonal, we have $P_{m_1+n^\beta}\left(\mathcal{B}-\mathcal{A} \right)= P_{m_1+n^\beta}\left(\mathcal{B}-\mathcal{A} \right)P_{m_1+n^\beta+1}.$ Also by Lemma~\ref{lemma: free exp}, we have $\left\| P_{m_1+n^\beta+1}Q_{m_2-n^\beta}\mathcal{A}^{-1}-P_{m_1+n^\beta+1}Q_{m_2-n^\beta}\mathcal{A}^{-1}P_{m_1+2n^\beta+1} \right\|_1=O(e^{-d'n^{\beta-\frac{\alpha}{2}}})$. Note that $\|\mathcal{A}^{-1}\|_\infty=O(n^{\alpha})$, $\|\mathcal{A}-\mathcal{B}\|_\infty=O(1)$.  Also note that the projections are commutative, i.e.,  $P_{m_1}Q_{m_2}=Q_{m_2}P_{m_1}$. Hence, \eqref{eq:thm3 lemma Qcommute} can be further estimated to be as $n\to\infty$
        \begin{multline}\label{eq:thm3 lemma pq commute}
          \left\| P_{m_1}\mathcal{A}^{-1}\left(\mathcal{B}-\mathcal{A} \right) \mathcal{A}^{-1}Q_{m_2} \right\|_1=\\
          \left\| P_{m_1}Q_{m_2-2n^\beta-1}\mathcal{A}^{-1}P_{m_1+n^\beta}Q_{m_2-n^\beta-1}\left(\mathcal{B}-\mathcal{A} \right) P_{m_1+n^\beta+1}Q_{m_2-n^\beta}\mathcal{A}^{-1}P_{m_1+2n^\beta+1}Q_{m_2} \right\|_1\\
          +O\left(n^{\alpha}e^{-d'n^{\beta-\frac{\alpha}{2}}}\right).
        \end{multline}
        Now use the trace norm inequality $\|ABC\|_1\leq \|A\|_2\|B\|_\infty\|C\|_2$ and we get as $n\to\infty$
        \begin{multline}\label{eq:thm3 lemma trace 1}
          \left\| P_{m_1}\mathcal{A}^{-1}\left(\mathcal{B}-\mathcal{A} \right) \mathcal{A}^{-1}Q_{m_2} \right\|_1\leq \\
          \left\| P_{m_1}Q_{m_2-2n^\beta-1}\mathcal{A}^{-1}\right\|_2\left\|P_{m_1+n^\beta}Q_{m_2-n^\beta-1}\left(\mathcal{B}-\mathcal{A} \right) P_{m_1+n^\beta+1}Q_{m_2-n^\beta}\right\|_\infty\left\|\mathcal{A}^{-1}P_{m_1+2n^\beta+1}Q_{m_2} \right\|_2\\
          +O\left(n^{\alpha}e^{-d'n^{\beta-\frac{\alpha}{2}}}\right).
        \end{multline}
        Recall that we assume $m_1-m_2=O(n^{\beta})$. By Lemma~\ref{lemma: HS norm}, we have  as $n\to\infty$
        \begin{align}
          \left\| P_{m_1}Q_{m_2-2n^\beta-1}\mathcal{A}^{-1}\right\|_2=O\left( n^{\frac{3\alpha}{4}+\frac{\beta}{2}} \right) \label{eq:thm3 lemma trace 11}\\
          \left\|\mathcal{A}^{-1}P_{m_1+2n^\beta+1}Q_{m_2} \right\|_2=O\left( n^{\frac{3\alpha}{4}+\frac{\beta}{2}} \right).\label{eq:thm3 lemma trace 12}
        \end{align}
        Since \eqref{eq:ass perturbation recurrence 1} holds for all $j>m_2-2n^\beta$ by assumption, we have as $n\to\infty$
        \begin{equation}
          \left\|P_{m_1+n^\beta}Q_{m_2-n^\beta-1}\left(\mathcal{B}-\mathcal{A} \right) P_{m_1+n^\beta+1}Q_{m_2-n^\beta}\right\|_\infty=O\left(n^{-2\beta}\right).\label{eq:thm3 lemma trace 13}
        \end{equation}
        Plug \eqref{eq:thm3 lemma trace 11}, \eqref{eq:thm3 lemma trace 12} and \eqref{eq:thm3 lemma trace 13} into \eqref{eq:thm3 lemma trace 1} to get as $n\to\infty$
        \begin{equation}\label{eq:thm3 lemma 01}
          \left\| P_{m_1}\mathcal{A}^{-1}\left(\mathcal{B}-\mathcal{A} \right) \mathcal{A}^{-1}Q_{m_2} \right\|_1 =O\left( n^{\frac{3\alpha}{2}-\beta} \right).
        \end{equation}
        
        For the second summand of the right-hand side of \eqref{eq:thm3 lemma AB}, we first use the resolvent identity and then continue with the same argument as \eqref{eq:thm3 lemma pq commute} to commute $Q_{m_2}$ from right to left to obtain that  as $n\to\infty$
        \begin{multline}\label{eq:thm3 lemma 22}
          \left\| P_{m_1} \left( \mathcal{B}^{-1} -\mathcal{A}^{-1} \right)P_{m_2}\left(\mathcal{B}-\mathcal{A} \right)\mathcal{A}^{-1}Q_{m_2} \right\|_1= \left\| P_{m_1}\mathcal{B}^{-1} \left( \mathcal{A} -\mathcal{B}  \right)\mathcal{A}^{-1}P_{m_2}\left(\mathcal{B}-\mathcal{A} \right)\mathcal{A}^{-1}Q_{m_2} \right\|_1 \\
          =\left\| P_{m_1}\mathcal{B}^{-1} \left( \mathcal{A} -\mathcal{B}  \right)Q_{m_2-2n^\beta-1}\mathcal{A}^{-1}P_{m_2}Q_{m_2-n^\beta-1}\left(\mathcal{B}-\mathcal{A} \right)P_{m_2+1}Q_{m_2-n^\beta}\mathcal{A}^{-1}Q_{m_2} \right\|_1 \\
          +O\left( n^{2\alpha}e^{-d'n^{\beta-\frac{\alpha}{2}}} \right).
        \end{multline}
        Use the trace norm inequality $\|ABCD\|_1\leq \|A\|_\infty\|B\|_2\|C\|_\infty\|D\|_2$ to get 
        \begin{multline} \label{eq:thm3 lemma 221}
          \left\| P_{m_1}\mathcal{B}^{-1} \left( \mathcal{A} -\mathcal{B}  \right)Q_{m_2-2n^\beta-1}\mathcal{A}^{-1}P_{m_2}Q_{m_2-n^\beta-1}\left(\mathcal{B}-\mathcal{A} \right)P_{m_2+1}Q_{m_2-n^\beta}\mathcal{A}^{-1}Q_{m_2} \right\|_1 \\
          \leq \left\| P_{m_1}\mathcal{B}^{-1} \left( \mathcal{A} -\mathcal{B}  \right)Q_{m_2-2n^\beta-1}\right\|_\infty\left\|\mathcal{A}^{-1}P_{m_2}Q_{m_2-n^\beta-1}\right\|_2
          \left\|P_{m_2}Q_{m_2-n^\beta-1}\left(\mathcal{B}-\mathcal{A} \right)\right\|_\infty\left\|P_{m_2+1}Q_{m_2-n^\beta}\mathcal{A}^{-1}Q_{m_2}\right\|_2
        \end{multline}
        In \eqref{eq:thm3 lemma 221}, two operator norms are of order $O(n^{\alpha-2\beta})$ and $O(n^{-2\beta})$ respectively by the assumption that  \eqref{eq:ass perturbation recurrence 1} holds for all $j>m_2-2n^\beta$; two Hilbert-Schmidt norms are both of order $O(n^{\frac{3\alpha}{4}+\frac{\beta}{2}})$ by Lemma~\ref{lemma: HS norm}. Hence, as $n\to\infty$
        \begin{equation} \label{eq:thm3 lemma 222}
          \left\| P_{m_1}\mathcal{B}^{-1} \left( \mathcal{A} -\mathcal{B}  \right)Q_{m_2-2n^\beta-1}\mathcal{A}^{-1}P_{m_2}Q_{m_2-n^\beta-1}\left(\mathcal{B}-\mathcal{A} \right)P_{m_2+1}Q_{m_2-n^\beta}\mathcal{A}^{-1}Q_{m_2} \right\|_1 =O \left( n^{\frac{5\alpha}{2}-3\beta} \right).
        \end{equation}
        Plug \eqref{eq:thm3 lemma 222} into \eqref{eq:thm3 lemma 22} to get
        \begin{equation}\label{eq:thm3 lemma 220}
          \left\| P_{m_1} \left( \mathcal{B}^{-1} -\mathcal{A}^{-1} \right)P_{m_2}\left(\mathcal{B}-\mathcal{A} \right)\mathcal{A}^{-1}Q_{m_2} \right\|_1= O \left( n^{\frac{5\alpha}{2}-3\beta} \right).
        \end{equation}
        We now use \eqref{eq:thm3 lemma 01} and  \eqref{eq:thm3 lemma 220} to estimate the trace norm of \eqref{eq:thm3 lemma AB}, and we have as $n\to\infty$
        \begin{equation}\label{eq:thm3 lemma 0 new 1}
          \left\| 	P_{m_1}\left( \mathcal{A}^{-1}-\mathcal{B}^{-1} \right)Q_{m_2} \left( \mathrm{Id}+\left(\mathcal{B}-\mathcal{A} \right)\mathcal{A}^{-1}Q_{m_2} \right) \right\|_1=O\left( n^{\frac{3\alpha}{2}-\beta} \right).
        \end{equation}
        Note that we use the assumption $0<\frac{\alpha}{2}<\beta<1$ to determine $\frac{3\alpha}{2}-\beta> \frac{5\alpha}{2}-3\beta$.		  
        
        By Lemma~\ref{lemma: free exp}, 
        \begin{equation}
          \left\|\left(\mathcal{B}-\mathcal{A} \right)\mathcal{A}^{-1}Q_{m_2} \right\|_\infty=\left\|\left(\mathcal{B}-\mathcal{A} \right)Q_{m_2-n^\beta}\mathcal{A}^{-1}Q_{m_2} \right\|_\infty+O\left( e^{-d'n^{\beta-\frac{\alpha}{2}}} \right).
        \end{equation}
        Using the operator norm inequality we have 
        \begin{equation}
          \left\|\left(\mathcal{B}-\mathcal{A} \right)Q_{m_2-n^\beta}\mathcal{A}^{-1}Q_{m_2} \right\|_\infty\leq\left\|\left(\mathcal{B}-\mathcal{A} \right)Q_{m_2-n^\beta}\right\|_\infty\left\|\mathcal{A}^{-1}Q_{m_2} \right\|_\infty.
        \end{equation}
        Since \eqref{eq:ass perturbation recurrence 1} holds for all $j>m_2-2n^\beta$, we have $\left\|\left(\mathcal{B}-\mathcal{A} \right)Q_{m_2-n^\beta}\right\|_\infty=O(n^{-2\beta})$, as $n\to\infty$. Recall that $\|\mathcal{A}^{-1}\|_{\infty}=O(n^{\alpha})$. Now we have 
        \begin{equation}\label{eq:thm3 lemma 211}
          \left\|\left(\mathcal{B}-\mathcal{A} \right)\mathcal{A}^{-1}Q_{m_2} \right\|_\infty=O\left( n^{\alpha-2\beta} \right).
        \end{equation}
        Recall $0<\frac{\alpha}{2}<\beta$, and hence \eqref{eq:thm3 lemma 211} is of order $o(1)$. Then for large $n$,
        \begin{equation}\label{eq:thm3 lemma 0 new 2}
          \left\|  \left( \mathrm{Id}+\left(\mathcal{B}-\mathcal{A} \right)\mathcal{A}^{-1}Q_{m_2} \right)^{-1} \right\|_\infty\leq \frac{1}{1-	\left\|\left(\mathcal{B}-\mathcal{A} \right)\mathcal{A}^{-1}Q_{m_2} \right\|_\infty}.
        \end{equation} 
        Moreover, by trace norm inequality $\|AB\|_1\leq \|A\|_1\|B\|_\infty$, we have 
        \begin{multline}\label{eq:thm3 lemma 0 new 0}
          \left\| 	P_{m_1}\left( \mathcal{A}^{-1}-\mathcal{B}^{-1} \right)Q_{m_2}  \right\|_1 \\
          \leq 	\left\| 	P_{m_1}\left( \mathcal{A}^{-1}-\mathcal{B}^{-1} \right)Q_{m_2} \left( \mathrm{Id}+\left(\mathcal{B}-\mathcal{A} \right)\mathcal{A}^{-1}Q_{m_2} \right) \right\|_1\left\| \left( \mathrm{Id}+\left(\mathcal{B}-\mathcal{A} \right)\mathcal{A}^{-1}Q_{m_2} \right)^{-1} \right\|_\infty.
        \end{multline}
        
        Plug \eqref{eq:thm3 lemma 0 new 1},\eqref{eq:thm3 lemma 211} and \eqref{eq:thm3 lemma 0 new 2} into \eqref{eq:thm3 lemma 0 new 0} to obtain as $n\to\infty$
        \begin{equation}
          \left\| 	P_{m_1}\left( \mathcal{A}^{-1}-\mathcal{B}^{-1} \right)Q_{m_2}  \right\|_1=O\left( n^{\frac{3\alpha}{2}-\beta} \right).
        \end{equation}
        
        Plug this estimate back into \eqref{eq:thm3 lemma 1} and we conclude as $n\to\infty$
        
        \begin{equation}
          \left\| P_{m_1}\left( \widetilde{F}-F \right)Q_{m_2} \right\|_1 = O\left( n^{\frac{3\alpha}{2}-\beta} \right).
        \end{equation}
        
      \end{proof}
      
      Lemma above also implies the following by induction. 
      
      \begin{lemma} \label{lemma: thm3 4}
        Let $0<\frac{\alpha}{2}<\beta<1$. 	Let $m\in \mathbb{N}$. Assume \eqref{eq:ass perturbation recurrence 1} holds. We have, for any $k=0,1,2,\dots, m$
        \begin{equation}\label{eq:thm3 lemma4 0}
          \left\| P_{n+mn^\beta}\left( \widetilde{F}-F \right)F^{k}Q_{n} \right\|_1=O\left( n^{k\alpha+\frac{3\alpha}{2}-\beta} \right).
        \end{equation}
      \end{lemma}
      \begin{proof}
        First note that the following operator norms are bounded
        \begin{equation}\label{eq:thm3 lemma opt}
          \sup_{n>0}\{n^{-\alpha}\|\widetilde{F}\|_\infty,n^{-\alpha}\|F\|_\infty\}<\infty.
        \end{equation}
        
        Take $m_1\equiv n+mn^\beta$ and $m_2\equiv n$ to shorten the notation. 
        
        Consider the induction argument. For $k=0$, \eqref{eq:thm3 lemma4 0} is reduced to Lemma~\ref{lemma: thm3 step0}. For any $k\geq0$, we use triangle inequality to obtain
        \begin{equation}\label{eq:thm3 lemma4 1}
          \left\| P_{m_1}\left( \widetilde{F}-F \right)F^{k+1}Q_{m_2} \right\|_1 \leq \left\| P_{m_1}\left( \widetilde{F}-F \right)\widetilde{F}^{k+1}Q_{m_2} \right\|_1+\left\| P_{m_1}\left( \widetilde{F}-F \right)\left( F^{k+1} -\widetilde{F}^{k+1}\right)Q_{m_2} \right\|_1.
        \end{equation}
        Use Lemma~\ref{lemma: free exp} and we estimate the first summand on the right-hand side of \eqref{eq:thm3 lemma4 1} to be, as $n\to\infty$,
        \begin{equation}\label{eq:thm3 lemma4 11}
          \left\| P_{m_1}\left( \widetilde{F}-F \right)\widetilde{F}^{k+1}Q_{m_2} \right\|_1= \left\| P_{m_1}\left( \widetilde{F}-F \right)Q_{m_2-(k+1)n^\beta}\widetilde{F}^{k+1}Q_{m_2} \right\|_1+O\left(n^{\alpha}e^{-d'n^{\beta-\frac{\alpha}{2}}}\right).
        \end{equation}
        By the trace norm inequality $\|AB\|_1\leq \|A\|_1\|B\|_\infty$, we have
        \begin{equation}
          \left\| P_{m_1}\left( \widetilde{F}-F \right)Q_{m_2-(k+1)n^\beta}\widetilde{F}^{k+1}Q_{m_2} \right\|_1\leq \left\| P_{m_1}\left( \widetilde{F}-F \right)Q_{m_2-(k+1)n^\beta} \right\|_1\left\| \widetilde{F}^{k+1}Q_{m_2} \right\|_\infty.\label{eq:thm3 lemma4 121}
        \end{equation}
        Then use Lemma~\ref{lemma: thm3 step0} to estimate the trace norm above to be of order $O(n^{3\alpha/2-\beta})$. By \eqref{eq:thm3 lemma opt}, the operator norm is of order  $O\left(n^{(k+1)\alpha}\right)$. Plug the estimate of \eqref{eq:thm3 lemma4 121} back to \eqref{eq:thm3 lemma4 11} to obtain, as $n\to\infty$
        \begin{equation}\label{eq:thm3 lemma4 12}
          \left\| P_{m_1}\left( \widetilde{F}-F \right)\widetilde{F}^{k+1}Q_{m_2} \right\|_1=O\left( n^{(k+1)\alpha+\frac{3\alpha}{2}-\beta} \right).
        \end{equation}
        For the second summand on the right-hand side of \eqref{eq:thm3 lemma4 1} we use the telescopic sum, $$F^{k+1} -\widetilde{F}^{k+1}=\sum_{l=0}^kF^{l}\left( F -\widetilde{F}\right)\widetilde{F}^{k-l},$$ the fact that $Q_{m_2}+P_{m_2}=\mathrm{Id}$, and triangle inequality to obtain
        \begin{multline}\label{eq:thm3 lemma4 13}
          \left\| P_{m_1}\left( \widetilde{F}-F \right)\left( F^{k+1} -\widetilde{F}^{k+1}\right)Q_{m_2} \right\|_1 \leq \sum_{l=0}^k\left\| P_{m_1}\left( \widetilde{F}-F \right)F^{l}\left( Q_{m_2}+P_{m_2} \right)\left( F -\widetilde{F}\right)\widetilde{F}^{k-l}Q_{m_2} \right\|_1 \\
          \leq \sum_{l=0}^k\left( \left\| P_{m_1}\left( \widetilde{F}-F \right)F^{l}Q_{m_2} \left( F -\widetilde{F}\right)\widetilde{F}^{k-l}Q_{m_2} \right\|_1 +\left\| P_{m_1}\left( \widetilde{F}-F \right)F^{l}P_{m_2} \left( F -\widetilde{F}\right)\widetilde{F}^{k-l}Q_{m_2} \right\|_1  \right).
        \end{multline}
        Use the trace norm inequality $\|ABC\|_1\leq \|A\|_1\|B\|_\infty\|C\|_\infty$ to obtain an estimate of the first term on the right-hand side of \eqref{eq:thm3 lemma4 13}, 
        \begin{equation}\label{eq:thm3 lemma4 131}
          \left\| P_{m_1}\left( \widetilde{F}-F \right)F^{l}Q_{m_2} \left( F -\widetilde{F}\right)\widetilde{F}^{k-l}Q_{m_2} \right\|_1 
          \leq \left\| P_{m_1}\left( \widetilde{F}-F \right)F^{l}Q_{m_2} \right\|_1  \left\|Q_{m_2}\left( F -\widetilde{F}\right)\right\|_\infty\left\|\widetilde{F}^{k-l}Q_{m_2}\right\|_\infty.
        \end{equation}
        
        The induction hypothesis for all $l=0,\dots, k$ implies the first term on the right-hand side of  \eqref{eq:thm3 lemma4 131} to be
        \begin{equation}\label{eq:thm3 lemma induciton1}
          \left\| P_{m_1}\left( \widetilde{F}-F \right)F^{l}Q_{m_2} \right\|_1=O\left( n^{l\alpha+\frac{3\alpha}{2}-\beta} \right).
        \end{equation}
        
        For the second term on the right-hand side of \eqref{eq:thm3 lemma4 131}, let us define, 
        Let us define
        \begin{align}
          \widetilde{F}^{(r)} \coloneqq &\left( \widetilde{\mathcal{J}}-x_0-\frac{\eta_r}{n^{\alpha}} \right)^{-1}\\
          F^{(r)} \coloneqq &  \left( \mathcal{J}-x_0-\frac{\eta_r}{n^{\alpha}} \right)^{-1}.
        \end{align}
        Recall the definition of $\widetilde{F}$ and $F$ in \eqref{eq:def F tilde} and  \eqref{eq:def F}. We apply the triangle inequality and the resolvent identity to obtain
        \begin{equation}\label{eq:thm3 lemma4 1311}
          \left\|Q_{m_2}\left( F -\widetilde{F}\right)\right\|_\infty\leq \sum_r|c_r|\left\|  Q_{m_2}
          \left( F^{(r)} -\widetilde{F}^{(r)}\right)\right\|_\infty=\sum_r|c_r|\left\|  Q_{m_2}
          \widetilde{F}^{(r)}\left(\widetilde{\mathcal{J}} -\mathcal{J}\right)F^{(r)} \right\|_\infty.
        \end{equation}
        Use  Lemma~\ref{lemma: free exp}, and we get
        \begin{equation}\label{eq:thm3 lemma4 131111}
          \left\|  Q_{m_2}
          \widetilde{F}^{(r)}\left(\widetilde{\mathcal{J}} -\mathcal{J}\right)F^{(r)} \right\|_\infty= 	\left\|  Q_{m_2}
          \widetilde{F}^{(r)}Q_{m_2-n^\beta}\left(\widetilde{\mathcal{J}} -\mathcal{J}\right)F^{(r)} \right\|_\infty+O\left( n^\alpha e^{-d'n^{\beta-\frac{\alpha}{2}}} \right) .
        \end{equation}
        
        By the operator norm inequality $\|AB\|_\infty\leq\|A\|_\infty\|B\|_\infty$ we obtain, 
        \begin{equation}\label{eq:thm3 lemma4 131112}
          \left\|  Q_{m_2}
          \widetilde{F}^{(r)}Q_{m_2-n^\beta}\left(\widetilde{\mathcal{J}} -\mathcal{J}\right)F^{(r)} \right\|_\infty \leq \left\|  Q_{m_2}
          \widetilde{F}^{(r)}\right\|_\infty\left\|Q_{m_2-n^\beta}\left(\widetilde{\mathcal{J}} -\mathcal{J}\right)\right\|_\infty\left\|F^{(r)} \right\|_\infty .
        \end{equation}
        By the assumption that \eqref{eq:ass perturbation recurrence 1} holds for all $j>m_2-2n^\beta$, we get $\left\|Q_{m_2-n^\beta}\left(\widetilde{\mathcal{J}} -\mathcal{J}\right)\right\|_\infty=O(n^{-2\beta})$. Together with the fact that $\|F^{(r)}\|_\infty=O\left( n^{\alpha} \right)$ and $\|\widetilde{F}^{(r)}\|_\infty=O\left( n^{\alpha} \right)$, we estimate \eqref{eq:thm3 lemma4 131112} to be of order $O(n^{2\alpha-2\beta})$. Plugging it into \eqref{eq:thm3 lemma4 131111} and further back into  \eqref{eq:thm3 lemma4 1311}, we get
        \begin{equation}
          \label{eq:thm3 lemma induciton2}
          \left\|Q_{m_2}\left( F -\widetilde{F}\right)\right\|_\infty=O\left( n^{2\alpha-2\beta} \right).
        \end{equation}
        The last term  on the right-hand side of  \eqref{eq:thm3 lemma4 131}  is estimated to be $\left\|\widetilde{F}^{k-l}Q_{m_2}\right\|_\infty=O\left( n^{(k-l)\alpha} \right)$. Hence, together with \eqref{eq:thm3 lemma induciton1} and \eqref{eq:thm3 lemma induciton2}, \eqref{eq:thm3 lemma4 131}  is estimated to be 
        
        \begin{equation}\label{eq:thm3 lemma4 1310}
          \left\| P_{m_1}\left( \widetilde{F}-F \right)F^{l}Q_{m_2} \left( F -\widetilde{F}\right)\widetilde{F}^{k-l}Q_{m_2} \right\|_1 =O\left( n^{(k+2)\alpha+\frac{3\alpha}{2}-3\beta} \right).
        \end{equation}
        
        Now we turn to the second term on the right-hand side of \eqref{eq:thm3 lemma4 13}. Use  Lemma~\ref{lemma: free exp}, 
        
        \begin{multline}\label{eq:thm3 lemma4 132}
          \left\| P_{m_1}\left( \widetilde{F}-F \right)F^{l}P_{m_2} \left( F -\widetilde{F}\right)\widetilde{F}^{k-l}Q_{m_2} \right\|_1 \\=\left\| P_{m_1}\left( \widetilde{F}-F \right)F^{l}P_{m_2} \left( F -\widetilde{F}\right)Q_{m_2-(k-l)n^\beta}\widetilde{F}^{k-l}Q_{m_2} \right\|_1 +O\left( n^{(l+2)\alpha}e^{-d'n^{\beta-\frac{\alpha}{2}}} \right). 
        \end{multline}
        Use the operator norm inequality $\|ABC\|_1\leq\|A\|_\infty\|B\|_1\|C\|_\infty$ to obtain
        \begin{multline}\label{eq:thm3 lemma4 1321}
          \left\| P_{m_1}\left( \widetilde{F}-F \right)F^{l}P_{m_2} \left( F -\widetilde{F}\right)Q_{m_2-(k-l)n^\beta}\widetilde{F}^{k-l}Q_{m_2} \right\|_1\\
          \leq \left\| P_{m_1}\left( \widetilde{F}-F \right)F^{l}\right\|_\infty\left\|P_{m_2} \left( F -\widetilde{F}\right)Q_{m_2-(k-l)n^\beta} \right\|_1\left\| \widetilde{F}^{k-l}Q_{m_2} \right\|_\infty.
        \end{multline}
        The operator norms are estimated to be $\left\| P_{m_1}\left( \widetilde{F}-F \right)F^{l}\right\|_\infty=O\left( n^{(l+1)\alpha} \right)$ and $\left\| \widetilde{F}^{k-l}Q_{m_2} \right\|_\infty=O\left( n^{(k-l)\alpha} \right)$. The trace norm can be estimated by Lemma~\ref{lemma: thm3 step0} to be $\left\|P_{m_2} \left( F -\widetilde{F}\right)Q_{m_2-(k-l)n^\beta} \right\|_1=O\left( n^{\frac{3\alpha}{2}-\beta} \right)$. Plug these three estimates into \eqref{eq:thm3 lemma4 1321}. Then we get an estimate of \eqref{eq:thm3 lemma4 132} to be 
        \begin{equation}\label{eq:thm3 lemma4 1320}
          \left\| P_{m_1}\left( \widetilde{F}-F \right)F^{l}P_{m_2} \left( F -\widetilde{F}\right)\widetilde{F}^{k-l}Q_{m_2} \right\|_1 \\=O\left( n^{(k+1)\alpha+\frac{3\alpha}{2}-\beta} \right).
        \end{equation}
        Plug \eqref{eq:thm3 lemma4 1310} and \eqref{eq:thm3 lemma4 1320} back into \eqref{eq:thm3 lemma4 13}, and we estimate the second summand on the right-hand side of \eqref{eq:thm3 lemma4 1} to be
        \begin{equation}\label{eq:thm3 lemma4 130}
          \left\| P_{m_1}\left( \widetilde{F}-F \right)\left( F^{k+1} -\widetilde{F}^{k+1}\right)Q_{m_2} \right\|_1 =O\left( n^{(k+1)\alpha+\frac{3\alpha}{2}-\beta}\right).
        \end{equation}
        Here we use the assumption $0<\frac{\alpha}{2}<\beta<1$ to deduce $O\left( n^{(k+2)\alpha+\frac{3\alpha}{2}-3\beta}\right)=o\left( n^{(k+1)\alpha+\frac{3\alpha}{2}-\beta}\right)$. 
        
        Plugging \eqref{eq:thm3 lemma4 12} and \eqref{eq:thm3 lemma4 130} into \eqref{eq:thm3 lemma4 1}, we obtain
        
        \begin{equation}\label{eq:thm3 lemma4 10}
          \left\| P_{m_1}\left( \widetilde{F}-F \right)F^{k+1}Q_{m_2} \right\|_1 =O\left( n^{(k+1)\alpha+\frac{3\alpha}{2}-\beta} \right).
        \end{equation}
        
        This concludes the induction argument. 
      \end{proof}

      \subsection{Proof of Theorem~\ref{thm: main 3}}
      The essential step of Theorem~\ref{thm: main 3} is the following proposition. Recall that $\mathcal{C}_m^{(n)}(\mathcal{A})$ is the $m$-th cumulant for a linear operator $\mathcal{A}$, defined as \eqref{eq:cumulant operator}.
      \begin{proposition}\label{prop: one side compare}
        Assume that the assumptions in Theorem~\ref{thm: main 3} are satisfied. Then we have
        \begin{align}
          \lim_{n\to\infty}n^{-m\alpha} \left| \mathcal{C}_m^{(n)}(\widetilde{F})  - 	\mathcal{C}_m^{(n)}(F)  \right|= 0.
        \end{align}
      \end{proposition}
      \begin{proof}
        First note that the following operator norms are bounded
        \begin{equation}
          \sup_{n>0}\{n^{-\alpha}\|\widetilde{F}\|_\infty,n^{-\alpha}\|F\|_\infty\}\eqqcolon C_{op}<\infty.
        \end{equation}
        Since both sums are finite in the cumulant formula  \eqref{eq:cumulant operator}, it is sufficient to show that 
        \begin{align}\label{eq:pf_cumulant_Comparison}
          \Tr(\widetilde{F})^{l_1}P_n\dots (\widetilde{F})^{l_j}P_n-\Tr(\widetilde{F}^mP_n)-\Tr(F)^{l_1}P_n\dots (F)^{l_j}P_n+\Tr(F^mP_n)=o(n^{m\alpha})
        \end{align}
        Then we rewrite the left-hand side of \eqref{eq:pf_cumulant_Comparison} by Lemma~\ref{lemma:traceAlgebras}.
        
        Then writing $\widetilde{F}^{l}-F^{l}=\sum_{k=0}^{l-1}F^{l-1-k}(\widetilde{F}-F)\widetilde{F}^{k}$,  by Lemma~\ref{lemma: free exp}, we commute $Q_n$ from the right to left to get
        \begin{multline}
          \left|\Tr\left((\widetilde{F}^{l_k}-F^{l_k})P_n\widetilde{F}^{l_{k+1}}\dots P_n\widetilde{F}^{l_j}P_n\widetilde{F}^{l_1+\dots+l_{k-1}}Q_n\right)\right| 
          \leq l_k C_{op}^{m-1}n^{(m-1)\alpha}\|(\widetilde{F}-F)P_nQ_{n-mn^\beta}\|_1 + R_n, \label{eq:perturbation modifiedJacobi_inequal1} 
        \end{multline}
        \begin{multline}
          \left|\Tr\left(F^{l_k}P_n\dots P_nF^{l_{i}}P_n(\widetilde{F}^{l_{i+1}}-F^{l_{i+1}})P_n\dots \widetilde{F}^{l_{j}}P_n\widetilde{F}^{l_1+\dots+l_{k-1}}Q_n\right)\right| \\
          \leq l_{i+1} C_{op}^{m-1}n^{(m-1)\alpha}\|(\widetilde{F}-F)P_nQ_{n-mn^\beta}\|_1 + R_n ,\label{eq:perturbation modifiedJacobi_inequal2}
        \end{multline}
        \begin{multline}
          \left|\Tr\left(F^{l_{k}}P_n\dots P_nF^{l_j}P_n(\widetilde{F}^{l_1+\dots+l_{k-1}}-F^{l_1+\dots+l_{k-1}})Q_n\right)\right|\\
          \leq C_{op}^{m-(l_1+\dots+l_{k-1})}n^{(m-(l_1+\dots+l_{k-1}))\alpha}\|P_n(\widetilde{F}^{l_1+\dots+l_{k-1}}-F^{l_1+\dots+l_{k-1}})Q_{n}\|_1 + R_n, \label{eq:perturbation modifiedJacobi_inequal3}
        \end{multline}
        where $R_n = O(n^{m\alpha}e^{-d_0n^{\beta-\frac{\alpha}{2}}})$ is exponentially small.
        
        Define 
        \begin{equation}
          \widetilde{F}^{(r)} \coloneqq \left( \widetilde{\mathcal{J}}-x_0-\frac{\eta_r}{n^{\alpha}} \right)^{-1},\quad 
          F^{(r)} \coloneqq   \left( \mathcal{J}-x_0-\frac{\eta_r}{n^{\alpha}} \right)^{-1}.
        \end{equation}
        Use the triangle inequality to obtain
        \begin{equation}
          \|(\widetilde{F}-F)P_nQ_{n-mn^\beta}\|_1\leq \sum_{r=1}^{2M}|c_r|\left\|(\widetilde{F}^{(r)}-F^{(r)})P_nQ_{n-mn^\beta}\right\|_1. \label{eq:perturbation 0}
        \end{equation}
        Recall that for any linear operator $\mathcal{A}$ and $\mathcal{B}$, we have the resolvent identity $\mathcal{A}^{-1}-\mathcal{B}^{-1}=\mathcal{B}^{-1}(\mathcal{B}-\mathcal{A})\mathcal{A}^{-1}$. Further, we also have $\mathcal{B}^{-1}=\left(\mathrm{Id}-\mathcal{B}^{-1}\left(  \mathcal{B}-\mathcal{A} \right)\right)\mathcal{A}^{-1}$. Combine these two formulas to get, 
        \begin{equation*}
          \mathcal{A}^{-1}-\mathcal{B}^{-1}=\left(\mathrm{Id}-\mathcal{B}^{-1}\left(  \mathcal{B}-\mathcal{A} \right)\right)\mathcal{A}^{-1}(\mathcal{B}-\mathcal{A})\mathcal{A}^{-1}.
        \end{equation*}
        Take  $\mathcal{A}=\widetilde{\mathcal{J}}-x_0-\frac{\eta_r}{n^{\alpha}}=\left(\widetilde{F}^{(r)}\right)^{-1}$ and $\mathcal{B}=\mathcal{J}-x_0-\frac{\eta_r}{n^{\alpha}}=\left(F^{(r)}\right)^{-1}$. Rewrite each summand of the right-hand side of \eqref{eq:perturbation 0} to be
        \begin{equation}\label{eq:thm3 prop 1}
          \left\|(\widetilde{F}^{(r)}-F^{(r)})P_nQ_{n-mn^\beta}\right\|_1= \left\|\left(\mathrm{Id}-F^{(r)}\left(  \mathcal{J}-\widetilde{\mathcal{J}} \right)\right)\widetilde{F}^{(r)}\left( \mathcal{J}-\widetilde{\mathcal{J}} \right)\widetilde{F}^{(r)}P_nQ_{n-mn^\beta}\right\|_1 .
        \end{equation}
        Recall that $\mathcal{J}-\widetilde{\mathcal{J}}$ is tri-diagonal and we have for any integer $m_1$ 
        
        \begin{align}
          \left( \mathcal{J}-\widetilde{\mathcal{J}} \right)P_{m_1}=P_{m_1+1}\left( \mathcal{J}-\widetilde{\mathcal{J}} \right)P_{m_1}.\\
          \left( \mathcal{J}-\widetilde{\mathcal{J}} \right)Q_{m_1}=Q_{m_1-1}\left( \mathcal{J}-\widetilde{\mathcal{J}} \right)Q_{m_1}.
        \end{align}
        Use Lemma~\ref{lemma: free exp} to get that $\widetilde{F}^{(r)}$  commutes with the projection operator $P_{m_1}$ and $Q_{m_1}$ with an exponential small error. That is
        \begin{multline}
          \left\|(\widetilde{F}^{(r)}-F^{(r)})P_nQ_{n-mn^\beta}\right\|_1\\
          \leq
          \Bigg\|\left(\mathrm{Id}-F^{(r)}\left( \mathcal{J}-\widetilde{\mathcal{J}}  \right)\right)P_{n+2n^\beta+1}Q_{n-(m+2)n^\beta-1}\widetilde{F}^{(r)}P_{n+n^\beta+1}Q_{n-(m+1)n^\beta-1}
          \\
          \left( \widetilde{\mathcal{J}}-\mathcal{J} \right)P_{n+n^\beta}Q_{n-(m+1)n^\beta}\widetilde{F}^{(r)}P_nQ_{n-mn^\beta}\Bigg\|_1+ R_n, \label{eq:perturbation 10}
        \end{multline}
        where $R_n = O(n^{m\alpha}e^{-d_0n^{\beta-\frac{\alpha}{2}}})$ is exponentially small. Use the trace norm inequality $\|ABCD\|_1\leq \|A\|_\infty\|B\|_2\|C\|_\infty\|D\|_2$  and \eqref{eq:thm3 prop 1} to obtain
        \begin{multline}
          \left\|(\widetilde{F}^{(r)}-F^{(r)})P_nQ_{n-mn^\beta}\right\|_1\leq  \|(\mathrm{Id}-F^{(r)}\left( \mathcal{J}-\widetilde{\mathcal{J}}  \right)P_{n+2n^\beta+1}Q_{n-(m+2)n^\beta-1}\|_\infty\\
          \|\widetilde{F}^{(r)}P_{n+n^\beta+1}Q_{n-(m+1)n^\beta-1}\|_2\|\left( \widetilde{\mathcal{J}}-\mathcal{J} \right)P_{n+n^\beta}Q_{n-(m+1)n^\beta}\|_\infty\|\widetilde{F}^{(r)}P_nQ_{n-mn^\beta}\|_2 + R_n, \label{eq:perturbation 1}
        \end{multline}
        where $R_n = O\left(n^{\alpha+2\beta}e^{-d_0n^{\beta-\frac{\alpha}{2}}}\right)$ is exponentially small. By the assumption \eqref{eq:ass perturbation recurrence 1} and the fact that $\|F^{(r)}\|_\infty=O(n^{\alpha})$ we have 
        \begin{align}
          \left\|\left( \widetilde{\mathcal{J}}-\mathcal{J} \right)P_{n+n^\beta}Q_{n-(m+1)n^\beta}\right\|_\infty = & O(n^{-2\beta}),  \label{eq:perturbation 1.1}\\
          \left\|F^{(r)}\left( \mathcal{J}-\widetilde{\mathcal{J}}  \right)P_{n+2n^\beta+1}Q_{n-(m+2)n^\beta-1}\right\|_\infty= & O(n^{\alpha-2\beta})  \label{eq:perturbation 1.2}.
        \end{align}
      Use Lemma~\ref{lemma: HS norm} to obtain
        \begin{equation}
          \|\widetilde{F}^{(r)}P_nQ_{n-mn^\beta}\|_2^2= O(n^{\beta+\frac{3\alpha}{2}})  \label{eq:perturbation 1.3}.
        \end{equation}
        Recall that $R_n$ is exponentially small for $n$ large. Plug \eqref{eq:perturbation 1.1}, \eqref{eq:perturbation 1.2} and \eqref{eq:perturbation 1.3} into \eqref{eq:perturbation 1} to obtain, as $n\to\infty$,
        \begin{equation}\label{eq:thm3 prop 2}
          \left\|(\widetilde{F}^{(r)}-F^{(r)})P_nQ_{n-mn^\beta}\right\|_1=O(n^{-\beta+\frac{3\alpha}{2}}).
        \end{equation}
        The estimate \eqref{eq:thm3 prop 2}, together with \eqref{eq:perturbation 0}, implies that \eqref{eq:perturbation modifiedJacobi_inequal1} and \eqref{eq:perturbation modifiedJacobi_inequal2} are both of order $O\left( n^{m\alpha-\beta+\frac{\alpha}{2}} \right)$. 
        
        For \eqref{eq:perturbation modifiedJacobi_inequal3}, consider any $l=1,2,\dots,m$. Recall the telescoping sum
        \begin{equation}
          \widetilde{F}^{l}-F^{l}= \sum_{k=0}^{l-1} \widetilde{F}^{l-k-1}\left(\widetilde{F}- F \right)F^{k}
        \end{equation}
        Apply Lemma~\ref{lemma: free exp} to obtain
        \begin{equation}\label{eq:perturbation modifiedJacobi_trouble_term_2}
          \left\|P_n(\widetilde{F}^{l}-F^{l})Q_{n}\right\|_1\leq \sum_{k=0}^{l-1} (n^\alpha C_{op})^{l-k-1}\left\|P_{n+(l-k-1)n^\beta}(\widetilde{F}-F)F^{k}Q_n\right\|_1+R_n,
        \end{equation}
        where $R_n = O\left(n^{\alpha+2\beta}e^{-d_0n^{\beta-\frac{\alpha}{2}}}\right)$ is exponentially small. We estimate \eqref{eq:perturbation modifiedJacobi_trouble_term_2}, by Lemma~\ref{lemma: thm3 4}, to be 
        \begin{align}
          \|P_n(\widetilde{F}^{l}-F^{l})Q_{n}\|_1=O\left( n^{(l-1)\alpha+\frac{3\alpha}{2}-\beta} \right)
        \end{align}
        This shows that \eqref{eq:perturbation modifiedJacobi_inequal3} is of order $O\left(n^{m\alpha+\frac{\alpha}{2}-\beta}\right)$.
        
        In summary we have shown that \eqref{eq:perturbation modifiedJacobi_inequal1}, \eqref{eq:perturbation modifiedJacobi_inequal2} and  \eqref{eq:perturbation modifiedJacobi_inequal3} are all of order of $O\left( n^{m\alpha-\beta+\frac{\alpha}{2}} \right)$. Hence, by assumption $0<\frac{\alpha}{2}<\beta<2$, we conclude that \eqref{eq:pf_cumulant_Comparison} holds. Plug this into the cumulant formula \eqref{eq:cumulant operator} and we complete the proof.
      \end{proof}

      Now we are ready to prove Theorem~\ref{thm: main 3}.
      \begin{proof}[Proof of Theorem~\ref{thm: main 3} ]
        
        First consider test functions like $f(x)=\sum_{r=1}^Md_r\Im (x- \lambda_r)^{-1}$ for $d_r\in \mathbb{R}$ and $\Im(\lambda_r)>0$,i.e., \eqref{eq:resolvent fun} . 

         Note that by Lemma~\ref{lemma:traceAlgebras}, we can remove the first block of $\widetilde{\mathcal{J}}$ in the cumulant expression $\mathcal{C}_m(\widetilde{F})$ for all $m
         \geq2$. Then $\mathcal{C}_m(\widetilde{F})$ can be regarded as the cumulant for the Chebyshev polynomial of the second kind at the edge as described in beginning of Subsection~\ref{sec:thm3 Toeplitz} of order $n^{\beta+\varepsilon/2}$ at the mesoscopic scale of $\frac{\alpha}{\beta+\varepsilon/2} =\frac{\alpha}{\alpha/2+\varepsilon}\in(0,2)$. Apply Theorem~\ref{thm: main 2} and get the following convergence in distribution as $n\to\infty$.
        
         This implies the convergence of the cumulants  
        \begin{align}
          \lim_{n\to\infty}n^{-m\alpha} \mathcal{C}_m(\widetilde{F}) =\begin{cases}
            \sigma_f^2, &\quad m=2 \\
            0, &\quad m>2
          \end{cases},
        \end{align}
        
        where $\sigma_f^2$ is the variance in Theorem~\ref{thm: main 2}.
        Then Proposition~\ref{prop: one side compare} implies that 
        \begin{align}
          \lim_{n\to\infty}n^{-m\alpha} \mathcal{C}_m(F) =\begin{cases}
            \sigma_f^2, &\quad m=2 \\
            0, &\quad m>2.
          \end{cases}
        \end{align}
        This implies that
        \begin{equation}\label{eq:free clt resolvent}
          X_{f_{\alpha,x_0}}^{(n)}-\mathbb{E}[X_{f_{\alpha,x_0}}^{(n)}] \to \mathcal{N}(0,\sigma_f^2),
        \end{equation}
        where $X_{f_{\alpha,x_0}}^{(n)}$ is the mesoscopic linear statistics of for OPEs whose recurrence coefficients satisfy \eqref{eq:ass perturbation recurrence 1}.
        
        Now recall the argument in Section~\ref{sec:Lipschitz extension}. Use the exact same proof of the last step of Theorem~\ref{thm: main 2} in Section~\ref{sec:Lipschitz extension} and we extend \eqref{eq:free clt resolvent} to the test functions $f\in C_c^1$.
      \end{proof}

      \section{Examples}	 \label{sec:examples}
      Theorems~\ref{thm: main 3}, ~\ref{thm: main 1} and~\ref{thm: main 2} can be used to obtain the asymptotics of mesoscopic fluctuations at the edges of many classes of OPEs, that are known in the literature. Here we present some of the interesting examples with both continuous and discrete measures. We include classical examples as well as uncommon but still popular ones. 
      
      \subsection{Laguerre Unitary Ensemble} \label{sec:eg Laguerre}
      The orthogonality measure of scaled Laguerre polynomials is given by
      $$d\mu(x) = x^{\gamma}e^{-nx}dx, \quad x\geq 0,$$
      for some parameter $\gamma>-1$. 
      The recurrence relation reads  $$xp_j(x)=\frac{\sqrt{(j+1)(j+1+\gamma)}}{n}p_{j+1}(x)+\frac{2j+\gamma+1}{n}p_j(x)+\frac{\sqrt{j(j+\gamma)}}{n}p_{j-1}(x)$$ 
      and hence the recurrence coefficients are 
      \begin{equation}
        a_{j,n} = \frac{\sqrt{j(j+\gamma)}}{n} \quad,\quad b_{j,n}=\frac{2j+\gamma+1}{n}
      \end{equation}
      \begin{figure}[t]
        \centering
        \includegraphics[scale=0.5]{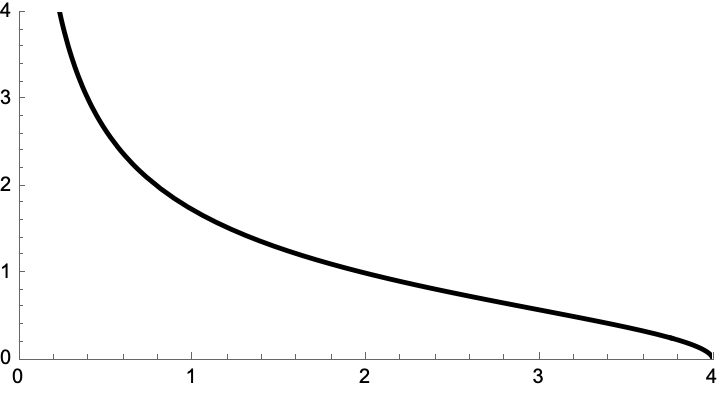}
        \caption{The equilibrium measure of Laguerre Unitary Ensemble.}
      \end{figure}
      
      Clearly, Condition~\ref{ass:slowly varying recurrence coe} is satisfied for all recurrence coefficients with indices in the following set $I$, 
      \begin{equation}
        I=\left\{j\in \mathbb{N}: \frac{j}{n}\to 1\right\}.
      \end{equation}

      At the left edge, $b_{n-1,n}-2\sqrt{a_{n,n}a_{n-1,n}}=\frac{\gamma^2+1}{4 n^2}+O\left(n^{-3}\right)$, by Taylor approximation. Take $x_0=0$. We compute that for all $j\in I_{n,m}^{(\beta)}$
      \begin{equation}
        a_{j,n}a_{j-2,n}-a_{j-1,n}^2 = -\frac{1}{n^2}+O(n^{-4}),
      \end{equation}
      \begin{equation}
        (b_{j-1,n}-x_0-a_{j,n})a_{j-2,n}-(b_{j-2,n}-x_0-a_{j-1,n})a_{j-1,n}=-\frac{\gamma^2}{2 n^3} +O(n^{\beta-4}).
      \end{equation}
      By Theorem~\ref{thm: main 2} we conclude that the mesoscopic CLT holds for all $0<\frac{\alpha}{2}<\beta<\frac{\alpha+1}{3}<1$. Taking $\beta\to\frac{\alpha}{2}$. We conclude that the mesoscopic CLT holds for all $0<\alpha<2$.
      
      At the right edge, $b_{n-1,n}+2\sqrt{a_{n,n}a_{n-1,n}}=4+(2\gamma-2)n^{-1}+O(n^{-2})$, by Taylor approximation. Take $x_0=4$, apply Theorem~\ref{thm: main 1} and we conclude that the asymptotics of the edge fluctuations holds for all $\alpha\in (0,\frac{2}{3})$.
      
      
      \subsection{Gaussian Unitary Ensemble}
      The orthogonality measure of scaled Hermite polynomials is given by
      $$d\mu(x) = e^{-nx^2/2}dx, \quad x\in\mathbb{R}. $$
      The recurrence coefficients read
      \begin{equation}
        a_{j,n} = \sqrt{\frac{j}{n}} \quad,\quad b_{j,n}=0.
      \end{equation}
      \begin{figure}[t]
        \centering
        \includegraphics[scale=0.5]{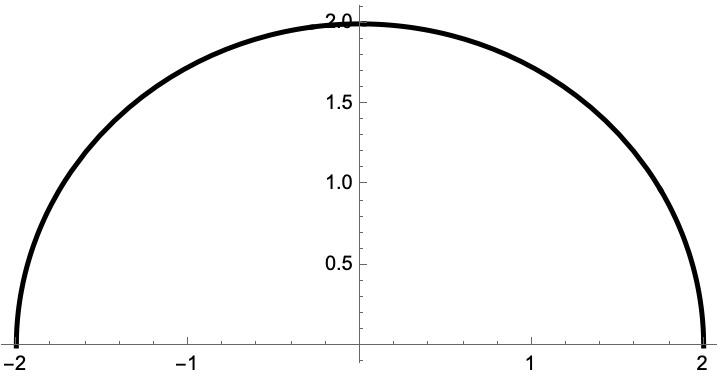}
        \caption{The equilibrium measure of Gaussian Unitary Ensemble.}
      \end{figure}
      Clearly, Condition~\ref{ass:slowly varying recurrence coe} is satisfied for all recurrence coefficients with indices in the following set $I$, 
      \begin{equation}
        I=\left\{j\in \mathbb{N}: \frac{j}{n}\to 1\right\}.
      \end{equation}
      
      Note that $b_{n-1,n}\pm2\sqrt{a_{n,n}a_{n-1,n}}=\pm 2 + O(n^{-1})$. Take $x_0=-2$ or $2$, apply Theorem~\ref{thm: main 1} and we conclude that the asymptotics of the edge fluctuations holds for all $\alpha\in (0,\frac{2}{3})$.
      
      \subsection{Jacobi Unitary Ensemble}\label{sec:eg JUE}
      The orthogonality measure of Jacobi polynomials is given by
      $$d\mu(x) = (2-x)^{\gamma_1}(x+2)^{\gamma_2}dx, \quad x\in[-2,2],$$
      for some parameter $\gamma_1, \gamma_2>-1$. The recurrence coefficients read 
      \begin{equation}
        a_{j} = \sqrt{\frac{16j(j+\gamma_1+\gamma_2)(j+\gamma_1)(j+\gamma_2)}{(2j+\gamma_1+\gamma_2-1)(2j+\gamma_1+\gamma_2)^2(2j+\gamma_1+\gamma_2+1)}},\quad
         b_{j}=\frac{2(\gamma_2^2-\gamma_1^2)}{(2j+\gamma_1+\gamma_2)(2j+\gamma_1+\gamma_2+2)}.
      \end{equation}
      Note that as $j\to\infty$, 
      $$a_j=1+\frac{1-2\gamma_1^2-2\gamma_2^2}{8j^2}+O(j^{-3}), \quad b_j=\frac{\gamma_2^2-\gamma_1^2}{2j^2}+O(j^{-3}).$$
      The equilibrium measure is given the following figure. 
      \begin{figure}[t]
        \centering
        \includegraphics[scale=0.5]{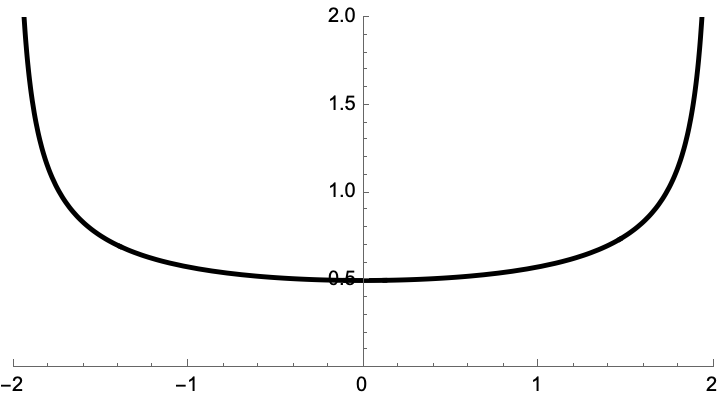}
        \caption{The equilibrium measure of Jacobi Unitary Ensemble.}
      \end{figure}
    
      The edges of fluctuations are at $x_0=2$ or $-2$. Hence, by Theorem~\ref{thm: main 2} or~\ref{thm: main 3} we have the mesoscopic fluctuations at the edges hold for all $\alpha\in(0,2)$. Note that for the Chebyshev polynomial of the second kind we have $\gamma_1=\gamma_2=-1/2$, $a_j=1$ and $b_j=0$.

      \subsection{Freud Weight}\label{sec:eg Freud}
        The orthogonality measure for Freud polynomials is given by Freud weight $$d\mu_n(x)=e^{-n|x|^\gamma}dx, \quad \gamma>0, \quad x\in\mathbb{R}.$$ 
        The recurrence is shown in \cite{lubinsky1988proof,kriecherbauer1999strong} that
      \begin{equation}
        a_{j,n}=\frac{1}{2}\left( \frac{\Gamma(\frac{\gamma}{2})\Gamma(\frac{1}{2})}{\Gamma(\frac{\gamma+1}{2})} \right)^{\frac{1}{\gamma}}\left( \frac{j}{n} \right)^{\frac{1}{\gamma}}(1+res_{\gamma}(j)) \quad\text{and } b_{j,n}=0,
      \end{equation}
      where $\Gamma$ is the Gamma function and 
      \begin{equation}
        res_{\gamma}(j)=\begin{cases}
          O(j^{-2}), &\quad \gamma\geq 2\text{ or } 0<\gamma\leq \frac{1}{2},\\
          O(j^{-\gamma}), &\quad 1<\gamma\leq 2,\\
          O(j^{-1}(\log j)^{-2}), &\quad \gamma=1,\\
          O(j^{-1/\gamma}), &\quad \gamma\geq 2\text{ or } \frac{1}{2}<\gamma <1
        \end{cases}, \quad \text{as } j\to\infty.
      \end{equation}
      Note that its Jacobi operator is essentially self-adjoint if and  only if $\gamma\geq 1$. In such case, the moment problem is determinate. Note that for $0<\gamma<1$, the set of Freud orthonormal polynomials is not dense in $L^2(\mu_n)$. The moment problem is indeterminate. 
      
      Clearly, Condition~\ref{ass:slowly varying recurrence coe} is satisfied for all recurrence coefficients with indices in the following set $I$, 
      \begin{equation}
        I=\left\{j\in \mathbb{N}: \frac{j}{n}\to 1\right\}.
      \end{equation}
      Note that $b_{n-1,n}\pm2\sqrt{a_{n,n}a_{n-1,n}}=\pm \left( \frac{\Gamma(\frac{\gamma}{2})\Gamma(\frac{1}{2})}{\Gamma(\frac{\gamma+1}{2})} \right)^{\frac{1}{\gamma}} + O(n^{-1})$. For both edges  $x_0=- \left( \frac{\Gamma(\frac{\gamma}{2})\Gamma(\frac{1}{2})}{\Gamma(\frac{\gamma+1}{2})} \right)^{\frac{1}{\gamma}} $ or $ \left( \frac{\Gamma(\frac{\gamma}{2})\Gamma(\frac{1}{2})}{\Gamma(\frac{\gamma+1}{2})} \right)^{\frac{1}{\gamma}} $, apply Theorem~\ref{thm: main 1} and we conclude that the asymptotics of the edge fluctuations holds for all $\alpha\in (0,\frac{2}{3})$.
      
      In particular, for $\gamma=2$, the Freud weight is reduced to be the Hermite case.

      \subsection{Tricomi-Carlitz Polynomial Ensemble}\label{sec:eg TC}
      Next, we consider the Tricomi-Carlitz polynomials. Their zero distributions were first studied by \cite{goh1994asymptotics, goh1997zero}. We mention this example to emphasise that the edge of fluctuations may not coincide with the edge of the equilibrium measure. Take $\gamma>1$. The orthogonal measure $\nu^{(\gamma)}$ is a step function with jumps at the points 
      \begin{align*}
        \nu^{(\gamma)}(x) = \frac{(k+\gamma)^{k-1}e^{-k}}{k!} \qquad \text{ at } x=\pm (k+\gamma)^{-\frac{1}{2}}, \qquad k=0,1,\dots.
      \end{align*}
      It gives a discrete polynomial $f_n^{(\gamma)}$. To obtain a reasonable limit we need to rescale $x$ by $n^{-1/2}$, see examples~$4.7$ in \cite{kuijlaars1999asymptotic}. That is we have a scaled measure 
      
      \begin{align*}
        \mu^{(\gamma)}(x)\coloneqq\nu^{(\gamma)}(\sqrt{n}x) = \frac{(k+\gamma)^{k-1}e^{-k}}{k!} \qquad \text{ at } x=\pm (k+\gamma)^{-\frac{1}{2}}\sqrt{n}, \qquad k=0,1,\dots.
      \end{align*}
      The scaled polynomial is given by $f_n^{(\alpha)}(n^{-1/2}x)$.
      However, its recurrence coefficients admit a simpler form 
      \begin{align}
        a_{j,n} = \sqrt{\frac{jn}{(j+\gamma-1)(j+\gamma)}} \quad,\quad b_{j,n}=0
      \end{align}
      \begin{figure}[t]
        \centering
        \includegraphics[scale=0.5]{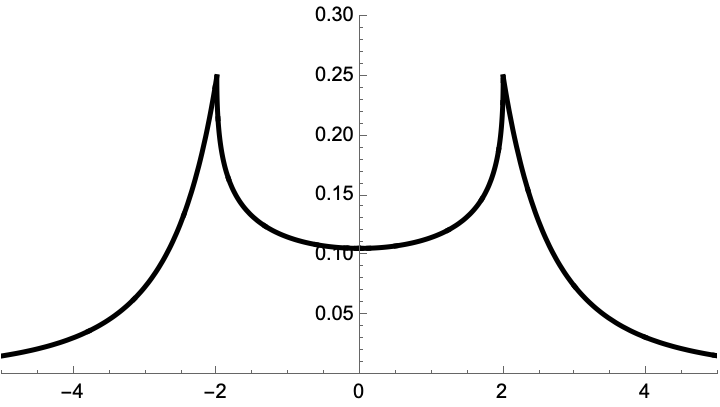}
        \caption{The equilibrium measure of Tricomi-Carlitz Polynomial Ensemble.}
      \end{figure}
      Clearly, Condition~\ref{ass:slowly varying recurrence coe} is satisfied for all recurrence coefficients with indices in the following set $I$, 
      \begin{equation}
        I=\left\{j\in \mathbb{N}: \frac{j}{n}\to 1\right\}.
      \end{equation}
      
      For both edges $b_{n-1,n}\pm2\sqrt{a_{n,n}a_{n-1,n}}=\pm 2 + O(n^{-1})$, by Taylor approximation. Take $x_0=-2$ or $2$, apply Theorem~\ref{thm: main 1} and we conclude that the asymptotics of the edge fluctuations holds for all $\alpha\in (0,\frac{2}{3})$. 

      Note that in this example the equilibrium measure is supported on $\mathbb{R}$. As is shown by \cite{duits2018global} (cf. \eqref{eq:global result}), $\pm 2$ are the global fluctuation edges. Points that are not in $[-2,2]$ are said to be saturated. The resolvent argument used in this paper can also be adapted to show
that mesoscopic linear statistics supported strictly inside these
saturated regions have asymptotically vanishing variance.
      
      \subsection{Weight with Logarithm-Singularity}\label{sec:example log}
      
      Deift and Piorkowski in the work \cite{deift2024recurrence} consider orthogonal polynomials with orthogonality measure by 
      \begin{equation*}
        d\mu(x) = \log\left( \frac{2}{1-x} \right), \qquad x\in[-1,1).
      \end{equation*} 
      They show the asymptotics of its recurrence coefficients to be 
      \begin{multline*}
        a_{j} = \frac{1}{2}-\frac{1}{16j^2} - \frac{3}{32j^2\log^2(j)} +O\left( \frac{1}{j^2\log^3(j)} \right), \quad
         b_{j} = \frac{1}{4j^2} - \frac{3}{16j^2\log^2(j)}+O\left( \frac{1}{j^2\log^3(j)} \right). 
      \end{multline*}
      Hence for both edges,  $b_{j-1}\pm2\sqrt{a_{j}a_{j-1}}=\pm 1 + O(j^{-2})$. Take $x_0= -1$ or $1$, apply Theorem~\ref{thm: main 3} and we conclude that the asymptotics of the edge fluctuations holds for all $\alpha\in (0,2)$.
      
      \subsection{Krawtchouk Polynomial Ensemble}\label{sec:eg Krawtchouk}
      Given $K\in \mathbb{N}$ and $p\in(0,1) $. The Krawtchouk polynomial $k_j(x;p, K)$ is a discrete polynomial on $\{0,1,\dots, K\}$. The orthogonality weight $\nu$ is
      \begin{equation*}
        \nu(x) = \binom{K}{x}p^x(1-p)^{K-x}, \qquad x=0,1,\dots, K.
      \end{equation*}
      The recurrence coefficients of the Krawtchouk polynomials are
      \begin{equation*}
        a_j=\sqrt{(K-j+1)jp(1-p)}, \quad b_j=(K-j)p+j(1-p).
      \end{equation*}
     The OPE with the Krawtchouk weight describes uniformly
     distributed domino tilings of the Aztec diamond, see \cite{johansson2002non}. Here $K$ is related to the size of the diamond and it is particularly interesting to let $K$ go to infinity. Following \cite{johansson2002non}, we consider the case where $\frac{K}{n}\to t$ for $t\geq 2$ as $n\to\infty$. Of interest are the scaled polynomials $k_j(nx;p,K)$. Then the scaled orthogonal weight $\nu$ is
      \begin{equation*}
        \mu(x) = \binom{K}{nx}p^{nx}(1-p)^{K-nx}, \qquad x=0,\frac{1}{n},\frac{2}{n}, \dots, \frac{K}{n}.
      \end{equation*}
      The scaled Krawtchouk polynomials have the recurrence coefficients, 
      \begin{equation*}
        a_{j,n}=\sqrt{\frac{(K-j+1)jp(1-p)}{n^2}}, \quad b_j=\frac{(K-j)p+j(1-p)}{n}.
      \end{equation*}
      Clearly, Condition~\ref{ass:slowly varying recurrence coe} is satisfied for all recurrence coefficients with indices in the following set $I$, 
      \begin{equation}
        I=\left\{j\in \mathbb{N}: \frac{j}{n}\to 1\right\}.
      \end{equation}
      
      For both edges $b_{n-1,n}\pm2\sqrt{a_{n,n}a_{n-1,n}}=(t-2)p+1\pm2\sqrt{(t-1)p(1-p)} + O(n^{-1})$, by Taylor approximation. Take $x_0=(t-2)p+1\pm(t-1)\sqrt{p(1-p)}$, apply Theorem~\ref{thm: main 1} and we conclude that the asymptotics of the edge fluctuations holds for all $\alpha\in (0,\frac{2}{3})$.
      
      \subsection{Hahn Polynomial Ensemble}\label{sec:eg Hahn}
      Given $a,b>-1$ and $N\in\mathbb{N}$. The Hahn orthonormal weight is defined as 
      \begin{equation}
        \nu_N^{(a,b)}(x)=\binom{a+x}{x}\binom{b+N-x}{N-x}, \quad, x=0,1,\dots, N.
      \end{equation}
      
      The Hahn ensemble appears in the lozenge tilings of a hexagon with uniform weights. The parameters $a,b,N$ are related to the size of the hexagon. To study the large hexagon, one requires the sizes grow linearly in $n$ as $n\to\infty$. Let 
      \begin{equation*}
        \frac{a}{n}\to t_1, \quad, \frac{b}{n}\to t_2, \quad \frac{N}{n}\to t_3, \quad \text{ for some } t_1,t_2>0., t_3\geq1.
      \end{equation*}
      The scaled Hahn weight is 
      \begin{equation}
        \mu_N^{(a,b)}(x)=\binom{a+nx}{nx}\binom{b+N-nx}{N-nx}, \quad, x=0,\frac{1}{n},\dots, \frac{N}{n}.
      \end{equation}
      Then the scaled Hahn polynomials have the recurrence coefficients, 
      \begin{equation*}
        a_{j,n}=\frac{j(j+a+b+N+1)(j+b)}{N(2j+a+b)(2j+a+b+1)}\sqrt{\frac{(N-j)(j+a+b)(a+j)(2j+a+b+1)}{j(j+a+b+N+1)(b+j)(2j+a+b-1)}},
      \end{equation*}
      \begin{equation*}
        b_{j,n}= \frac{(N-j)(j+a+b+1)(j+a+1)}{N(2j+a+b+N+1)(2j+a+b+2)}.
      \end{equation*}
      Clearly, Condition~\ref{ass:slowly varying recurrence coe} is satisfied for all recurrence coefficients with indices in the following set $I$, 
      \begin{equation}
        I=\left\{j\in \mathbb{N}: \frac{j}{n}\to 1\right\}.
      \end{equation}
      
      For both edges, clearly, the limits $\lim_na_{n,n}\eqqcolon a_0$ and $\lim_nb_{n-1,n}\eqqcolon b_0$ exist, with the rate of convergence at most $O(n^{-1})$. Take $x_0=b_0\pm 2a_0$, apply Theorem~\ref{thm: main 1} and we conclude that the asymptotics of the edge fluctuations holds for all $\alpha\in (0,\frac{2}{3})$.

      \section*{Acknowledgements}
       The author would like to express his deepest gratitude to Maurice Duits, for proposing this research question and for his invaluable guidance and supports throughout this work. The author is also very grateful to Grzegorz \'Swiderski for bringing the Freud polynomials and the moment problem to the author's attention; to Mateusz Pi\'{o}rkowski for bringing the logarithm singularity example which motivates the development of Theorem~\ref{thm: main 3}. 

\bibliographystyle{abbrv} 
\bibliography{mesoedge_abbrv.bib}       

\end{document}